%% file: Finfinity.tex
\documentclass[11pt]{amsart}
\usepackage{amssymb}
\usepackage{mathrsfs}

\input{macros}

\usepackage{mathtools}
\usepackage{tikz-cd}
\usepackage{indentfirst}
\usepackage{microtype}
\usepackage[colorlinks=true,linkcolor=blue,citecolor=magenta]{hyperref}
\usepackage{tikz}
\usetikzlibrary{topaths}
\usetikzlibrary{calc}
\usetikzlibrary{trees}
\usetikzlibrary{automata,positioning}

\makeatletter
\def\@tocline#1#2#3#4#5#6#7{\relax
  \ifnum #1>\c@tocdepth 
  \else
    \par \addpenalty\@secpenalty\addvspace{#2}%
    \begingroup \hyphenpenalty\@M
    \@ifempty{#4}{%
      \@tempdima\csname r@tocindent\number#1\endcsname\relax
    }{%
      \@tempdima#4\relax
    }%
    \parindent\z@ \leftskip#3\relax \advance\leftskip\@tempdima\relax
    \rightskip\@pnumwidth plus4em \parfillskip-\@pnumwidth
    #5\leavevmode\hskip-\@tempdima
      \ifcase #1
       \or\or \hskip 1em \or \hskip 2em \else \hskip 3em \fi%
      #6\nobreak\relax
    \hfill\hbox to\@pnumwidth{\@tocpagenum{#7}}\par
    \nobreak
    \endgroup
  \fi}
\makeatother

\keywords{finiteness properties, amenable, finitely presented, free group, piecewise, projective, 
Thompson's group, torsion free}

\subjclass[2010]{Primary: 43A07; Secondary: 20F05}

\dedicatory{This paper is dedicated to my mentor, teacher and friend, Prof. Dr. Thomas Zaslavsky.} 

\thanks{
This research was supported in part by
NSF grant DMS--1262019 at Cornell University, an EPFL-Marie Curie postdoctoral fellowship and a Swiss national science foundation grant ``Ambizione" project number PZ00P2 174137.
The author would like to thank the free spirit hostel in Saas-Balen and Pension Heino in Saas-Grund for their hospitality where much of the revision was completed.}

\address{Institute of Mathematics, EPFL\\ SB MATHGEOM EGG\\ Station 8, MA B3 514\\ Lausanne, CH-1015\\ Switzerland}

\email{{\tt yl763@cornell.edu}}
 
\begin{document}

\title[A nonamenable type $F_{\infty}$ group]{A nonamenable type $\textup{F}_{\infty}$ group of piecewise projective homeomorphisms.}

\author{Yash Lodha}

\date \today

\begin{abstract}
We prove that the group of homeomorphisms of the circle
introduced by the author with Justin Moore (Groups, Geometry and Dynamics 2015) is of type $F_{\infty}$.
This makes the group the first example of a type $F_{\infty}$ group which is nonamenable and
does not contain nonabelian free subgroups.
To prove our result we provide a certain generalisation of cube complexes, which we refer to as \emph{cluster complexes}.
We also obtain a computable \emph{normal form}, or a canonical unique 
choice of a word for each element of the group. 
\end{abstract}

\maketitle

\tableofcontents

\section{Introduction}

The von Neumann--Day problem asks whether every nonamenable group contains nonabelian
free subgroups (see \cite{VonNeumann} and \cite{Day}).
In 1980, Ol'shanskii solved the problem by 
constructing counterexamples, which are the so called \emph{torsion free Tarski Monsters} \cite{Olsh}.
Soon after, Adyan showed that certain \emph{Burnside groups} are also counterexamples 
\cite{Adyan1,Adyan2}.
Ol'shanskii and Sapir constructed the first finitely presented counterexamples in 2003 
\cite{OlshSap}.
The examples of Ol'shanskii and Sapir emerge from difficult inductive constructions,
and the number of relations in their prescribed presentations are estimated to be $>10^{200}$ (See the discussion on page $287$ in \cite{Sapir}).

In \cite{Monod} Monod discovered a family of counterexamples that are remarkable owing to the fact that they admit elegant descriptions
as groups of homeomorphisms of the real line.
In particular he showed that the group of piecewise 
$\textup{PSL}_2(\mathbf{R})$-projective homeomorphisms of $\mathbb{RP}^{1}$
that fix infinity is nonamenable and does not contain nonabelian free subgroups.
However, Monod's examples are not finitely presentable.

In \cite{LodhaMoore} we constructed a finitely presentable nonamenable subgroup of Monod's
group with $3$ 
generators and $9$ relations. 
This group denoted by $G$ is generated by $a(t) = t+1$ together with the following
two homeomorphisms of $\mathbf{R}$:
\[
b(t)=
\begin{cases}
 t&\text{ if }t\leq 0\\
 \frac{t}{1-t}&\text{ if }0\leq t\leq \frac{1}{2}\\
 3-\frac{1}{t}&\text{ if }\frac{1}{2}\leq t\leq 1\\
 t+1&\text{ if }1\leq t\\
\end{cases}
\qquad
c(t)=
\begin{cases}
 \frac{2t}{1+t}&\text{ if }0\leq t\leq 1\\
t&\text{ otherwise}\\
\end{cases}
\]

The finiteness properties \emph{type} $\mathbf{F}_{n}$ are natural topological generalisations of the properties of finite generation and finite presentability.
A group is said to be of type $\mathbf{F}_{\infty}$ if it is of type $\mathbf{F}_n$ for all $n$, or equivalently if it is the fundamental group of a connected, aspherical CW complex
with finitely many cells in every dimension.
The study of counterexamples to the von Neumann--Day problem has an interesting historical connection with finiteness properties of groups.
In the $1970$s Thompson discovered a remarkable finitely presented group, which is now known as Thompson's group $F$.
The elements $a,b$ from above generate a copy of $F$ acting on the real line by piecewise projective homeomorphisms.
In $1979$ Geoghegan made the following conjectures about $F$ (See \cite{GeogConj}.)
\begin{enumerate}
\item $F$ is of type $\mathbf{F}_{\infty}$.
\item $F$ does not contain nonabelian free subgroups.
\item $F$ is nonamenable.
\item All homotopy groups of $F$ at infinity are trivial.
\end{enumerate}

Conjectures $(1)$ and $(4)$ were proved by Brown and Geoghegan \cite{BrownGeoghegan}. 
Conjecture $(2)$ was proved by Brin and Squier \cite{BrinSq}.
The status of $(3)$ still remains open.
There is considerable interest in Conjecture $(3)$ especially because if $F$ is nonamenable
it would be an elegant counterexample to the von Neumann--Day problem.

The following question is natural,
especially in light of Geoghegan's conjectures about Thompson's group $F$.

\begin{question}
Does there exist a group which is of type $F_{\infty}$, nonamenable and does not contain nonabelian free subgroups?
\end{question}

In this article we establish that our group $G$ is of type $F_{\infty}$, and thus answer the above question in the affirmative.
More particularly, we prove the following.

\begin{thm}\label{main}
 The group $G$ acts on a connected cell complex $X$
by cell permuting homeomorphisms such that the following holds.
\begin{enumerate}
 \item $X$ is contractible.
\item The quotient $X/G$ has finitely many cells in each dimension.
\item The stabilizer of each cell is of type $\mathbf{F}_{\infty}$.
\end{enumerate}
It follows that the group $G$ is of type $\mathbf{F}_{\infty}$.
\end{thm}

Our complex $X$ is a special type of CW complex built out of CW subdivisions of Euclidean cubes.
Although there is no natural metric on such a complex, 
notions of metric and curvature shall play an essential role in proving that the complex is contractible.
Finally, we remark that it is observed in \cite{Zaremsky} that all homotopy groups of $G$ at infinity are trivial.
So $G$ satisfies all four statements of Geoghegan's conjecture for Thompson's group $F$, and is the first such example.

\section{Proof Strategy}

In this section we provide an outline of the proof.
Our complex $X$
is a special type of CW complex,
which we call a \emph{complex of clusters}.
A \emph{cluster} is a certain CW subdivision of a cube.
An $n$-cluster is a CW subdivision of an $n$ dimensional cube $[0,1]^n$,
which is obtained by considering the intersection pattern with $[0,1]^n\subset \mathbf{R}^n$ of a family of hyperplanes of $\mathbf{R}^n$.
We give a precise definition of clusters in Section $5$.

For instance, a $2$-cluster can be a square with a single $2$-cell that fills it,
or a square with a diagonal $1$-cell, and two triangular $2$-cells that fill it. 
The latter cluster is obtained from the intersection pattern of $[0,1]^2$ 
with the hyperplane $x=y$ in $\mathbf{R}^2$.
A subcluster of the latter $2$-cluster is either a $0$-cell, or
one of the five $1$-cells (including the diagonal $1$-cell). 

To describe a \emph{complex of clusters},
we draw an analogy with cube complexes.
A cube complex is a CW complex obtained by gluing
cubes along their faces,
i.e. gluing along facial ``sub-cubes".
A complex of clusters is a CW complex obtained by gluing
clusters along subclusters.
The key difference is that in a cluster complex, subclusters could be placed diagonally
in a cluster.
For instance, we can take two $2$-clusters described in the previous paragraph,
and glue along the diagonal $1$-cells.
This way we obtain a complex with four triangular $2$-cells,
or four triangles all glued along one edge.

To give the reader a more concrete idea of what our complex $X$ looks like, 
we give a brief description of $X$.
A detailed description will appear in Sections $6$ and $7$.
The $0$-skeleton $X^{(0)}$ is defined as 
the set of right cosets of $F$ in $G$,
on which $G$ acts on the right.
Two cosets $F\tau_1, F\tau_2$
are connected by an edge in $X^{(1)}$
if the double coset $F\tau_1\tau_2^{-1} F$
is one of four distinguished double cosets. 
It is then immediate from the definition that $G$ acts on $X^{(1)}$
in a way that preserves the edge relation,
since $F\tau_1 g ( \tau_2g)^{-1} F= F\tau_1\tau_2^{-1} F$.

In order to describe the cluster complex structure on $X$,
we identify an infinite family $\mathcal{H}$ of finite, connected subgraphs of $X^{(1)}$,
whose elements are in a natural way $1$-skeletons of clusters.
We demonstrate that $\mathcal{H}$ is closed under finite intersections,
i.e. the intersection of a finite number of graphs in $\mathcal{H}$ is also a graph in $\mathcal{H}$
if it is nonempty.

Then we ``fill" the $1$-skeleton of each graph in $\mathcal{H}$
by adding higher cells to obtain a cluster.
We demonstrate that these fillings are ``compatible",
i.e. given two graphs $\Gamma_1,\Gamma_2\in \mathcal{H}$ and their respective fillings, 
the filling of the intersection $\Gamma_1\cap \Gamma_2$ agrees with the induced
fillings obtained from the restrictions of $\Gamma_1,\Gamma_2$ respectively.
Moreover, we demonstrate that this filling is $G$-equivariant.
The union of $X^{(1)}$ together with the fillings of graphs in $\mathcal{H}$ will define our complex of clusters $X$.

The structure of the paper is as follows.
In Section $4$ we construct a normal form for elements of $G$.
This normal form will play a small but essential role in the rest of the paper,
but is of independent interest since it can be a useful tool to study the group.
In Section $5$ we define the notion of a cluster, and a cluster complex.

In Section $6$ we define $X^{(1)}$ and study the action of $G$ on $X^{(1)}$.
In Section $6$, we shall also prove the \emph{expansion lemmas},
which shall be needed later in the paper, in particular during the proof of asphericity of $X$.
However, we prove these Lemmas here since their proofs shall emerge naturally from the ideas in Section $6$. 
In Section $7$, we define $\mathcal{H}$ and the higher cells of $X$.

In Section $8$ we prove that $X/G$ has finitely many cells in every dimension,
and that $Stab_G(e)$ for each cell $e$ is a group of type $\mathbf{F}_{\infty}$.
In fact, we demonstrate that the stabilizer of a cell is a product of Higman-Thompson groups.
In Section $9$ we prove that $X$ is simply connected,
and in Section $10$ we prove that $X$ is aspherical.
Then we conclude that $X$ is contractible.
At the end of the article we state a question, suggested to the author by Gromov, that could provide an interesting direction for further research.

A complex of clusters in general does not admit a natural piecewise Euclidean metric.
For instance, one cannot simply declare all the $1$-cells to be isometric to $[0,1]$,
since each $1$-cell in our complex will occur as a diagonal $1$-cell of clusters of any given dimension.
Moreover, each $1$-cell sits ``diagonally" inside some cluster and ``facially" inside others. 
Hence one cannot simply use the natural metric structure on the CW-subdivision of $[0,1]^n$ that describes a cluster,  as a model for a piecewise Euclidean metric.
Even if one perturbs this to a metric, it does not appear to be useful in proving that the complex is contractible.
However, our proof that the complex is aspherical shall involve notions of metric and curvature.
More particularly, we demonstrate that every finite subcomplex $Y$ of $X$ is contained in a subcomplex
$\bf{Y}$ of $X$ such that $\bf{Y}$ is homeomorphic to a nonpositively curved cube complex, and hence aspherical.

\section{Preliminaries}

In this section we will review some terminology needed later in the paper. 
Readers may wish to skim or skip the material and refer back to it only as necessary.
In this article, all actions shall be right actions.
However, when we make use of function notation we shall differ from this convention, 
i.e. for instance we shall write $x\cdot fg= g(f(x))$.

\subsection{Finiteness properties of groups.}
The classical finiteness properties of groups are that of 
being finitely generated and finitely presented. 
These notions were generalized by C.T.C. Wall \cite{Wall}. 
In this paper we are concerned with the properties \emph{type} $\mathbf{F}_n$. 
These properties are quasi-isometry invariants of groups \cite{Alonso}. 
In order to discuss these properties first we need to define Eilenberg-Maclane complexes.

An Eilenberg-Maclane complex for a group $G$, or a $K(G,1)$, 
is a connected CW-complex $X$ such that 
$\pi_1(X)=G$ and $\widetilde{X}$ is contractible.
It is a fact that for any group $G$, 
there is an Eilenberg-Maclane complex $X$ which is unique up to homotopy type.
A group is said to be \emph{of type} $\mathbf{F}_n$ if it admits an 
Eilenberg-Maclane complex with a finite $n$-skeleton.  
Clearly, a group is finitely generated if and only if it is of type $\mathbf{F}_1$, 
and finitely presented if and only if it is of type $\mathbf{F}_2$.
(For more details see \cite{Geoghegan}.)
A group is said to be of type $\mathbf{F}_{\infty}$ if it is of type $\mathbf{F}_n$ for all $n\in \mathbf{N}$.
Equivalently, a group is of type $\mathbf{F}_{\infty}$ if it admits a $K(G,1)$ with finitely many cells in each dimension.

The following is a special case of a well known result 
(See Proposition $1.1$ in \cite{BrownFiniteness}.
Note that the property discussed there is type $\mathbf{FP}_{\infty}$,
but since for the class of finitely presentable groups type $\mathbf{FP}_{\infty}$
and type $\mathbf{F}_{\infty}$ are the same, we state this proposition in terms of type $\mathbf{F}_{\infty}$).

\begin{prop}\label{criterion}
 Let $\Gamma$ be a finitely presentable group that acts on a cell complex $X$
by cell permuting homeomorphisms such that the following holds:
\begin{enumerate}
 \item $X$ is contractible.
\item The quotient $X/ \Gamma$ has finitely many cells in each dimension.
\item The stabilizers of each cell are of type $\mathbf{F}_{\infty}$.
\end{enumerate}
 Then $\Gamma$ is of type $\mathbf{F}_{\infty}$.
\end{prop}

\subsection{Nonpositively curved cube complexes}
Although the complex we construct in this article is not a cube complex, 
cube complexes will play an important role in the proof of contractibility of our complex.

By a $\emph{regular }n\emph{-cube}$ $\square^n$ 
we mean a cube which is isometric 
to the cube $[0,1]^n$ in $\mathbf{R}^n$. 
A cube complex is a cell complex of 
regular Euclidean cubes glued along their faces by homeomorphisms.
The metric on such cube complexes is the piecewise Euclidean metric.
(see \cite{BridsonH} for details). 

 Given a vertex $v$ of a regular cube $\square^n$, the link $Lk(v,\square^n)$ is the set of unit tangent vectors at 
 $v$ that point in $\square^n$. 
 This is a subset of the unit sphere $S^{n-1}$ which is 
 homeomorphic to a simplex of dimension $n-1$. 
 

 A simplicial complex $Z$ is called a ``flag'' complex 
 if any set $v_1,...,v_n$ of vertices of $Z$ that are pairwise 
 connected by an edge span a simplex. 
 This is also known as the ``no empty triangles'' condition.
Now we are ready to state the main definition.

\begin{defn}
 A cube complex $X$ is said to be nonpositively curved 
 if the link of each vertex is a flag complex. 
\end{defn}

For basic results concerning these complexes, we refer the reader to \cite{BridsonH}.
We shall only use the fact that nonpositively curved cube complexes are aspherical,
which is a corollary of the following theorem of Gromov.

\begin{thm}\label{npc}
 (Gromov) A cube complex $X$ is $\textup{CAT}(0)$ 
 if and only if it is nonpositively curved and simply connected. 
\end{thm}

This means that the universal cover of a nonpositively curved cube complex is $CAT(0)$, and hence contractible.
So we have the following.

\begin{cor}
Nonpositively curved cube complexes are aspherical.
\end{cor}

\subsection{Coset graphs}\label{vertrangraph}

We shall use the notion of a \emph{coset graph}, which is a certain generalisation of the more familiar notion of a \emph{Cayley graph}.
We remark that this notion may not be 
familiar to a group theorist,
although it is more likely familiar to a graph theorist.
Let $G$ be a group and $H$ be a subgroup of $G$.
Let $S$ be a finite set of elements of $G$ such that:
\begin{enumerate}
 \item $S\subseteq G\setminus H$ and for each $s\in S$, $s^{-1}\in S$.
\item $S\cup H$ generates $G$.
\end{enumerate} 
Then we can form the so called \emph{coset graph}
$\textup{Cos}(G,H,S)$ as follows.
The vertices of $\textup{Cos}(G,H,S)$ are the right cosets
of $H$ in $G$, and two cosets $Hg_1,Hg_2$ are connected by an edge if 
$$g_1g_2^{-1}\in H(S)H=\bigcup_{s\in S} HsH$$
Here $$HsH=\{h_1sh_2\mid h_1,h_2\in H\}$$ is a double coset.

The group $G$ acts on the graph $\textup{Cos}(G,H,S)$ on the right,
and this action is vertex transitive.
Moreover the action of $G$ on this graph is faithful if and only if the core
$\bigcap_{g\in G}g^{-1}Hg$ is trivial.
By the following result (this is an easy exercise, and is stated as Theorem $3.8$ in \cite{Lauri}), in fact every vertex transitive graph, i.e. a graph whose group of automorphisms is vertex transitive, can be described in this way.

\begin{prop}\label{Sabidussi}
Let $\Omega$ be some vertex transitive graph.
Then $\Omega$ is isomorphic to some coset graph $\textup{Cos}(G,H,S)$.
\end{prop}

The $1$-skeleton of our complex $X$ will emerge as a vertex transitive graph in this way.

\subsection{Hyperplane arrangements and the face complex}\label{hyparr}

Consider $\mathbf{R}^n$ endowed with the usual orthonormal basis,
and with variables $x_1,...,x_n$ representing coordinates in this basis.
An affine hyperplane is an $(n-1)$-dimensional subspace
comprising of solutions of an equation of the form
$$a_1x_1+...+a_nx_n=a_{n+1}$$
where $a_i\in \mathbf{R}$.
The hyperplane itself is denoted as the set 
$$\{a_1x_1+...+a_nx_n=a_{n+1}\}$$
  
A finite \emph{hyperplane arrangement} $\mathcal{A}$ is a finite set of affine hyperplanes in $\mathbf{R}^n$.
A \emph{region} of $\mathcal{A}$ is a connected component of $\mathbf{R}^n\setminus \bigcup_{H\in \mathcal{A}}H$.
Let $\mathcal{R(A)}$ denote the set of regions of $\mathcal{A}$.
It is an elementary exercise to prove that every region is open, convex and thus homeomorphic to the interior of an $n$-dimensional ball.

The set of \emph{flats} of the arrangement are the affine subspaces of $\mathbf{R}^n$ obtained by taking an
intersection of hyperplanes in $\mathcal{A}$.
The trivial intersection is also a flat, which equals $\mathbf{R}^n$.
Given a flat $T$, we denote the hyperplane arrangement $\mathcal{A}\restriction T$
consisting of the set of hyperplanes $$\{ T\cap H\mid H\in \mathcal{A}, T\not \subseteq H\}$$
in the subspace $T$.
We define the regions of $T$ in a similar way as above,
i.e. $\mathcal{R(A)}\restriction T$ equals the set of connected components of $T\setminus  \bigcup_{H\in \mathcal{A}\restriction T} H$.

The union of the sets of regions $$\bigcup_{T\text{ is a flat of }\mathcal{A}} \mathcal{R(A)}\restriction T $$
is called the \emph{face complex} of the arrangement $\mathcal{A}$.
The face complex provides a cellular structure on $\mathbf{R}^n\bigcup \partial\mathbf{R}^n=\mathbb{B}^n$.

\subsection{Binary sequences and the group $F$}
We will take $\mathbf{N}$ to include $0$. 
Let $2^\Nbb$ denote the collection of all infinite binary sequences and
let $2^{<\Nbb}$ denote the set of all finite binary sequences.
For $s\in 2^{<\mathbf{N}}$, $s(i)$ denotes the $i$'th digit of $s$.
If $i \in \Nbb$ and $u$ is a binary sequence of length at least $i$, we will let 
$u\restriction i$ denote the initial part of $u$ of length $i$.
We denote by $|s|$ the length of $s$,
which is the number of digits in $s$.

The infinite rooted binary tree shall play an important role in this article,
and the reader may find it useful to visualise certain definitions within this tree.
We shall associate the standard picture with the tree,
where the branches labelled by $0$'s appear to the left of the branches labelled by $1$'s.
We shall also visualise finite rooted binary trees in the same way.
\begin{center}
\begin{tikzpicture}[level distance=1.5cm,
  level 1/.style={sibling distance=3cm},
  level 2/.style={sibling distance=1.5cm}]
  \centering
  \node {$\emptyset$}
    child {node {0}
      child {node {00}}
      child {node {01}}
    }
    child {node {1}
    child {node {10}}
      child {node {11}}
    };
\end{tikzpicture}
\end{center}

If $s$ and $t$ are finite binary sequences, then we will write $s \subseteq t$
if $s$ is an initial segment of $t$ and $s \subset t$ if $s$ is a proper
initial segment of $t$.
If neither $s \subseteq t$ nor $t \subseteq s$, then we will say that
$s$ and $t$ are \emph{independent}.
A list of finite binary sequences $s_1,...,s_n$ is said to be \emph{independent},
if they are pairwise independent.

The set $2^{<\Nbb}$ is equipped with an order defined
by $s < t$ if either $t \subset s$ or if $s$ and $t$ are independent
and $s(i) < t(i)$ where $i$ is the smallest number such that $s(i) \ne t(i)$.
If $u, s_1,...,s_n\in 2^{<\mathbf{N}}$, we say that
$u$ \emph{dominates} $s_1,...,s_n$ if for each $1\leq i\leq n$,
either $s_i,u$ are independent or $s_i\subset u$.

The finite binary sequences $s,t$ are said to be \emph{consecutive}
if there is a binary sequence $u$ and numbers $n_1,n_2\in \mathbf{N}$
such that $s=u 0 1^{n_1}$ and $t=u 1 0^{n_2}$. 
A list of finite binary sequences $s_1,...,s_n$ is said to be \emph{consecutive}
if each pair $s_i,s_{i+1}$ is consecutive for $1\leq i\leq n-1$.
Note that if $s_1,...,s_n$ are consecutive then they are automatically independent. 
In particular, a list of consecutive binary sequences is an ordered subset of 
the set of leaves of a finite rooted binary tree, listed in order from left to right.
For instance, the list $s_1=01, s_2=100, s_3=101$ is consecutive as viewed in the finite rooted binary tree in the picture below.

\begin{center}
\begin{tikzpicture}[level distance=1.5cm,
  level 1/.style={sibling distance=3cm},
  level 2/.style={sibling distance=1.5cm}]
  \centering
  \node {$\emptyset$}
    child {node {0}
      child {node {00}}
      child {node {01}}
    }
    child {node {1}
    child {node {10} child {node{100}}
    child {node{101}}}
      child {node {11}}
    };
\end{tikzpicture}
\end{center}

If $\xi$ and $\eta$ are infinite binary sequences, then we will say that
$\xi$ and $\eta$ are \emph{tail equivalent} if there are $s,t\in 2^{<\mathbf{N}}$ and $\zeta\in 2^{\mathbf{N}}$ such that
$\xi = s \zeta$ and $\eta = t \zeta$.
We use $0^{\infty}, 1^{\infty}$ to denote the constant infinite sequences
$000....$ and $111...$ respectively.
More generally, given a finite binary sequence $s$, $s^{\infty}$ denotes
the sequence $sss....$.

We denote the collection of all finite rooted binary trees by $\Tcal$.
A tree $T$ in $\Tcal$ will be denoted by a set of finite binary sequences $s_1,...,s_n$
which are the addresses of leaves in $T$.
The indices are ordered so that if $i<j$ then $s_i<s_j$.
 We view elements $T$ of $\Tcal$ as \emph{prefix} sets.
So we view $T$ as a set of finite binary sequences with the property that every infinite 
binary sequence has a unique initial segment in $T$. 

A \emph{tree diagram} is a pair $(L,R)$ of elements of $\Tcal$
with the property that $|L|=|R|$.
A tree diagram describes a map of infinite binary sequences
as follows:
$$
s_i \xi \mapsto t_i \xi
$$
where $s_i$ and $t_i$ are the $i$th elements of $L$ and $R$, respectively, in the
order as defined above and $\xi$ is any binary sequence.
The collection of all such functions from $2^\Nbb$ to $2^\Nbb$ defined
in this way, under the operation of composition, is \emph{Thompson's group $F$}.
The function associated to a tree diagram
is also defined on any finite binary sequence $u$ such that
$u$ has a prefix in $L$.
So the group $F$ admits a partial action on $2^{<\mathbf{N}}$.
Given $f\in F$ and $s\in 2^{<\mathbf{N}}$, we say that $f$ \emph{acts on} $s$ if $s\cdot f$ is defined.
Similarly, given $s_1,...,s_n\in 2^{<\mathbf{N}}$ we say that $f$ \emph{acts on} $s_1,...,s_n$ if $s_1\cdot f,...,s_n\cdot f$
are all defined.

We direct the reader to the standard reference \cite{Belk} for
the definition and properties of Thompson's group $F$; 
additional information can be found in \cite{Cannon}.
We shall mostly follow the notation and conventions
of \cite{LodhaMoore}.

\subsection{The group $G$}

To prove that the group $G$ is of type $\mathbf{F}_{\infty}$,
the action of the group on the real line by piecewise projective homeomorphisms
does not appear to be useful.
In \cite{LodhaMoore} we describe a \emph{combinatorial model} for $G$
by means of a faithful action of $G$ by homeomorphisms of the Cantor set $2^{\mathbf{N}}$.
This model was used to prove that $G$ is finitely presentable.
This combinatorial model will be used throughout this paper,
and we shall not refer to the functions $a,b,c:\mathbf{R}\to \mathbf{R}$
which were defined in the introduction.

We recall the combinatorial description, and describe the generators and relations.
We start with the following two primitive functions:
\[
\xi\cdot x
 =
\begin{cases}
\seq{0}\eta & \textrm{ if } \xi = \seq{00} \eta \\
\seq{10}\eta & \textrm{ if } \xi = \seq{01} \eta \\
\seq{11}\eta & \textrm{ if } \xi = \seq{1} \eta \\
\end{cases}
\qquad
\xi\cdot y =
\begin{cases}
\seq{0}(\eta\cdot y) & \textrm{ if } \xi = \seq{00} \eta \\
\seq{10}(\eta\cdot y^{-1}) & \textrm{ if } \xi = \seq{01} \eta \\
\seq{11}(\eta\cdot y) & \textrm{ if } \xi = \seq{1} \eta \\
\end{cases}
\]

From these functions, we define families of functions $x_s$ $(s \in 2^{<\Nbb})$ and 
$y_s$ $(s \in 2^{<\Nbb})$
which act just as $x$ and $y$, but localised to those binary sequences which extend $s$.
\[
\xi\cdot x_s
 =
\begin{cases}
s (\eta\cdot x) & \textrm{ if } \xi = s \eta \\
\xi & \textrm{otherwise}
\end{cases}
\qquad
\xi\cdot y_s
 =
\begin{cases}
s (\eta\cdot y) & \textrm{ if } \xi = s \eta \\
\xi & \textrm{otherwise}
\end{cases}
\]
If $s$ is the empty-string, it will be omitted as a subscript.
Let $$S=\{x_t,y_s\mid s,t\in 2^{<\Nbb},s\neq \seq{0}^k,s\neq \seq{1}^k, s\neq \emptyset\}$$
Our group $G$ is generated by functions in the set $S$.
In fact, $G$ is generated by $x,x_1,y_{10}$ which correspond respectively to the functions $a,b,c$ that were defined in the introduction.
The infinite generating set $S$ is much more desirable for the purpose of our proofs and will be used throughout this article.
We now list an infinite set of relations $R$ satisfied by the generators in $S$.
Note that for finite binary sequences $t,s$, $t\cdot x_s$ is defined if either $s_1\subseteq t$, for some $s_1\in \{s00,s01,s1\}$,  or if $s,t$ are independent.

\begin{enumerate}

\item If $t\cdot x_s$ is defined, then $x_t x_s = x_s x_{t\cdot x_s}$ 

\item $x_{s}^2=x_{s \seq{0}}x_sx_{s \seq{1}}$.

\item If $t\cdot x_s$ is defined, then
$y_t x_s = x_s y_{t\cdot x_s}$.

\item if $s$ and $t$ are independent, then $y_s y_t = y_t y_s$.

\item $y_s = x_s y_{s \seq{0}} y_{s \seq{10}}^{-1} y_{s \seq{11}}$.
\end{enumerate}

Note that in the above there is no occurrence of the functions
$y^{\pm}_s$ where $s=\seq{1}^k$, $s=\emptyset$ or $s=\seq{0}^k$.
We proved in \cite{LodhaMoore} that:

\begin{thm}
 The group $G\cong\langle S\mid R\rangle$ is finitely presented, 
 nonamenable, and does not contain nonabelian free subgroups.
\end{thm}

The generators 
$\{x_s\mid s\in 2^{<\mathbf{N}}\}$
together with the relations 

\begin{enumerate}

\item If $t\cdot x_s$ is defined, then
$x_t x_s = x_s x_{t\cdot x_s}$.

\item $x_{s}^2=x_{s \seq{0}}x_sx_{s \seq{1}}$

\end{enumerate}
describe an infinite presentation for Thompson's group $F$.

Recall that in \cite{Cannon} 
$$\{x^{-1},x_1^{-1},x_{11}^{-1},x_{111}^{-1},...\}\subseteq X$$
is used as the infinite generating set for $F$.
We now recall the normal form for Thompson's group $F$
in this generating set (See Corollary-Definition 2.7 in \cite{Cannon}).
We rephrase the normal form to suit our conventions,
which differ from those used in \cite{Cannon}.
We shall refer to this normal form as the \emph{$F$-normal form}.

\begin{thm}\label{Fnormalform}
Every non-trivial element of $F$ can be expressed in unique normal form:
$$x^{-b_0}x_1^{-b_1}x_{1^2}^{-b_2}...x_{1^n}^{-b_n}x_{1^n}^{a_n}...x_1^{a_1}x^{a_0}$$
such that:
\begin{enumerate}
\item $a_0,...,a_n$ and $b_0,...,b_n$ are nonnegative numbers.
\item Exactly one of $a_n,b_n$ is nonzero.
\item If for some $0\leq k\leq n-1$ it is true that $a_k,b_k>0$, then $a_{k+1}>0$ or $b_{k+1}>0$.
\end{enumerate}
Furthermore, every such normal form word is nontrivial in $F$.
\end{thm}

Now we recall from \cite{LodhaMoore} the following definitions.

\begin{defn}
An $X$-word is a word in the letters
$$\{x_s^t\mid s\in 2^{<\mathbf{N}}, t\in \mathbb{Z}\}$$
A $Y$-word is a word in the 
letters $$Y=\{y_s^t\mid s\in 2^{<\mathbf{N}},s\neq \seq{0}^k,s\neq \seq{1}^k, s\neq \emptyset, t\in \mathbb{Z}\}$$
An $S$-word is a word in the generators 
$$S=\{y_s^t,x_u^v\mid s,u\in 2^{<\mathbf{N}},s\neq \seq{0}^k,s\neq \seq{1}^k, s\neq \emptyset,\text{ and } t,v\in \mathbb{Z}\}$$
The elements of $Y$
are called \emph{percolating elements}.
\end{defn}

Recall that for $s,t\in 2^{<\Nbb}$ we have that $s < t$ if either $t \subset s$ or if $s$ and $t$ are independent
and $s(i) < t(i)$ where $i$ is the smallest number such that $s(i) \ne t(i)$.
(Here $s(i)$ is the $i$'th digit of $s$.)

\begin{defn}
An $S$-word $fy_{s_1}^{t_1}...y_{s_n}^{t_n}$ is in \emph{standard form} if
it is the concatenation of an $X$-word $f$ followed by a $Y$-word 
$y_{s_1}^{t_1}...y_{s_n}^{t_n}$
with the property that $s_i<s_j$ if $i<j$.
We will write \emph{standard form} to mean an $S$-word in standard form.

It is convenient to denote a standard form as $f\lambda$,
where $f$ is an $X$-word and $\lambda$ is a $Y$-word in standard form.
We will mostly use the convention of denoting $X$-words by letters $f,g,h$,
often with subscripts, and $Y$-words in standard form with either greek letters,
or explicitly as $y_{s_1}^{t_1}...y_{s_n}^{t_n}$.
 
The \emph{depth} of $fy_{s_1}^{t_1}...y_{s_n}^{t_n}$
equals $\textup{min}\{|s_i|\mid 1\leq i\leq n\}$.
We use the convention that standard forms that are $X$-words have infinite depth.
\end{defn}

We also define a slightly weaker notion than a standard form,
which will be useful in phrasing our arguments.

\begin{defn}
An $S$-word $fy_{s_1}^{t_1}...y_{s_n}^{t_n}$ is in \emph{weak standard form} if
it is the concatenation of an $X$-word $f$ followed by a $Y$-word 
$y_{s_1}^{t_1}...y_{s_n}^{t_n}$
with the property that if $s_j\subset s_i$ then $i<j$.

Note that this notion is slightly weaker than the standard form.
For instance $y_{100}y_{101}$ is in standard form
but not $y_{101}y_{100}$, although the latter is in weak standard form.
\end{defn}
The definition implies the following elementary fact.
\begin{lem}
A weak standard form can be converted into a standard form by the application of a finite sequence of commuting relations,
which are relations $(4)$ of $R$.
\end{lem}

\begin{defn}\label{calculationexp}

Associated with a (weak) standard form
$fy_{s_1}^{t_1}...y_{s_n}^{t_n}$ and a sequence $\sigma\in 2^{\omega}$
is the notion of a \emph{calculation}.
This is an infinite string in letters $y,y^{-1},\seq{0},\seq{1}$.
The evaluation of $fy_{s_1}^{t_1}...y_{s_n}^{t_n}$ on $\sigma$ comprises of a 
prefix replacement that is determined by the transformation $f\in F$
followed by an infinite sequence of
applications of the transformations described by the percolating elements.
The latter is encoded as an infinite string with letters $y,y^{-1},0,1$ and denoted as the calculation of
$y_{s_1}^{t_1}...y_{s_n}^{t_n}$ on $\sigma\cdot f$.
The calculation is equipped with the following substitutions:
\[
y\seq{00} \rightarrow \seq{0}y 
\qquad
y\seq{01} \rightarrow \seq{10}y^{-1}
\qquad
y\seq{1} \rightarrow \seq{11}y
\]
\[
y^{-1} \seq{0} \rightarrow \seq{00}y^{-1}
\qquad
y^{-1} \seq{10} \rightarrow \seq{01}y
\qquad
y^{-1} \seq{11} \rightarrow \seq{1}y^{-1}
\]
\end{defn}
For example, consider the action of the word $x^{-1}y_{100}^{-1}y_{10}$
on the binary sequence $11001111...$.
First, we have the prefix replacement map $$(11001111...)\cdot x^{-1}= 1001111...$$
Then the resulting calculation string is $$10y0y^{-1}1111...$$
The \emph{output string} of the evaluation of the word on the binary string
is the limit of the strings obtained from performing these substitutions.
So $1010111...$ is the output string of the calculation $10y0y^{-1}1111...$.

The calculation of a weak standard form $fy_{s_1}^{t_1}...y_{s_n}^{t_n}$
on an infinite binary sequence $\sigma$ is the same as the calculation
of $y_{s_1}^{t_1}...y_{s_n}^{t_n}$ on $\sigma\cdot f$.
The set of all strings that encode calculations is a subset of $\{0,1,y,y^{-1}\}^{\mathbf{N}}$
consisting of elements with the property that there are only finitely many occurrences of $y^{\pm}$.
Note that the calculation of a weak standard form on a binary sequence $\sigma$ does not change if we perform a sequence of commuting moves (relations $(4)$ of $R$) on the weak standard form.

The percolating element $y$ can be expressed as a finite state transducer, which is a certain generalisation of a finite state automaton.
The difference between transducers and automata is that in the case of transducers, not every letter in the alphabet is read.
Our transducer is diagrammatically represented in Figure $1$ above.
\begin{figure}\label{TransducerDiagram}
\begin{center}
\begin{tikzpicture}[shorten >=1pt,node distance=2cm,on grid,auto] 
   \node[state] (y)   {$y$}; 
   \node[state] (yy) [right=of y] {$y^{-1}$}; 
    \path[->] 
    (y) edge  [loop above] node {$00\mid  0$} (y);
    \path[->]  (y)   edge  [loop below] node {$1\mid 11$} (y);
   \path[->] (y) edge  [bend left, above]node  {$01\mid 10$} (yy);
     \path[->]     (yy) edge [bend left, below] node{$10\mid 01$} (y);
   \path[->] (yy) edge [loop above]  node {$0\mid 00$} (yy) ;
        \path[->] (yy) edge  [loop below]  node {$11\mid 1$} (yy);
\end{tikzpicture}
\end{center}
\caption{The transducer $y$} \label{fig}
\end{figure}

In this diagram, each edge is labelled by a pair of finite strings of the form $\sigma\mid \tau$.
Here $\sigma$ represents the input string, and $\tau$ represents the output string.
Note that the set of input strings for each state forms a prefix set.
For $y$, this is $\{00,01,1\}$ and for $y^{-1}$ this is $\{0,10,11\}$.
The same holds for the set of output strings at a given state.

\begin{defn}
In a calculation, a \emph{potential cancellation} is a substring of the form $$y^{t_1}\sigma y^{t_2}\qquad \sigma\in 2^{<\omega},t_1,t_2\in \{\pm1\}$$ satisfying that upon performing a finite set of
substitutions on this substring we encounter a substring of the form $yy^{-1}$ or $y^{-1}y$.
\end{defn}

Potential cancellations can be identified in the following way using the transducer diagram above.
For example, consider the case when our string is of the form $y \sigma y$. 
Then we regard $y$ in the diagram as a start state, $y^{-1}$ as the end state, and $\sigma$ as the word we read.
The string is a potential cancellation if and only if $\sigma$ is read from the start state and accepted to the end state.

\begin{example}
The string $y10001y$ is a potential cancellation, since $$y10001y\to 11y0001y\to 110y01y\to 11010y^{-1}y$$
The strings $y000110y$ and $y000y^{-1}$ are not potential cancellations, since $$y000110y\to 0y0110y\to 010y^{-1}10y\to 01001yy$$
$$y000y^{-1}\to 0y0y^{-1}$$
and $0y0y^{-1}$ does not admit any further moves.
\end{example}

\begin{defn}\label{exponentofcalculation}
When a calculation has no potential cancellations, we say that it has
\emph{exponent} $n$ if $n$ is the number of occurrences of the symbols $y^{\pm}$.
For example, there is no potential cancellation in the calculation $10y0y^{-1}1111...$, and the exponent
is $2$.
\end{defn} 

The following was proved in \cite{LodhaMoore}. (Lemma $5.9$.)

\begin{lem} \label{potcanlemma}
Suppose that $\Lambda$ is in $\{0,1,y,y^{-1}\}^{\Nbb}$ and
contains no potential cancellations.
Then advancing any occurrence of a $y^{\pm}$, by applying a substitution move from the list described in Definition \ref{calculationexp}, results in a word with no
potential cancellations.
\end{lem}

It is also natural to consider finite strings in letters $y,y^{-1},0,1$,
equipped with the same substitutions as above.
Of course, for such strings one may not be able to advance the rightmost
occurrence of $y^{\pm 1}$.
However, such strings will be useful to consider in the following context.
Let $fy_{s_1}^{t_1}...y_{s_n}^{t_n}$ be a standard form
and let $u$ be a finite binary sequence such that $f$ acts on $u$
and $u\cdot f$ dominates $s_1,...,s_n$.
Given any infinite binary sequence $\psi$,
we know that $u\psi\cdot f=(u\cdot f)\psi$. 
Also, the calculation
of $fy_{s_1}^{t_1}...y_{s_n}^{t_n}$
on $u\psi$, which equals the calculation of 
$y_{s_1}^{t_1}...y_{s_n}^{t_n}$ on $(u\cdot f)\psi$,
contains a tail $\psi$.
Upon deleting this tail from this calculation,
we obtain the calculation of $fy_{s_1}^{t_1}...y_{s_n}^{t_n}$
on $u$.
Note that the choice of the sequence $\psi$ does not affect the definition of this calculation.
We denote the set of all such strings as $\{0,1,y,y^{-1}\}^{<\Nbb}$.
The following lemma was proved in \cite{LodhaMoore} (Lemma $5.10$.)

\begin{lem}\label{ulemma}
Suppose that $\Lambda$ is in $\{0,1,y,y^{-1}\}^{<\Nbb}$ and
contains no potential cancellations, and contains $n$ occurrences of $y^{\pm}$, and hence has exponent $n$. 
Then there is a finite binary sequence $u$ such that the following holds.
We can perform a sequence of substitutions on the calculation $\Lambda u$ 
to obtain a calculation of the form $v y^n$, where $v$ is a finite binary sequence.
\end{lem}

Note that the notion of potential cancellation for a calculation extends naturally to an analogous notion for (weak) standard forms.

\begin{defn}\label{potcandefn}
(Potential cancellation)
A (weak) standard form $fy_{s_1}^{t_1}...y_{s_n}^{t_n}$ is said to have a 
\emph{potential cancellation} if
there is an infinite binary sequence $\tau$ such that the calculation of 
$y_{s_{1}}^{t_1}...y_{s_n}^{t_n}$ on $\tau$
contains a potential cancellation. 
\end{defn}

We now describe a family of basic operations involving standard forms.
Notice that each of these substitutions corresponds either to a relation in $R$ or to
a group theoretic identity.

\begin{defn}\label{moves}
The following manipulations of $S$-words that emerge from the relations in $R$ will be used throughout the article and denoted as \emph{moves}.
\begin{enumerate}
 \item (Rearranging move)  $y_t^i x_s^{\pm1} \rightarrow x_s^{\pm1} y_{t\cdot x_s^{\pm1}}^i$
 where $s,t\in 2^{<\Nbb}$ are such that $t.x_s^{\pm 1}$ is defined.
 \item (Expansion move) $y_s \rightarrow x_s y_{s\seq{0}} y_{s\seq{10}}^{-1} y_{s\seq{11}}$
 and $y_s^{-1} \rightarrow x_s^{-1} y_{s\seq{00}}^{-1} y_{s\seq{01}} y_{s\seq{1}}^{-1}$
 where $s\in 2^{<\Nbb}$.
 \item (Commuting move) $y_u y_v \leftrightarrow y_v y_u$ 
 where $u,v\in 2^{\Nbb}$ and $u,v$ are independent.
 \item (Cancellation move) Delete an occurrence of $y^i_s y^{-i}_s$.
 \item (ER moves)
This is a combination of moves applied on (weak) standard forms, and is called the \emph{expansion followed by rearrangement}, or \emph{ER move}.
Let $$f(y_{s_1}^{t_1}....y_{s_n}^{t_n})y_u^v(y_{p_1}^{q_1}...y_{p_m}^{q_m})$$
be a (weak) standard form such that either: 
\begin{enumerate}
\item $v=1$ and $x_u$ acts on each sequence in the set $\{s_1,...,s_n\}$.
\item $v=-1$ and $x_u^{-1}$ acts on each sequence in the set $\{s_1,...,s_n\}$.
\end{enumerate}
Then the two step move in either case is defined as follows:

\begin{enumerate}

\item $$f(y_{s_1}^{t_1}....y_{s_n}^{t_n})y_u(y_{p_1}^{q_1}...y_{p_m}^{q_m})$$ $$\to f(y_{s_1}^{t_1}....y_{s_n}^{t_n})(x_uy_{u0}y_{u10}^{-1}y_{u11})(y_{p_1}^{q_1}...y_{p_m}^{q_m})$$
$$\to  fx_{u}(y_{s_1\cdot x_u}^{t_1}....y_{s_n\cdot x_u}^{t_n})(y_{u0}y_{u10}^{-1}y_{u11})(y_{p_1}^{q_1}...y_{p_m}^{q_m})$$

\item $$f(y_{s_1}^{t_1}....y_{s_n}^{t_n})y_u^{-1}(y_{p_1}^{q_1}...y_{p_m}^{q_m})$$ $$\to f(y_{s_1}^{t_1}....y_{s_n}^{t_n})(x_u^{-1}y_{u00}^{-1}y_{u01}y_{u1}^{-1})(y_{p_1}^{q_1}...y_{p_m}^{q_m})$$
$$\to  fx_{u}^{-1}(y_{s_1\cdot x_u^{-1}}^{t_1}....y_{s_n\cdot x_u^{-1}}^{t_n})(y_{u00}^{-1}y_{u01}y_{u1}^{-1})(y_{p_1}^{q_1}...y_{p_m}^{q_m})$$

\end{enumerate}

\end{enumerate}
\end{defn}

Note that performing any of the above moves does not change the group element that the word represents.
The following Lemma was proved in \cite{LodhaMoore}.

\begin{lem}\label{stdform}
If $W$ is any $S$-word and $l \in \Nbb$, then $W$ can be converted into a standard form $f\lambda$ using the relations in $R$,
so that $f\lambda$ has depth at least $l$.
Moreover, each combinatorial manipulation performed in this process is one of the moves $(1)-(4)$ described in Definition \ref{moves}.
\end{lem}

Moreover, the following holds, and is an easy exercise which follows from the definitions.

\begin{lem}
If we apply a finite sequence of ER moves on a weak standard form, then the resulting word is also a weak standard form.
\end{lem}

ER moves shall be used frequently in the proofs in the paper.
The percolating elements $y_{s0}, y_{s10}^{-1}, y_{s11}$ 
obtained upon performing an ER move on $y_s$ are called \emph{offsprings} of $y_s$.
Similarly, the elements $y_{s00}^{-1},y_{s01},y_{s1}^{-1}$ obtained upon performing an ER move
on $y_s^{-1}$ are offsprings of $y_s^{-1}$. 
If we perform a sequence of such moves, we will also refer to offsprings of offsprings as offsprings themselves.
For instance, consider the move $y_s\to x_s y_{s0}y_{s10}^{-1}y_{s11}$ produces offsprings $y_{s0},y_{s10}^{-1},y_{s11}$
of $y_s$. 
Now upon performing a subsequent ER move on $y_{s10}^{-1}$, we obtain
$$x_s y_{s0} (x_{s10}^{-1}y_{s1000}^{-1}y_{s1001}y_{s101}^{-1}) y_{s11}\to x_s x_{s10}^{-1} y_{s0} (y_{s1000}^{-1}y_{s1001}y_{s101}^{-1}) y_{s11}$$
So we say that $$y_{s0,} y_{s1000}^{-1},y_{s1001},y_{s101}^{-1}, y_{s11}$$ are offsprings of $y_s$.

\begin{lem}\label{potcanstd}
Let $f\lambda_1$ be a weak standard form which does not have a potential cancellation.
Let $g \lambda_2$ be a weak standard form obtained by performing a sequence of ER moves on 
$f\lambda_1$.
Then $g\lambda_2$ does not have a potential cancellation.
\end{lem}

\begin{proof}
Let $\lambda_1=y_{s_1}^{t_1}...y_{s_n}^{t_n}$,
such that $t_i\in \{1,-1\}$ (for notational convenience).
Since the existence of potential cancellations in $fy_{s_1}^{t_1}...y_{s_n}^{t_n}$ is not affected by $f\in F$, we assume that this is the empty word.
Assume that $x_{s_i}$ acts on each $s_1,...,s_{i-1}$, and that $t_i=1$.
The case when $t_i=-1$ is completely analogous.
Applying an ER move, we obtain 
$$(y_{s_1}^{t_1}...y_{s_{i-1}}^{t_{i-1}})(x_{s_i}y_{s_i0}y_{s_i10}^{-1}y_{s_i11})(y_{s_{i+1}}^{t_{i+1}}...y_{s_n}^{t_n})$$
$$= (x_{s_i}) (y_{s_1'}^{t_1}...y_{s_{i-1}'}^{t_{i-1}})(y_{s_i0}y_{s_i10}^{-1}y_{s_i11})(y_{s_{i+1}}^{t_{i+1}}...y_{s_n}^{t_n})$$
where $s_j'=s_j\cdot x_{s_i}$ for $1\leq j\leq i-1$.

Now assume that we have introduced a potential cancellation after performing this move.
Let $\tau$ be an infinite binary sequence for which the associated calculation $\Lambda$ of $$(y_{s_1'}^{t_1}...y_{s_{i-1}'}^{t_{i-1}})(y_{s_i0}y_{s_i10}^{-1}y_{s_i11})(y_{s_{i+1}}^{t_{i+1}}...y_{s_n}^{t_n})$$
contains a potential cancellation.
It must be the case that $\tau$ contains either $s_i0, s_i10$, or $s_i11$ as a prefix,
since for all other sequences the calculation remains the same before and after the move.

Now let $\Lambda'$ be the calculation produced by the evaluation of $y_{s_1}^{t_1}...y_{s_n}^{t_n}$ on $\tau\cdot x_{s_i}^{-1}$.
The calculation $\Lambda$ is in fact obtained from $\Lambda'$ by advancing a symbol $y$.
By our hypothesis, there is no potential cancellation in $\Lambda'$.
So by Lemma \ref{potcanlemma} there cannot be a potential cancellation in $\Lambda$.
Hence we obtain a contradiction.
\end{proof}

The action of $F$ on infinite binary sequences preserves tail equivalence, yet the elements 
$y_s^n$ do not.
This is a fundamental and important distinction between the groups $F$ and $G$. 
The advantage of standard forms without potential cancellations is that we can identify precisely where the relation is not preserved.
Recall that the \emph{support} of a homeomorphism $\phi$ is the set of elements in the domain that are moved by the homeomorphism.
We denote this as $Supp(\phi)$.

\begin{defn}\label{denselymixing}
The action of an element of $G$ on an open subset $U\subset 2^{\mathbf{N}}$ is said to be \emph{densely mixing},
if there is a dense subset $V$ of $U$ such that the action of the element on any sequence in $V$ does not preserve tail equivalence.
\end{defn}

\begin{lem}\label{taileq}
Let $fy_{s_1}^{t_1}...y_{s_n}^{t_n}$ be a standard form that does not contain potential cancellations.
Let $U\subset 2^{\mathbf{N}}$ be the support of $y_{s_1}^{t_1}...y_{s_n}^{t_n}$.
\begin{enumerate}
\item $fy_{s_1}^{t_1}...y_{s_n}^{t_n}$ is densely mixing on $U\cdot f^{-1}$.
\item $fy_{s_1}^{t_1}...y_{s_n}^{t_n}$ preserves tail equivalence on $2^{\mathbf{N}}\setminus (U\cdot f^{-1})$.
\end{enumerate}
\end{lem}

\begin{proof}
Since $$U\cdot f^{-1}\cdot f=U=Supp(y_{s_1}^{t_1}...y_{s_n}^{t_n})$$
$fy_{s_1}^{t_1}...y_{s_n}^{t_n}$ acts like an element of $F$ on the complement of $U\cdot f^{-1}$, hence it preserves tail equivalence on this set.
For simplicity, we may assume that $f$ is trivial, and show that the action of $y_{s_1}^{t_1}...y_{s_n}^{t_n}$ on $U$ is densely mixing.

Let $U'$ be an open subset of $U$.
Let $\tau$ be a finite binary sequence such that the following holds:
\begin{enumerate}
\item $[\tau 0^{\infty}, \tau 1^{\infty}]\subset U'$ 
\item $\tau$ dominates $s_1,...,s_n$.
\end{enumerate}
The associated calculation $\Lambda$ of $y_{s_1}^{t_1}...y_{s_n}^{t_n}$ on $\tau$ does not contain potential cancellations,
by our hypothesis.
By Lemma \ref{ulemma}, there is a finite binary sequence $u$ such that one can perform substitutions on $\Lambda u$
to obtain $v y^m$ where $v$ is a finite binary sequence and $m\neq 0$.
We assume that $m>0$, the other case is analogous.  
It follows that the associated calculation of $y_{s_1}^{t_1}...y_{s_n}^{t_n}$ on $\tau u 0^{2^m}10^{2^m}1...$
equals $$vy^m0^{2^m}10^{2^m}1...$$
This calculation produces the output $$v 01^{2^m} 0 1^{2^m}...$$ which is not tail equivalent to the input.
\end{proof}

\section{Normal Forms}

For our group presentation $G=\langle S,R\rangle$,
we shall describe a \emph{normal form}- a unique, canonical choice of word for each group element together with a procedure that converts a given word into such a word.
The normal form described in this section (or slightly weaker variants of it) will be used throughout the paper,
in particular in the proof that the complex $X$ is contractible.
However, aside from the goals of this paper, 
this normal form is a useful tool to study the group
and is therefore of independent interest.
The reader is encouraged to develop a technical familiarity of the moves described in
\ref{moves}, before reading the proofs in this section.

It is useful to imagine a weak standard form $y_{s_1}^{t_1}...y_{s_n}^{t_n}$ as a 
set of decorations on the infinite rooted binary tree.
This is done by means of adding integer labels $t_i$ on the nodes $s_i$. 
Using such a picture, one can read off a calculation by reading the labels along the infinite path formed
by the infinite binary sequence on which the calculation is being performed.
The two central notions in this section are \emph{potential cancellation} and \emph{potential contraction}.
In particular, our \emph{normal form} will be a standard form $fy_{s_1}^{t_1}...y_{s_n}^{t_n}$
such that $f$ is in the usual normal form for Thompson's group $F$, and $y_{s_1}^{t_1}...y_{s_n}^{t_n}$
does not contain any potential cancellations or potential contractions.

We have already defined the notion of potential cancellations.
We shall now provide a tautological reformulation of the definition that will be useful in the proofs in this section.

\begin{defn}
For a weak standard form $fy_{s_1}^{t_1}...y_{s_n}^{t_n}$, 
we say that the pair $y_{s_j}^{t_j}, y_{s_i}^{t_i}$ is \emph{adjacent}
if $s_i\subset s_j$ and for each sequence $u$ satisfying $s_i\subset u\subset s_j$, it holds that
$u\notin \{s_1,...,s_n\}$.
This means that if we walk along the geodesic in the tree connecting $s_i$ and $s_j$,
we do not encounter any labels besides those at the first and the last nodes.

Such an adjacent pair is said to be a \emph{potential cancellation}
if $$y^{t_i} \sigma y^{t_j}\qquad \text{ for }s_i\sigma=s_j$$
is a potential cancellation string.
\end{defn}

The following Lemma is an immediate consequence of the definitions.

\begin{lem}
A weak standard form contains a potential cancellation if and only if it contains an adjacent pair that is a potential cancellation.
\end{lem}


\begin{defn}\label{potentialcontraction}
A weak standard form is said to contain a \emph{potential contraction} 
if either of the following holds.
\begin{enumerate}
\item It contains an occurrence of a subword of the form 
$y_{s0}y_{s10}^{-1}y_{s11}$ but no occurrences of $y_{s1}^{\pm}$.
\item It contains an occurrence of a subword of the form 
$y_{s00}^{-1}y_{s01}y_{s1}^{-1}$ but no occurrences of $y_{s0}^{\pm}$.
\end{enumerate}
If a weak standard form $fy_{s_1}^{t_1}...y_{s_n}^{t_n}$ contains a potential contraction of the form 
$y_{s0}y_{s10}^{-1}y_{s11}$ (in particular, no occurrences of $y_{s1}^{\pm}$), we can replace 
the subword $y_{s0}y_{s10}^{-1}y_{s11}$ with the word
$x_{s}^{-1}y_s$.
If $y_u^v$ is a percolating element occurring to the left of $x_s^{-1}$ in the resulting word,
then either $u,s$ are independent or $s\subset u$.  
Since there are no occurrences of $y_{s1}^{\pm}$,
the element $x_s^{-1}$ must act on $u$. 
So we apply rearranging substitutions to move the $x_{s}^{-1}$ to 
the left of all percolating elements.
The resulting word is a weak standard form.
We call this a \emph{contraction move}.
The contraction move for the other case, i.e. subwords of the form 
$y_{s00}^{-1}y_{s01}y_{s1}^{-1}$ (but no occurrences of $y_{s0}^{\pm}$), is defined in a similar way.
\end{defn}

Recall that any $S$-word can be converted into a standard form using a sequence of moves thanks to Lemma \ref{stdform}.
The procedure for converting a standard form $fy_{s_1}^{t_1}...y_{s_n}^{t_n}$ into a normal form using moves involves three steps:
\begin{enumerate}
\item {\bf Step 1:} Using a sequence of moves, convert $fy_{s_1}^{t_1}...y_{s_n}^{t_n}$ into a weak standard form $f'y_{u_1}^{v_1}...y_{u_m}^{v_m}$
that does not contain potential cancellations.
\item {\bf Step 2}: Using a sequence of contraction moves, convert $f'y_{u_1}^{v_1}...y_{u_m}^{v_m}$ into a weak standard form 
which does not contain potential contractions. 
Moreover, the moves performed do not introduce any potential cancellations.
Finally, using commuting moves, convert this weak standard form into a standard form $f'' y_{r_1}^{q_1}...y_{r_k}^{q_k}$.
\item {\bf Step 3} Convert $f''$ into a word $g$ which is in $F$-normal form as described in Theorem \ref{Fnormalform}.
\end{enumerate}
{\bf Output}: A standard form $g y_{r_1}^{q_1}...y_{r_k}^{q_k}$ such that $g$ is in $F$-normal form as in \ref{Fnormalform},
and $y_{r_1}^{q_1}...y_{r_k}^{q_k}$ does not contain any potential cancellations or potential contractions.
This word shall be the unique normal form as stated in the main theorem below.

\begin{thm}\label{normalform}
 For each element in $G$, there is a unique word $fy_{s_1}^{t_1}...y_{s_n}^{t_n}$ such that 
\begin{enumerate}
 \item $fy_{s_1}^{t_1}...y_{s_n}^{t_n}$ is a standard form with no potential contractions or potential cancellations.
\item $f\in F$ is an $X$-word in $F$-normal form in the sense of \ref{Fnormalform}.
\end{enumerate}
Moreover, given any word that represents this element in $G$, we can convert it into the word $fy_{s_1}^{t_1}...y_{s_n}^{t_n}$ using a sequence of moves
from \ref{moves}.
\end{thm}

The goal of the rest of this section is to describe Steps $1,2$ of the procedure and then prove the main theorem.
For Step $3$ we refer the reader to \cite{Cannon}. 

\subsection{Step $1$ of the procedure}

First we make a fundamental observation concerning potential cancellations
in standard forms.
This will provide a part of the reduction in Step $1$.

\begin{lem}\label{simplepotcan}
Let $(y_{s_1}^{t_1}...y_{s_n}^{t_n})y_s^t$ be a standard form such that the following holds.
\begin{enumerate}
\item $t_1,...,t_n,t\in \{1,-1\}$.
\item The sequences $s_1,...,s_n$ are independent.
\item Each pair $y_{s_i}^{t_i}, y_s^t$ is an adjacent pair 
that is a potential cancellation.
\end{enumerate}
Then we can perform a sequence of ER and cancellation moves on this word to produce a standard form $f y_{u_1}^{v_1}...y_{u_m}^{v_m}$
such that $u_1,...,u_m$ are independent and $v_1,...,v_m\in \{1,-1\}$.
\end{lem}

\begin{proof}
We assume that $t=1$, the other case is completely analogous.
If $y_s=y_{s_i}^{-t_i}$ for some $1\leq i\leq n$, we simply reduce the word.
Assume that $y_s\neq y_{s_i}^{-t_i}$. Performing an expansion move on $y_s$ produces a word $x_sy_{s0}y_{s10}^{-1}y_{s11}$.
We claim that $x_s$ acts on $s_1,...,s_n$.
Note that the only node in the subtree rooted at $s$ that $x_s$ does not act on is $s0$.
However, $s0\notin \{s_1,...,s_n\}$, since the word $y_{s0}^{t} y_{s}$ does not contain a potential cancellation for any $t\in \{\pm 1\}$.
This proves the claim.


Applying the rearrangement move, we obtain
$$x_s (y_{s_1'}^{t_1}...y_{s_n'}^{t_n})(y_{s0}y_{s10}^{-1}y_{s11})$$
where $s_i'=s_i\cdot x_s$.

Let $\sigma$ be whichever of $s0,s10$ or $s11$ is an initial segment of $s_i'$,
and let $y_{\sigma}^{\tau}$ be the corresponding percolating element.
Note that the pair $y_{s_i'}^{t_i}, y_{\sigma}^{\tau}$ will be an adjacent pair with a potential cancellation in the word.
Also, note that the distance in the tree between $s_i'$ and $\sigma$
is smaller than the distance between $s_i$ and $s$.
In this way, we can continue to apply ER moves
on $y_{s0}, y_{s10}^{-1}, y_{s11}$ and their offsprings, until this distance is $0$ for each such pair.
Then we can perform the required cancellations.
At this stage, any remaining percolating elements $y_{u_1}^{v_1},...,y_{u_n}^{v_n}$
are offsprings of $y_s$ and hence have the property that $u_1,...,u_n$ are pairwise independent.
\end{proof}

Now we describe the final part of the first step.

\begin{lem}\label{nopotcan}
Any weak standard form can be converted to a weak standard form with no potential 
cancellations using the relations.
\end{lem}

\begin{proof}
We prove this by induction on the number of percolating elements in the weak standard form.
The base case is trivial.
Let this be true for $n-1$.
Let $fy_{s_1}^{t_1}...y_{s_n}^{t_n}$ be a weak standard form such that $t_i\in \{1,-1\}$ (for convenience of notation).

By the inductive hypothesis it follows that we can convert 
$fy_{s_1}^{t_1}...y_{s_{n-1}}^{t_{n-1}}$ into a weak standard form
with no potential cancellations using the relations.
Moreover, from applying Lemmas \ref{stdform} and \ref{potcanstd} 
it follows that this can further be converted into a weak standard form $hy_{p_1}^{q_1}...y_{p_m}^{q_m}$ 
with no potential cancellations and depth strictly larger than $|s_n|$.
It follows that $hy_{p_1}^{q_1}...y_{p_m}^{q_m}y_{s_n}^{t_n}$ is a weak standard form,
since none of the sequences $p_i$ are initial segments of $s_n$.

Any potential cancellation is witnessed by an adjacent pair of the form
$y_{p_i}^{q_i},y_{s_n}^{t_n}$ for some $1\leq i\leq m$.
We make a record of all such adjacent pairs.
Using commuting moves we convert $hy_{p_1}^{q_1}...y_{p_m}^{q_m}y_{s_n}^{t_n}$ into a weak standard form 
$$h(y_{r_1}^{k_1}...y_{r_l}^{k_l})(y_{u_1}^{v_1}...y_{u_o}^{v_o})y_{s_n}^{t_n}$$ such that 
\begin{enumerate}
\item The word $h(y_{r_1}^{k_1}...y_{r_l}^{k_l})(y_{u_1}^{v_1}...y_{u_o}^{v_o})y_{s_n}^{t_n}$
is a weak standard form and $v_1,...,v_o\in \{1,-1\}$.
\item Each pair $y_{u_j}^{v_j},y_{s_n}^{t_n}$
is an adjacent pair that is a potential cancellation.
\item No adjacent pair of the form $y_{r_i}^{k_i},y_{s_n}^{t_n}$ or $y_{r_i}^{k_i},y_{u_j}^{v_j}$
is a potential cancellation.
\end{enumerate}

By Lemma \ref{simplepotcan}, we can perform a sequence of ER and cancellation substitutions on the word $(y_{u_1}^{v_1}...y_{u_o}^{v_o}) y_{s_n}^{t_n}$
to obtain a weak standard form with no potential cancellations.

Now we apply a sequence of ER moves to convert the weak standard form $hy_{r_1}^{k_1}...y_{r_l}^{k_l}$
into a weak standard form $h' \lambda_1$ of large depth so that the following holds. 
In the reduction process applied to $(y_{u_1}^{v_1}...y_{u_o}^{v_o}) y_{s_n}^{t_n}$ in the previous paragraph,
for each expansion move $y_s\to x_sy_{s0}y_{s10}^{-1}y_{s11}$ or $y_s^{-1}\to x_s^{-1}y_{s00}^{-1}y_{s01}y_{s1}^{-1}$,
the elements $x_s,x_s^{-1}$ act on each percolating element of $\lambda_1$.
Note that replacing $$h(y_{r_1}^{k_1}...y_{r_l}^{k_l})(y_{u_1}^{v_1}...y_{u_o}^{v_o})y_{s_n}^{t_n}$$ by $$h'\lambda_1(y_{u_1}^{v_1}...y_{u_o}^{v_o})y_{s_n}^{t_n}$$ does not introduce any new potential cancellations thanks to an application of Lemma \ref{potcanlemma}.

This means that we can perform a sequence of ER moves on the weak standard form $$h'\lambda_1(y_{u_1}^{v_1}...y_{u_o}^{v_o}) y_{s_n}^{t_n}$$
to obtain a weak standard form $$h''\lambda_2(y_{l_1}^{m_1}...y_{l_o}^{m_o})$$
without potential cancellations, and such that $y_{l_1}^{m_1},...,y_{l_o}^{m_o}$ are offsprings of $y_{s_n}^{t_n}$.
After applying a sequence of commutation moves, we obtain a standard form with no potential cancellations.
\end{proof}

\subsection{Step 2 of the procedure}
\begin{lem}\label{nopotcont}
 Any standard form can be converted into a standard form with no potential contractions 
 and no potential cancellations.
\end{lem}

\begin{proof}
 By Lemma \ref{nopotcan} we can convert our standard form into a 
 standard form with no potential cancellations.
On the resulting standard form, we perform contraction moves (defined in \ref{potentialcontraction}), one by one, 
until no contraction moves can be performed.
This process terminates, since upon performing a contraction substitution we 
obtain a standard form which has fewer occurrences of $y^{\pm}$.
Moreover, thanks to Lemma \ref{potcanlemma} we do not introduce any potential cancellations in this process.
\end{proof}

Before we proceed to prove the main theorem, we shall make a fundamental observation concerning exponents
of calculations.
(Recall the Definition of exponents of calculations from \ref{exponentofcalculation}.)
This provides an invariant that distinguishes different standard forms.

\begin{lem}\label{exponent}
Let $fy_{s_1}^{t_1}...y_{s_n}^{t_n}$ and $gy_{p_1}^{q_1}...y_{p_m}^{q_m}$ be standard forms without potential cancellations
that represent the same group element.
Let $u$ be a finite binary sequence such that $u_1=u\cdot f$ and $u_2=u\cdot g$ are defined and dominate $s_1,...,s_n$ and $p_1,...,p_m$ respectively.
Consider the calculations $\Theta$ of $fy_{s_1}^{t_1}...y_{s_n}^{t_n}$ on $u$ and $\Lambda$ of $gy_{p_1}^{q_1}...y_{p_m}^{q_m}$ on $u$.
Then the exponents of $\Theta$ and $\Lambda$ are the same. 
\end{lem}

\begin{proof}
Assume without loss of generality that the exponent of $\Lambda$ is greater than that of $\Theta$.
Assume that $k>0$ is the exponent of $\Lambda$.
(Recall that the exponent is always nonnegative.)
By Lemma \ref{ulemma}, there is a finite binary sequence $v$ such that
we can perform moves on $\Lambda v$ to obtain a calculation of the form $w y^k$,
where $w$ is a finite binary sequence.

Now consider the infinite binary sequence $v 0^{2^k} 1 0^{2^k} 10^{2^k}1...$.
We know that our standard forms represent the same group element.
So the calculations $$\Theta v 0^{2^k} 1 0^{2^k} 10^{2^k}1...\qquad \Lambda v 0^{2^k} 1 0^{2^k} 10^{2^k}1...$$
produce the same output.
Note that 
$$\Lambda v 0^{2^k} 1 0^{2^k} 10^{2^k}1...= w y^k 0^{2^k} 1 0^{2^k} 10^{2^k}1...=w 0 1^{2^k} 01^{2^k}0 1^{2^{k}}...$$

The calculation $$\Theta v 0^{2^k} 1 0^{2^k} 10^{2^k}1...$$
has fewer than $k$ occurrences of $y^{\pm}$. 
It is a pleasant visual exercise for the reader to show that this calculation
cannot produce an output that is tail equivalent to a sequence of the form
$$w 0 1^{2^k} 01^{2^k}0 1^{2^{k}}...$$
This means that our assumption that the exponents are different must be false.
\end{proof}

{\bf Proof of Theorem \ref{normalform}}

\begin{proof}
From Lemmas \ref{nopotcan} and \ref{nopotcont}, it follows that 
for any given group element we can find a word satisfying the hypothesis of the theorem.
It remains to show that this is unique.
By way of contradiction assume that 
$$fy_{s_1}^{t_1}...y_{s_n}^{t_n}\qquad gy_{p_1}^{q_1}...y_{p_m}^{q_m}$$ are two distinct standard forms 
that satisfy conditions $(1)$ and $(2)$ of the Theorem,
and represent the same element of $G$.
Assume without loss of generality that $p_m\geq s_n$.

There are three cases:
\begin{enumerate}
\item $s_n=p_m$ and either $t_n,q_m<0$ or $t_n,q_m>0$.
\item $s_n=p_m$ and either $t_n<0<q_m$ or $q_m<0<t_n$.
\item $s_n< p_m$.
\end{enumerate}

Case $(1)$: We can perform a cancellation on both sides of the equality, until we land in one of the remaining cases.
Case $(2)$ reduces to case $(3)$ by moving $y_{s_n}^{t_n}$ from one side to the other to obtain 
$$fy_{s_1}^{t_1}...y_{s_{n-1}}^{t_{n-1}}\qquad gy_{p_1}^{q_1}...y_{p_m}^{q_m-t_n}$$

Now we consider case $(3)$. Consider the standard form $(y_{s_1}^{t_1}...y_{s_n}^{t_n})y_{p_m}^{-1}$.
Note that either of the following holds:

\begin{enumerate}
\item There is an infinite binary sequence $\sigma$ such that for each finite $\sigma_1\subset \sigma$,
$p_m\sigma_1\notin \{s_1,...,s_n\}$.
\item There is an adjacent pair of the form $s_i,p_m$ which is not a potential cancellation.
\end{enumerate}

Note that if both are false, then we can perform contraction moves on $y_{s_1}^{t_1}...y_{s_n}^{t_n}$,
thus contradicting our hypothesis.
(This is easy to check in case $(3)$.)
Hence one of the above must hold.
Therefore it follows that there is a finite binary sequence $u$ such that the associated calculation $\Lambda$ of $(y_{s_1}^{t_1}...y_{s_n}^{t_n})y_{p_m}^{-1}$ on $u$ does not contain a potential cancellation.

We consider the following additional calculations:
\begin{enumerate}
\item[(1)] $\Theta$ is the calculation of $gy_{p_1}^{q_1}...y_{p_m}^{q_m-1}$ on $u$.
\item[(2)] $\Lambda'$ is the calculation of $fy_{s_1}^{t_1}...y_{s_n}^{t_n}$ on $u$.
\item[(3)] $\Theta'$ is the calculation of  $gy_{p_1}^{q_1}...y_{p_m}^{q_m}$ on $u$.
\end{enumerate}

According to our assumptions, the calculations $\Lambda, \Lambda', \Theta,\Theta'$
do not contain any potential cancellations.
First observe that:
\begin{enumerate}
\item Since $$f(y_{s_1}^{t_1}...y_{s_n}^{t_n})y_{p_m}^{-1} = g y_{p_1}^{q_1}...y_{p_m}^{q_m-1}$$ from Lemma \ref{exponent} the exponent of $\Lambda$ is the same as that of $\Theta$.
\item Since $$fy_{s_1}^{t_1}...y_{s_n}^{t_n} = g y_{p_1}^{q_1}...y_{p_m}^{q_m}$$ from Lemma \ref{exponent} the exponent of $\Lambda'$ is the same as that of $\Theta'$.
\end{enumerate}

However, it also holds that:
\begin{enumerate}
\item The exponent of $\Lambda$ is greater than that of $\Lambda'$.
\item The exponent of $\Theta$ is less than or equal to that of $\Theta'$.
\end{enumerate}

This is a contradiction.
Therefore the claim must be false, and cases $(3)$ is not possible.
It follows that the words $$y_{s_1}^{t_1}...y_{s_n}^{t_n}\qquad y_{p_1}^{q_1}...y_{p_m}^{q_m}$$
must be equal.
Upon cancellation, we obtain that $f=g$, and since these words are in the normal form for $F$,
they must also be equal.
Therefore, the $S$-words $$fy_{s_1}^{t_1}...y_{s_n}^{t_n}\qquad gy_{p_1}^{q_1}...y_{p_m}^{q_m}$$
are equal.
\end{proof}

The following is an immediate consequence, and provides a normal form representative for right cosets of $F$ in $G$.
We leave the proof as an elementary exercise for the reader.

\begin{cor}\label{normalformcor}
Let $y_{s_1}^{t_1}...y_{s_n}^{t_n}$ be a $Y$-word in normal form.
If $y_{u_1}^{v_1}...y_{u_m}^{v_m}$ is also $Y$-word in normal form such that $$y_{u_1}^{v_1}...y_{u_m}^{v_m}\in F(y_{s_1}^{t_1}...y_{s_n}^{t_n})$$ 
then $y_{u_1}^{v_1}...y_{u_m}^{v_m}$ and $y_{s_1}^{t_1}...y_{s_n}^{t_n}$ are the same words.
\end{cor}

\section{Cluster complexes}

The complex $X$ we construct in this article will be a certain type of CW complex which we shall call a \emph{cluster complex}.
The purpose of this section is to provide the definition of a cluster complex, and discuss some basic formalisms surrounding this notion.
In this section, we shall use the concepts concerning hyperplane arrangements described in Subsection \ref{hyparr} from the Preliminaries.

We consider $\mathbf{R}^n$ with variables $x_1,...,x_n$
that represent coordinates in the usual orthonormal basis.
We consider $\{x_1,...,x_n\}$ as an ordered set.
In particular, we fix the order on this set which is induced by the usual ordering of $\mathbf{N}$ on the indices.

We consider two types of affine hyperplanes in $\mathbf{R}^n$.
\begin{enumerate}
\item An affine hyperplane of $\mathbf{R}^n$ is said to be of type $1$ 
if it is of the form $\{x_i=0\}$ or $\{x_i=1\}$ for some $1\leq i\leq n$.

\item A hyperplane of $\mathbf{R}^n$ is said to be of type $2$ if 
it is of the form
$\{x_i=x_{i+1}\}$ for some $1\leq i\leq n-1$.
\end{enumerate}

Let $\mathcal{A}$ be a hyperplane arrangement in $\mathbf{R}^n$.
We say that $\mathcal{A}$ is \emph{admissible} if the following holds.
\begin{enumerate}
\item Each hyperplane in $\mathcal{A}$ is of type $1$ or of type $2$.
\item $\mathcal{A}$ contains all the hyperplanes of type $1$.
\end{enumerate}

Now we define the notion of an $n$-cluster.
Recall the notion of the face complex of a finite hyperplane arrangement from subsection \ref{hyparr}.

\begin{defn}\label{cluster}
An $n$-cluster is a CW subdivision of $[0,1]^n$ obtained by restricting the face complex
of an admissible hyperplane arrangement to $[0,1]^n$.

Let $\mathcal{A}$ be an admissible hyperplane arrangement in $\mathbf{R}^n$.
We denote by $\mathcal{C(A)}$ as the restriction of the face complex of $\mathcal{A}$ to $[0,1]^n$,
which is a subcomplex.
In particular, we say that $\mathcal{C(A)}$ is the $n$-cluster (or simply cluster) associated with $\mathcal{A}$.
\end{defn}

Hence we have a correspondence:
$$\{\text{admissible hyperplane arrangements in }\mathbf{R}^n\}\leftrightarrow \{\text{$n$-clusters}\}$$

\begin{example}
The admissible arrangement $\mathcal{A}$ given by the set of all type $1$ hyperplanes corresponds
to the standard CW structure of the regular Euclidean cube $[0,1]^n$.
If $\mathcal{A}$ consists of all type $1$ hyperplanes of $\mathbf{R}^2$ and the type $2$ hyperplane $\{x_1=x_2\}$,
then the resulting $2$-cluster is a square with a diagonal connecting $(0,0)$ and $(1,1)$.
This has four $0$-cells, five $1$-cells, and two triangular $2$-cells.
\end{example}

Our next step is to define the notion of a \emph{subcluster} of a cluster.
First we shall define a codimension $1$ subcluster, and then subsequently define lower dimensional subclusters. 


\begin{defn}
Let $\mathcal{A}$ be an admissible hyperplane arrangement, and let $\mathcal{C(A)}\subset \mathbf{R}^n$ be the associated $n$-cluster.
A \emph{codimension $1$ subcluster} of $\mathcal{C(A)}$ is obtained by taking the intersection of a hyperplane $H$ in $\mathcal{A}$ with $\mathcal{C(A)}$.
\end{defn}

A codimension $1$ subcluster automatically inherits a definition in terms of an admissible hyperplane arrangement in $\mathbf{R}^{n-1}$.
To see this, let $H\in \mathcal{A}$ be a hyperplane.
Now $H$ is of one of the following forms
$$\{x_k=x_{k+1}\}\qquad \{x_k=1\}\qquad \{x_k=0\}$$
In each respective case, we endow $H$ with a set of coordinates 
$y_1,...,y_{n-1}$ where $y_j=x_j$ if $j< k$ and $y_j=x_{j+1}$ if $k\leq j\leq n-1$.
The restriction $\mathcal{A}\restriction H$ provides an admissible hyperplane arrangement 
for $H\cong \mathbf{R}^{n-1}$.

This generalises to the following notion of subclusters.

\begin{defn}
A \emph{subcluster} of $\mathcal{C(A)}$ is a subcomplex obtained by an intersection
of $\mathcal{C(A)}$ with a flat $T$ in $\mathcal{A}$.
\end{defn}

We categorize subclusters of a given cluster as \emph{facial} and \emph{diagonal}.
This is done by separating the flats of $\mathcal{A}$ into two types.
We say that a flat $T$ of $\mathcal{A}$ is of \emph{type 1}, if it is of the form $\bigcap_{H\in A}H$,
where $A\subseteq \mathcal{A}$ is a collection of type $1$ hyperplanes.

We say that a flat $T$ of $\mathcal{A}$ is of \emph{type 2}, if is of the form
$\bigcap_{H\in A}H$, where $A\subseteq \mathcal{A}$ has the following property.
$A$ contains at least one hyperplane of the form $\{x_k=x_{k+1}\}$ so that $T$
is not contained in any of the following hyperplanes:
$$\{x_k=0\}\qquad \{x_k=1\}\qquad \{x_{k+1}=0\}\qquad \{x_{k+1}=1\}$$

\begin{defn}
Consider a subcluster obtained by taking an intersection of the cluster with a flat $T$.
If $T$ is a type $1$ flat, then the subcluster is said to be a \emph{facial subcluster}.
If $T$ is a type $2$ flat, then the subcluster is said to be a \emph{diagonal subcluster}.
\end{defn}

The following Lemma is a corollary to a similar statement
concerning the face complex.

\begin{lem}
Each cell $e$ of a cluster is the interior of a convex polytope.
The set of $0$-cells incident to $e$ is convex independent,
i.e. no one element lies in the convex span of the others.
Moreover, the closure of $e$ equals the convex span of the $0$-cells incident to $e$. 
\end{lem}

We now define the notion of complex of clusters.

\begin{defn}
(Complex of clusters) A complex of clusters is a CW complex obtained by gluing clusters along subclusters using cellular homeomorphisms.
\end{defn}

The notion of a \emph{nonpositvely curved cube complex} admits
a natural generalisation in this setting, despite the absence of a natural metric.
One may view this notion as a combinatorial criterion that guarantees asphericity.

\begin{defn}\label{npccluster}
Let $C$ be a complex of clusters.
We say that $C$ is \emph{nonpositively curved} if for every finite subcomplex $Y$ of $C$,
there is a nonpositively curved cube complex $Z$, and continuous maps $f:Y\to Z, g:Z\to C$ so that the following diagram commutes:

\[
\begin{tikzcd}
 & Z \arrow{dr}{g} \\
Y \arrow{ur}{f} \arrow[rr,hook] && C
\end{tikzcd}
\]
\end{defn}

The following is an immediate consequence of the definition.

\begin{prop}
If $C$ is a nonpositively curved cluster complex, then $C$ is aspherical.
\end{prop}

We shall demonstrate that the complex $X$ we construct in this article is nonpositively curved in this sense.
We end this section by providing the reader with some pictures of clusters in $\mathbf{R}^3$.
I thank Matt Zaremsky for allowing me to use these pictures which he has meticulously constructed.
The reader may ignore the \emph{parametrizations} that label each picture, until they read further.
It shall become apparent in subsequent sections what these parametrizations mean and the precise clusters of our complex $X$
that they describe.

\begin{figure}[htb]
\centering\input{crowded_cluster.tex}
\caption{The $3$-cluster based at $F$ parameterized by the special forms $y_{s0}$, $y_{s10}^{-1}$, and $y_{s11}$.}
\label{fig:crowded_cluster}
\end{figure}

\begin{figure}[htb]
\centering\input{medium_cluster.tex}
\caption{The $3$-cluster based at $F$ parameterized by the special forms $y_{s0}$, $y_{s10}^{-1}$, and $y_{s111}$.}
\label{fig:medium_cluster}
\end{figure}


\section{The $1$-skeleton of $X$}

In this section, we describe and study $X^{(1)}$.

\begin{defn}
The $0$-skeleton $X^{(0)}$ is defined to be the set of right cosets of $F$ in $G$.
The action of $G$ on $X^{(0)}$ is the usual right action on right cosets.
\end{defn}

We observe that this action is faithful.
While this is not necessary for the proof of type $F_{\infty}$,
we include the proof since it is short and clarifies the group action on the $0$-skeleton of our complex.

\begin{lem}
The action of $G$ on the right cosets of $F$ in $G$ is faithful.
\end{lem}

\begin{proof}
The kernel of the action is the normal core $\bigcap_{g\in G} g^{-1} F g$.
Let $f$ be a nontrivial element of the normal core.
This necessarily must be an element of $F$,
since $\bigcap_{g\in G} g^{-1} F g\subseteq F$.
Since $f$ is a nontrivial element of $F$,
there is a non-constant finite binary sequence $s$ such that
$s$ and $s\cdot f$ are independent.

Now from our assumption it follows that $f\in y_s^{-1} F y_s$.
So there is an element $h\in F$ such that $y_s f y_s^{-1}=h$.
After applying a rearrangement move, we obtain $h= fy_{s\cdot f} y_s^{-1}$.
We claim that the element $fy_{s\cdot f} y_s^{-1}\notin F$.
The element $y_{s\cdot f} y_s^{-1}$ does not preserve the tail equivalence relation on the sequence 
$s011011011....$ since $s,s\cdot f$ are independent.
Our claim follows.
This means that our assumption must be false, and hence the normal core is trivial.
\end{proof}

To define the edge relation, it will be useful to consider certain types of standard forms, 
which we shall call \emph{special forms}.

\begin{defn}
A (nonempty) standard form $y_{s_1}^{t_1}...y_{s_n}^{t_n}$ is called a \emph{special form}, if the following holds.
\begin{enumerate}
\item $s_1,...,s_n$ are consecutive.
\item $t_1,...,t_n\in \{-1,1\}$ and $t_{i+1}=-t_i$.
\end{enumerate}
\end{defn}

The definition implies that each pair of percolating elements in the special form
commute.
Moreover, there is a finite rooted binary tree $T$, rooted at some $s\in 2^{<\mathbf{N}}$ such that 
$s_1,...,s_n$ are consecutive leaves in $T$.

Given a percolating element $y_s$, 
if we apply an expansion move, we obtain $x_s y_{s0}y_{s10}^{-1}y_{s11}$.
The word $y_{s0}y_{s10}^{-1}y_{s11}$ is a special form.
Now we apply an ER move on $y_{s10}^{-1}$ to obtain:

$$x_s y_{s0}y_{s10}^{-1}y_{s11}\to x_s y_{s0}(x_{s10}^{-1} y_{s1000}^{-1}y_{s1001}y_{s101}^{-1})y_{s11}$$

$$\to x_sx_{s10}^{-1} (y_{s0} y_{s1000}^{-1}y_{s1001}y_{s101}^{-1}y_{s11})$$
The $Y$-word $$y_{s0} y_{s1000}^{-1}y_{s1001}y_{s101}^{-1}y_{s11}$$
is also a special form.
This generalises to the following in a straightforward manner.

\begin{lem}\label{percspecial}
Let $s\in 2^{<\Nbb}$ such that $s\neq 0^k, 1^k, \emptyset$. 
If we apply a sequence of ER moves on $y_{s}^{\pm 1}$,
to obtain a standard form $fy_{s_1}^{t_1}...y_{s_n}^{t_n}$,
then the $Y$-word $y_{s_1}^{t_1}...y_{s_n}^{t_n}$
is a special form.
Moreover, the support of $f\in F$ lies in the interval $[s0^{\infty}, s1^{\infty}]$.
\end{lem}

Two basic observations about special forms are the following.

\begin{lem}\label{special2}
Let $y_{s_1}^{t_1}...y_{s_n}^{t_n}$ be a special form, and let $f\in F$ be an element such that
$f$ acts on each sequence in the set $\{s_1,...,s_n\}$.
It follows that $y_{s_1\cdot f}^{t_1}...y_{s_n\cdot f}^{t_n}$ is also a special form.
\end{lem}
\begin{proof}
The proof is elementary, and uses the fact that tree diagrams
for elements of $F$ map consecutive pairs of sequences, upon which they act, to consecutive pairs.
\end{proof}

\begin{lem}\label{largedepth}
Let $y_{s_1}^{t_1}...y_{s_n}^{t_n}$ be a special form and let $f\in F$ and $l\in \mathbf{N}$.
We can perform a sequence of ER moves on $y_{s_1}^{t_1}...y_{s_n}^{t_n}$ to obtain a standard form $f_1y_{u_1}^{v_1}...y_{u_m}^{v_m}$, for some $f_1\in F$,
such that the following holds:
\begin{enumerate}
\item $v_1=t_1$.
\item $y_{u_1}^{v_1}...y_{u_m}^{v_m}$ is a special form.
\item $f$ acts on $u_1,...,u_m$ and $y_{u_1\cdot f}^{v_1}...y_{u_m\cdot f}^{v_m}$ is a special form of depth at least $l$.
\end{enumerate}
\end{lem}

\begin{proof}
If $f_1y_{u_1}^{v_1}...y_{u_m}^{v_m}$ is a standard form obtained from applying a sequence of ER moves to a special form $y_{s_1}^{t_1}...y_{s_n}^{t_n}$,
then it is apparent from the definition of the expansion moves that $v_1=t_1$.
If this has sufficiently large depth, then the conclusion clearly holds.

\end{proof}

We remark that given a special form $y_{s_1}^{t_1}...y_{s_n}^{t_n}$, it holds that $$(y_{s_1}^{t_1}...y_{s_n}^{t_n})^{-1}=y_{s_n}^{-t_n}...y_{s_1}^{-t_1}=y_{s_1}^{-t_1}...y_{s_n}^{-t_n}$$
where the last equality follows from applying a sequence of commuting moves.
So by convention, we refer to the \emph{formal inverse} of the special form $y_{s_1}^{t_1}...y_{s_n}^{t_n}$ as $y_{s_1}^{-t_1}...y_{s_n}^{-t_n}$.
Now we are ready to define the edges in $X^{(1)}$.

\begin{defn}\label{edgerelation}
The following equivalent conditions determine the edge relation in $X^{(1)}$.
\begin{enumerate}
\item (E1) $F\tau_1, F\tau_2\in X^{(0)}$ are connected by an edge if the double coset $F\tau_1\tau_2^{-1} F$
is equal to one of the following double cosets.  
$$F y_{10} F\qquad Fy_{10}^{-1} F$$
$$ Fy_{100}^{-1} y_{101} F\qquad F y_{100} y_{101}^{-1} F$$

\item (E2) $F\tau_1, F\tau_2\in X^{(0)}$ are connected by an edge if the element $\tau_1\tau_2^{-1}$
admits a standard form $fy_{s_1}^{t_1}...y_{s_n}^{t_n}$
such that $y_{s_1}^{t_1}...y_{s_n}^{t_n}$ is a special form.
\end{enumerate}
\end{defn}
It is not clear a priori that $(E2)$ is well defined, however this follows from Lemma \ref{axiomsedge} below.
The advantage of the formulation in $(E1)$ is that it provides a succinct definition of the $1$-skeleton.
However, in practise, $(E2)$ will be more useful.
Note that the special forms $y_{10}^{-1}, y_{100}y_{101}^{-1}$ are the formal inverses of $y_{10},y_{100}^{-1}y_{101}$ respectively.

\begin{lem}\label{axiomsedge}
The two conditions $(E1)$ and $(E2)$ in Definition \ref{edgerelation} are equivalent.
\end{lem}

\begin{proof}
First we will show that $(E1)\implies (E2)$.
Assume that $F\tau_1, F\tau_2$ satisfy $(E1)$.
Hence $$\tau_1\tau_2^{-1}= f_1 \nu f_2$$ for some $$f_1,f_2\in F\qquad \nu\in \{y_{10}, y_{100}^{-1}y_{101}, y_{10}^{-1}, y_{100}y_{101}^{-1}\}$$
Using Lemma \ref{largedepth}, we apply a sequence of ER moves on $\nu$
to obtain a special form $f_3y_{s_1}^{t_1}...y_{s_n}^{t_n}$
of sufficiently large depth such that $f_2$ acts on $s_1,...,s_n$ and $$y_{s_1}^{t_1}...y_{s_n}^{t_n}\qquad y_{s_1\cdot f_2}^{t_1}...y_{s_n\cdot f_2}^{t_n}$$ are special forms.

So
$$\tau_1\tau_2^{-1}= (f_1 f_3)(y_{s_1}^{t_1}...y_{s_n}^{t_n}) f_2$$
Applying a rearrangement move, we obtain
$$\tau_1\tau_2^{-1}= (f_1f_3f_2) (y_{s_1\cdot f_2}^{t_1}...y_{s_n\cdot f_2}^{t_n})$$
which is the desired conclusion.



Next we show that if a pair $F\tau_1,F\tau_2\in X^{(0)}$ satisfies $(E2)$,
then it also satisfies $(E1)$.
By our assumption,
there is a special form $y_{s_1}^{t_1}...y_{s_n}^{t_n}$ and $f\in F$
such that $\tau_1\tau_2^{-1}= fy_{s_1}^{t_1}...y_{s_n}^{t_n}$.
We assume that $t_1=1$, the other case is very similar.
We consider two cases.



\begin{enumerate}
\item $n$ is odd.

\item $n$ is even.

\end{enumerate}

Case $1: $ Upon applying a sequence of $\frac{n-1}{2}$ ER moves to the percolating element $y_{10}$
and its offsprings, one by one in an ad hoc fashion, we obtain a standard form $hy_{u_1}^{v_1}...y_{u_n}^{v_n}$ so that (by Lemma \ref{percspecial}) $y_{u_1}^{v_1}...y_{u_n}^{v_n}$ is a special form.
Observe that $v_1=1$.

Now we construct a pair of finite rooted binary trees $T_1,T_2$ such that:
\begin{enumerate}
\item $T_1,T_2$ have the same number of leaves.
\item $s_1,...,s_n$ are consecutive leaves of $T_1$.
\item $u_1,...,u_n$ are consecutive leaves of $T_2$.
\item The tree diagram $(T_1,T_2)$ maps $s_i$ to $u_i$
for each $1\leq i\leq n$.
\end{enumerate}
The construction of such a tree pair is an elementary exercise in the group $F$.
The main idea is the following.
Start with trees $T_3,T_4$ such that $s_1,..., s_n$ are leaves in $T_3$ and $u_1,...,u_n$ are leaves in $T_4$.
Now add carets to leaves of $T_3,T_4$ to obtain trees $T_1,T_2$ so that:
\begin{enumerate}
\item $s_1,...,s_n$ are leaves in $T_1$ and $u_1,...,u_n$ are leaves in $T_2$.
\item The number of leaves to the left of $u_1$ in $T_2$ is the same as the number of leaves to the left of $s_1$
in $T_1$.
\item The number of leaves to the right of $u_n$ in $T_2$ is the same as the number of leaves to the right of 
$s_n$ in $T_1$.
\end{enumerate}

Let $g=(T_1,T_2)\in F$.
It follows that 
$$\tau_1\tau_2^{-1}g=hy_{s_1}^{t_1}...y_{s_n}^{t_n}g=hgy_{s_1\cdot g}^{t_1}...y_{s_n\cdot g}^{t_n}=hgy_{u_1}^{v_1}...y_{u_n}^{v_n}$$

Therefore, $$F\tau_1\tau_2^{-1} F= F y_{u_1}^{v_1}...y_{u_n}^{v_n} F= Fy_{10} F$$

Case $2$: Just as in case $1$, we assume that $t_1=1$. First we apply a sequence of $\frac{n-2}{2}$ ER moves to 
the percolating elements of $y_{100}y_{101}^{-1}$, and their offsprings in an ad hoc fasion,
to obtain $hy_{u_1}^{v_1}...y_{u_n}^{v_n}$ such that $y_{u_1}^{v_1}...y_{u_n}^{v_n}$ is a special form.
Now just as in case $(1)$ we construct an element $g\in F$
such that $s_i\cdot g=u_i$.
And we obtain $$\tau_1\tau_2^{-1}g= hy_{s_1}^{t_1}...y_{s_n}^{t_n}g= hg y_{u_1}^{v_1}...y_{u_n}^{v_n}$$
It follows that $$F \tau_1\tau_2^{-1} F=F  y_{u_1}^{v_1}...y_{u_n}^{v_n} F=F y_{100}y_{101}^{-1}F$$
\end{proof}

\begin{defn}
We say that two special forms $y_{s_1}^{t_1}....y_{s_n}^{t_n}$ and $y_{u_1}^{v_1}...y_{u_m}^{v_m}$
are \emph{equivalent} if they belong to the same right coset of $F$ in $G$. 
\end{defn}

Our analysis of $X$ shall involve studying special forms up to equivalence.
Hence it is natural to specialise basic manipulations of standard forms to special forms.

\begin{defn}\label{movesdefinition}
Given a special form $\lambda=y_{s_1}^{t_1}...y_{s_n}^{t_n}$,
an \emph{expansion move} at $y_{s_i}^{t_i}$ entails one of the following:
\begin{enumerate}
\item Replacing $y_{s_i}^{t_i}$ by $y_{s_i\seq{0}}y^{-1}_{s_i\seq{10}}y_{s_i\seq{11}}$
if $t_i=1$.
\item Replacing $y_{s_i}^{t_i}$ by 
$y_{s_i\seq{00}}^{-1}y_{s_i\seq{01}}y^{-1}_{s_i\seq{1}}$
if $t_i=-1$.
\end{enumerate}
A \emph{contraction move} in $\lambda$ entails one of the following:
\begin{enumerate}
\item Replacing a subword $y_{s\seq{0}}y^{-1}_{s\seq{10}}y_{s\seq{11}}$
by $y_s$.
\item Replacing a subword $y_{s\seq{00}}^{-1}y_{s\seq{01}}y^{-1}_{s\seq{1}}$
by $y_s^{-1}$ 
\end{enumerate}
\end{defn}

When we use the phrase \emph{expansion move} or \emph{contraction move},
it will be clear from the context whether we mean this in the sense of Definition \ref{moves}
or Definition \ref{movesdefinition}.
In particular, moves on special forms that produce new special forms will be in the sense of Definition \ref{movesdefinition}.

\begin{lem}\label{moveslemma}
The following conditions hold for special forms.
\begin{enumerate}
 \item Performing contraction and expansion moves on a special form produces special forms that are equivalent.
\item Given any pair of equivalent special forms,
one can be obtained from the other by performing a
sequence of expansion and contraction moves.
\end{enumerate}
\end{lem}

\begin{proof}
 Let $y_{s_1}^{t_1}...y_{s_n}^{t_n}$ be a special form.
 If one of the substitutions $$y_s\to x_sy_{s\seq{0}}y_{s\seq{10}}^{-1}y_{s\seq{11}}\qquad y_s^{-1}=x_{s}^{-1}y_{s\seq{00}}^{-1}y_{s\seq{01}}y_{s\seq{1}}^{-1}$$
 is performed, then $x_s^{\pm}$ commutes with every percolating element
 that occurs on the left.
 This means that $x_s^{\pm}$ can be moved to the left past the percolating elements and then deleted from the word to obtain 
 a special form that lies in the same coset.
 This implies that performing expansion moves on special forms in the sense of Definition \ref{movesdefinition} yields special forms.
 The same holds for contractions.
 
 Consider two special forms that lie in the same coset.
 Special forms cannot have potential cancellations, by definition.
 So we apply a sequence of contraction moves (in the sense of \ref{movesdefinition}) to both, until we obtain 
 normal forms
 $$y_{s_1}^{t_1}...y_{s_n}^{t_n}\qquad y_{p_1}^{q_1}...y_{p_m}^{q_m}$$
 Since both these special, normal forms are equivalent by the above, they lie in the same coset. 
 By Theorem \ref{normalform} they must in fact be equal as words.
 This provides the proof of the second claim.
\end{proof}

\begin{remark}\label{specialnormalform}
Special normal forms are characterised by the property that they are special forms upon which no contraction move can be performed.
Given a special normal form, any equivalent special form can be obtained by performing a sequence of expansion moves.
\end{remark}

\begin{defn}
We define $\Omega$ to be the set of right cosets of $F$ in $G$
which admit a special form as a coset representative.
Note that since special forms must be nonempty, the trivial coset is not an element of $\Omega$.
\end{defn}

\begin{remark}
Let $y_{s_1}^{t_1}...y_{s_n}^{t_n}$ be a special form and let $f\in F$. 
As seen before, we can apply a sequence of expansion moves in the sense of \ref{movesdefinition} to $y_{s_1}^{t_1}...y_{s_n}^{t_n}$
to obtain an equivalent special form $y_{u_1}^{v_1}...y_{u_m}^{v_m}$ with the property that
$f$ acts on $u_1,...,u_m$.
This means that $$(Fy_{s_1}^{t_1}...y_{s_n}^{t_n})\cdot f= (Fy_{u_1}^{v_1}...y_{u_m}^{v_m})\cdot f= Fy_{u_1\cdot f}^{v_1}...y_{u_m\cdot f}^{v_m}$$
Thanks to Lemma \ref{special2}, $Fy_{u_1\cdot f}^{v_1}...y_{u_m\cdot f}^{v_m}\in \Omega$.
It follows that the right action of $F$ on $X^{(0)}$ restricts to a right action of $F$ on $\Omega$.
\end{remark}

\subsection{The actions of $F$ on $\Omega$ and $X^{(1)}$}\label{Omega}

Our main goal for the remainder of the section is to understand the action of $G$ on $X^{(1)}$
and demonstrate that this is cocompact and the stabilizers are of type $\mathbf{F}_{\infty}$.
Any $1$-cell contains in its $G$-orbit a $1$-cell of the form $F,Fy_{s_1}^{t_1}...y_{s_n}^{t_n}$ where $y_{s_1}^{t_1}...y_{s_n}^{t_n}$
is a special form.
To understand the stabilizer of this cell,
an important step is to understand the stabilizer of $Fy_{s_1}^{t_1}...y_{s_n}^{t_n}$ under the action of $F$.
Indeed it turns out that the two stabilizers are the same.
For this reason, it is crucial to study the action of $F$ on $\Omega$ in more detail.
We do this below.

\begin{lem}\label{Factionspecial}
Consider elements $$Fy_{s_1}^{t_1}...y_{s_n}^{t_n}, Fy_{u_1}^{v_1}...y_{u_m}^{v_m}\in \Omega$$
where $y_{s_1}^{t_1}...y_{s_n}^{t_n}$ and $y_{u_1}^{v_1}...y_{u_m}^{v_m}$ are special forms. 
These elements are in the same $F$-orbit if and only if $n,m$ have the same parity and $t_1=v_1$.
In particular, the right action of $F$ on $\Omega$ has precisely four orbits.
\end{lem}

\begin{proof}
Consider a special form $y_{u_1}^{v_1}...y_{u_m}^{v_m}$ and $f\in F$.
By performing expansion moves on $y_{u_1}^{v_1}...y_{u_m}^{v_m}$ in the sense of \ref{movesdefinition} we obtain an equivalent special form
$y_{s_1'}^{t_1}...y_{s_n'}^{t_n}$ such that $f$ acts on $s_1',...,s_n'$. Note that the parity of $n,m$ is the same, since each expansion move increases the number of percolating elements by a factor of $2$. Furthermore, any expansion move
applied to the leftmost percolating element produces a new special form whose leftmost percolating element has the same exponent. It follows that $t_1=v_1$. 
So we obtain $$Fy_{s_1'}^{t_1}...y_{s_n'}^{t_n}\cdot f= Fy_{s_1}^{t_1}...y_{s_n}^{t_n}$$
where $s_i=s_i'\cdot f$.
It follows that for $$Fy_{u_1}^{v_1}...y_{u_m}^{v_m}\qquad Fy_{s_1}^{t_1}...y_{s_n}^{t_n}$$
$t_1=v_1$ and the parity of $m,n$ is the same.
 
 Now let $$y_{s_1}^{t_1}...y_{s_n}^{t_n}\qquad y_{u_1}^{v_1}...y_{u_m}^{v_m}$$
 be two special forms such that $n,m$ have the same parity and $t_1=v_1$.
 We wish to show that they are in the same $F$-orbit.
 Assume without loss of generality that $n\leq m$. If $n<m$, upon performing a sequence of expansion moves on $y_{s_1}^{t_1}...y_{s_n}^{t_n}$ we obtain a special form $y_{s_1'}^{t_1'}...y_{s_m'}^{t_m'}$ with precisely $m$ percolating elements
 and $t_1=t_1'$.
 If $m=n$, we simply refer to $y_{s_1}^{t_1}...y_{s_n}^{t_n}$ as $y_{s_1'}^{t_1'}...y_{s_m'}^{t_m'}$ for the rest of the proof.
 In the same manner as in the proof of Lemma \ref{axiomsedge}, we construct an $f\in F$ such that $s_i'\cdot f=u_i$ for $1\leq i\leq m$.
 We obtain:
  $$F(y_{s_1}^{t_1}...y_{s_n}^{t_n})\cdot f= F(y_{s_1'}^{t_1'}...y_{s_m'}^{t_m'})\cdot f=F(y_{u_1}^{v_1}...y_{u_m}^{v_m})$$
  as required.
\end{proof}

We now state and prove the main Proposition concerning the action of $G$ on $X^{(1)}$.
In particular, we show that this is cocompact and 
the stabilizers of cells are of type $\mathbf{F}_{\infty}$.
The proofs involve the \emph{Higman-Thompson groups} $F_{3}$, and $F_{3,r}$ for $r\in \mathbf{N}_{>0}$.
For an introduction to these groups, we refer the reader to \cite{BrownFiniteness}.
We shall only use the definitions of these groups and the fact that these are of type $\mathbf{F}_{\infty}$, which has been demonstrated in \cite{BrownFiniteness}.

\begin{prop}\label{Stab1cell}
 The action of $G$ on $X^{(1)}$ satisfies the following:
 \begin{enumerate}
 \item The action is transitive on the $0$-cells and there are precisely two orbits of $1$-cells.
  \item For each $1$-cell $e$ in $X^{(1)}$, $\textup{Stab}_G(e)$ is isomorphic to $F_3\times F\times F$.
  \end{enumerate}
\end{prop}

{\bf Proof of part (1)}:
The transitivity of the action on $X^{(0)}$ is apparent and we focus on studying the action on the set of edges.
First, observe that the $1$-cells $$F,Fy_{s_1}^{t_1}...y_{s_n}^{t_n}\qquad F,Fy_{s_1}^{-t_1}...y_{s_n}^{-t_n}$$ 
for a special form $y_{s_1}^{t_1}...y_{s_n}^{t_n}$, are in the same $G$-orbit.
This is because the former is obtained from acting on the latter with the element $y_{s_1}^{t_1}...y_{s_n}^{t_n}$.
So we conclude that any $1$-cell contains in its $G$-orbit a $1$-cell of the form
$F, Fy_{s_1}^{t_1}...y_{s_n}^{t_n}$ where $t_1=1$.

Consider two $1$-cells $$e_1=F, Fy_{s_1}^{t_1}...y_{s_n}^{t_n}\qquad e_2=F, Fy_{u_1}^{v_1}...y_{u_m}^{v_m}$$ such that $t_1=v_1=1$.
First we claim that any element of $G$ that maps $e_1$ to $e_2$ must be an element of $F$.
Indeed, this would follow if any element that maps $e_1$ to $e_2$ maps the trivial coset $F$ to itself.
Assume by way of contradiction that there is an element $g\in G$ such that: 
$$F\cdot g=Fy_{u_1}^{v_1}...y_{u_m}^{v_m}\qquad Fy_{s_1}^{t_1}...y_{s_n}^{t_n}\cdot g=F$$
Consider a standard form representative $g=fy_{u_1}^{v_1}...y_{u_m}^{v_m}$ for some $f\in F$.

It follows that $$F y_{s_1}^{t_1}...y_{s_n}^{t_n}\cdot fy_{u_1}^{v_1}...y_{u_m}^{v_m}= F$$
Now let $y_{p_1}^{q_1}...y_{p_k}^{q_k}$ be a special form obtained from $y_{s_1}^{t_1}...y_{s_n}^{t_n}$
by expansion moves so that $f$ acts on $p_1,...,p_k$ and each $p_i\cdot f$ dominates $u_1,...,u_m$.
Note that $q_1=1$.

It follows that $$F y_{s_1}^{t_1}...y_{s_n}^{t_n}\cdot fy_{u_1}^{v_1}...y_{u_m}^{v_m}= F y_{p_1\cdot f}^{q_1}...y_{p_k\cdot f}^{q_k}y_{u_1}^{v_1}...y_{u_m}^{v_m}=F$$
and that $$\tau=y_{p_1\cdot f}^{q_1}...y_{p_k\cdot f}^{q_k}y_{u_1}^{v_1}...y_{u_m}^{v_m}$$
is a weak standard form.

In order for the equality $F \tau=F$ to hold, the percolating elements of $\tau$ vanish upon performing a sequence of ER moves. 
Moreover, both $y_{u_1}^{v_1}...y_{u_m}^{v_m}$ and $y_{p_1\cdot f}^{q_1}...y_{p_k\cdot f}^{q_k}$ are special forms.
Hence the pair $y_{p_1\cdot f}^{q_1}, y_{u_1}^{v_1}$ must be an adjacent pair that is a potential cancellation,
and that $p_1\cdot f=u_1 0^l$ for some $l\in \mathbf{N}$.
The calculation of $\tau$ on $u_10^{\infty}$ equals the calculation of $y_{p_1\cdot f}^{q_1}y_{u_1}^{v_1}$ on $u_10^{\infty}$ which equals
$$u_1y^{v_1}0^l y^{q_1}0^{\infty}=u_1y0^ly0^{\infty}$$
This cannot contain a potential cancellation since advancing a $y$ along a sequence of $0$'s does not change the sign of $y$ in a calculation.

It follows that $g\in F$ and maps $Fy_{s_1}^{t_1}...y_{s_n}^{t_n}$ to $Fy_{u_1}^{v_1}...y_{u_m}^{v_m}$.
Since $t_1=v_1$, by Lemma \ref{Factionspecial}, such an element exists if and only if the parity of $n,m$ is the same.
It follows that the action of $G$ on the $1$-cells has precisely two orbits.

{\bf Proof of Part (2)}:

Now we observe that the statement about the stabilizers reduces to Proposition \ref{stabilizerspecialform}, which is stated below.
As in the proof of part (1), since $G$ acts transitively on $X^{(0)}$, it follows that
each $1$-cell contains in its $G$-orbit a $1$-cell of the form $F,Fy_{s_1}^{t_1}...y_{s_n}^{t_n}$
where $y_{s_1}^{t_1}...y_{s_n}^{t_n}$ is a special form.
By the above, it suffices to understand 
the stabilizer of a $1$-cell $F,Fy_{s_1}^{t_1}...y_{s_n}^{t_n}$
where $y_{s_1}^{t_1}...y_{s_n}^{t_n}$ is a special form and $t_1=1$.
Moreover, also from the proof of part (1), any element $g\in G$ that maps the $1$-cell
$F,Fy_{s_1}^{t_1}...y_{s_n}^{t_n}$ onto itself must map $F$ to $F$ and
$Fy_{s_1}^{t_1}...y_{s_n}^{t_n}$ to $Fy_{s_1}^{t_1}...y_{s_n}^{t_n}$.
It follows that 
$$\textup{Stab}_G(F,Fy_{s_1}^{t_1}...y_{s_n}^{t_n})=
 \textup{Stab}_F(Fy_{s_1}^{t_1}...y_{s_n}^{t_n})$$ 
This reduces the proof of part (2) to Proposition \ref{stabilizerspecialform}.

\begin{prop}\label{stabilizerspecialform}
Let $y_{s_1}^{t_1}...y_{s_n}^{t_n}$ be a special form. Then $$\textup{Stab}_F(Fy_{s_1}^{t_1}...y_{s_n}^{t_n})\cong F_{3}\times F\times F$$
\end{prop}

Before we proceed to the proof of the above, we define some concepts that will be used in the proof.
The key idea is to provide a precise connection between tree diagrams for the group $F_3$ and the stabilizer above.

Let $s_1,...,s_n$ be consecutive finite binary sequences.
We consider the set of finite sub-forests of the infinite rooted binary tree,
with roots $s_1,...,s_n$.
Such a forest is determined by its leaves, which is a set of finite binary sequences $\{u_1,...,u_m\}$ satisfying:
\begin{enumerate}
\item For each $u_i$ there is an $s_j$ such that $s_j\subset u_i$. 
\item Let $\nu$ be an infinite binary sequence which contains an element of
$\{s_1,...,s_n\}$ as a prefix.
Then there is a unique element of the set $\{u_1,...,u_m\}$
which is a prefix of $\nu$.
\end{enumerate}
In other words, $s_1,...,s_n$ are the roots and $u_1,...,u_m$ are the leaves of the forest.
We also consider labels on the set of leaves of such a forest with elements of the set $\{-1,1\}$,
so that any pair of leaves that are consecutive binary sequences have distinct labels.
The set of all such \emph{forests with labeled leaves} and roots $s_1,...,s_n$ will be denoted by $\mathcal{F}(s_1,...,s_n)$.

Consider a special normal form $y_{s_1}^{t_1}...y_{s_n}^{t_n}$.
Each equivalent special form $y_{u_1}^{v_1}...y_{u_m}^{v_m}$ determines an element of $\mathcal{F}(s_1,...,s_n)$,
which has leaves $u_1,...,u_m$ labeled respectively by $v_1,...,v_m$ such that $v_1=t_1$.  
We shall now characterise precisely which elements of $\mathcal{F}(s_1,...,s_n)$
correspond to special forms equivalent to $y_{s_1}^{t_1}...y_{s_n}^{t_n}$.
Let $C_1$ and $C_2$ be finite binary trees rooted at the empty string, 
and with leaves $\{00,01,0\}$ and $\{0,10,11\}$ respectively.

The forest $T$ in 
$\mathcal{F}(s_1,...,s_n)$ corresponding to the special form $y_{s_1}^{t_1}...y_{s_n}^{t_n}$
is simply the forest with $n$ roots labelled by $t_1,...,t_n$, and no edges.

If $t_i=1$, then performing an expansion move at $y_{s_i}^{t_i}$
in $y_{s_1}^{t_1}...y_{s_n}^{t_n}$ produces 
$$(y_{s_1}^{t_1}...y_{s_{i-1}}^{t_{i-1}})(y_{s_i0}y_{s_i10}^{-1}y_{s_i11})(y_{s_{i+1}}^{t_{i+1}}...y_{s_n}^{t_n})$$
The forest corresponding to this special form is obtained by gluing the root of $C_2$ at $s_i$
in $T$, deleting the label on $s_i$,
and labelling $s_i0, s_i10, s_i11$ by $1,-1,1$ respectively.

If $t_i=-1$, then performing an expansion move at $y_{s_i}^{t_i}$
in $y_{s_1}^{t_1}...y_{s_n}^{t_n}$ produces 
$$(y_{s_1}^{t_1}...y_{s_{i-1}}^{t_{i-1}})(y_{s_i00}^{-1}y_{s_i01}y_{s_i1}^{-1})(y_{s_{i+1}}^{t_{i+1}}...y_{s_n}^{t_n})$$
The forest corresponding to this special form is obtained by gluing the root of $C_1$ at $s_i$
in $T$, deleting the label on $s_i$, and labelling $s_i00,s_i01,s_i1$ by $-1,1,-1$ respectively.

So each expansion move corresponds to attaching exactly one of $C_1$ or $C_2$
to the leaf where the expansion move is performed,
and relabelling.
This is determined by the exponent of the percolating element upon which it is performed.
We make this precise below.
In this context it makes sense to assume that the special form we begin with is a special normal form.
(Recall the discussion in \ref{specialnormalform}.)

\begin{lem}\label{forests}
Let $y_{s_1}^{t_1}...y_{s_n}^{t_n}$ be a special normal form.
The elements of $\mathcal{F}(s_1,...,s_n)$ that correspond to equivalent special forms are determined recursively by the following:
\begin{enumerate}
\item (The base case) The forest $T$ with no edges, and $n$ roots $s_1,...,s_n$ labelled with $t_1,...,t_n$ respectively. 
\item (Expansion) Let $T_1$ be an element of $\mathcal{F}(s_1,...,s_n)$ that corresponds to a special form equivalent to $y_{s_1}^{t_1}...y_{s_n}^{t_n}$.
\begin{enumerate}
\item Let $s$ be a leaf of $T_1$ labeled by $1$.
We obtain an element $T_2\in \mathcal{F}(s_1,...,s_n)$ by gluing the root of $C_2$ at $s$, deleting the label at $s$, and labelling $s0,s10,s11$ by $1,-1,1$ respectively.
$T_2$ corresponds to a special form equivalent to $y_{s_1}^{t_1}...y_{s_n}^{t_n}$.
\item Let $s$ be a leaf of $T_1$ labeled by $-1$.
We obtain an element $T_2\in \mathcal{F}(s_1,...,s_n)$ by gluing the root of $C_1$ at $s$, deleting the label at $s$, and labelling $s00,s01,s1$ by $-1,1,-1$ respectively.
$T_2$ corresponds to a special form equivalent to $y_{s_1}^{t_1}...y_{s_n}^{t_n}$.
\end{enumerate}
\end{enumerate}
Every element of $\mathcal{F}(s_1,...,s_n)$ that corresponds to a special form equivalent to $y_{s_1}^{t_1}...y_{s_n}^{t_n}$
can be obtained by applying a sequence of expansions described in $(2)$ to the forest $T$ described in the base case $(1)$.
\end{lem}

We denote the set of $3$-branching forests with $n$ roots by $\mathcal{F}_{3,n}$.
The set of forests obtained in Lemma \ref{forests} are in natural bijective correspondence
with $\mathcal{F}_{3,n}$.
We leave this a straightforward exercise for the reader.
The main idea is that both trees $C_1,C_2$ can be associated with a single
\emph{ternary} caret, i.e. a tree with a root connected with three leaves.
Since each addition of $C_1$ or $C_2$ is determined by the label,
this determines the correspondence.

Now consider two elements $T_1,T_2$ of $\mathcal{F}(s_1,...,s_n)$
such that $T_1,T_2$ have the same number of leaves.
The forest pair $(T_1,T_2)$ determines a prefix replacement 
homeomorphism of the rational interval $[s_10^{\infty},s_n1^{\infty}]$.
The set of all such forest pairs is closed under composition and inverses,
and hence forms a group of homeomorphisms of $[s_10^{\infty},s_n1^{\infty}]$.
In fact, the following holds.

\begin{lem}\label{subgroup}
Consider forest pairs $(T_1,T_2)\in Homeo^+([s_10^{\infty},s_n1^{\infty}])$ such that:
\begin{enumerate}
\item $T_1,T_2$ are forests in $\mathcal{F}(s_1,...,s_n)$ corresponding to special forms
equivalent to $y_{s_1}^{t_1}...y_{s_n}^{t_n}$.
\item $T_1,T_2$ have the same number of leaves.
\end{enumerate}
The set of all such pairs forms a group under composition.
This group is isomorphic to $F_{3,n}$.
The isomorphism is described by the aforementioned bijection
between forests pairs in $\mathcal{F}(s_1,...,s_n)$ that correspond to 
equivalent special forms and forests diagrams for $F_{3,n}$.
\end{lem}

\begin{proof}
The idea behind the proof is to mimic the tree diagrams for $F_{3,n}$
for this group while replacing the role of the ternary caret with $C_1$ and $C_2$.
The reader familiar with tree diagrams for $F_{3,n}$ will find this elementary.
\end{proof}

Recall that the groups $F_{3,n}$ for $n\in \mathbf{N}_{>0}$, are all isomorphic to the group $F_{3,1}=F_3$ (see \cite{BrownFiniteness}).
With this in mind, we proceed to the proof of our proposition.\\

{\bf Proof of Proposition \ref{stabilizerspecialform}}:

\begin{proof}
Let $y_{s_1}^{t_1}...y_{s_n}^{t_n}$ be a special normal form.
Consider $f\in \textup{Stab}_F(Fy_{s_1}^{t_1}...y_{s_n}^{t_n})$.
If $f$ does not act on $s_1,...,s_n$,
we perform a sequence of expansion moves to obtain a special form
$y_{u_1}^{v_1}...y_{u_m}^{v_m}$
such that $f$ acts on $u_1,...,u_m$.
By our assumption, $$Fy_{u_1}^{v_1}...y_{u_m}^{v_m}\cdot f= Fy_{u_1\cdot f}^{v_1}...y_{u_m\cdot f}^{v_m}$$
$$=Fy_{s_1}^{t_1}...y_{s_n}^{t_n}$$
It follows that we can perform a sequence of contraction moves on 
$y_{u_1\cdot f}^{v_1}...y_{u_m\cdot f}^{v_m}$ to obtain $y_{s_1}^{t_1}...y_{s_n}^{t_n}$.

We observe that the following holds for the element $f$:
\begin{enumerate}
\item $f$ stabilises the intervals $[0^{\infty}, s_10^{\infty}]$, $[s_10^{\infty},s_n1^{\infty}]$
and $[s_n1^{\infty},1^{\infty}]$, because $u_1=s_10^k$ and $u_m=s_n1^l$ for some $k,l\in \mathbf{N}$. 
\item The restriction of $f$ to $[s_10^{\infty},s_n1^{\infty}]$ admits a forest diagram $(T_1,T_2)$,
where $T_1, T_2$ are forests in $\mathcal{F}(s_1,...,s_n)$ that correspond to 
special forms $y_{u_1}^{v_1}...y_{u_m}^{v_m}$ and $y_{u_1\cdot f}^{v_1}...y_{u_m\cdot f}^{v_m}$
respectively.
\end{enumerate}
This defines a monomorphism of $Stab_F(Fy_{s_1}^{t_1}...y_{s_n}^{t_n})$
into $F\times F_{3,n}\times F$.
Thanks to Lemma \ref{subgroup} the projection to $F_{3,n}$ is surjective.
Observe that we also have an epimorphism, since if we modify the restrictions to $[0^{\infty}, s_10^{\infty}]$
and $[s_n1^{\infty},1^{\infty}]$ to any element of the copies of $F$ supported on these respective intervals, the new element is also in $Stab_F(Fy_{s_1}^{t_1}...y_{s_n}^{t_n})$.
Therefore, this is an isomorphism.
\end{proof}




\subsection{A detour on expansion moves and the expansion Lemmas}
We now take a detour to prove two Lemmas concerning this notion of expansion moves. 
We prove it here since this is the most natural place to prove it, even though we shall only make use of it later in the paper in Subsection \ref{intersectionclus} and also in the final Section.
In particular, the reader will benefit from the intuition and some notation from the previous subsection while reading these proofs.

In this section we use both notions of expansion coming from Definitions \ref{movesdefinition} and \ref{moves}.
The two notions are closely related, and we explore this relationship in some depth in this subsection.
The notion of expansion moves as defined for special forms in \ref{movesdefinition} can be defined for any word $y_{s_1}^{t_1}...y_{s_n}^{t_n}$
with the property that $s_1,...,s_n$ are independent and $t_i\in \{-1,1\}$.
And the contents of the Lemma \ref{moveslemma} also hold for such words with the same convention that equivalence for such words 
means that they lie in the same coset.

Now expansion moves in the sense of \ref{movesdefinition} on percolating elements of $y_{s_1}^{t_1}...y_{s_n}^{t_n}$ correspond to a subset of elements of $\mathcal{F}(s_1,...,s_n)$
as in the discussion above.
Similar ideas as above apply in the situation where  $y_{s_1}^{t_1}...y_{s_n}^{t_n}$ is not a special form, but satisfies that $s_1,...,s_n$ are independent and $t_i\in \{-1,1\}$.
In particular, the set of equivalent $Y$-words obtained by applying such expansion moves on $y_{s_1}^{t_1}...y_{s_n}^{t_n}$ is in natural bijective correspondence with finite $3$-ary trees with $n$ roots.

Before we proceed, we state the elementary relationship between the two notions of expansion.

\begin{lem}\label{elemrelmoves}
Consider a word $y_{s_1}^{t_1}...y_{s_n}^{t_n}$
with the property that $s_1,...,s_n$ are independent and $t_i\in \{-1,1\}$.
Let $y_{u_1}^{v_1}...y_{u_m}^{v_m}$ be obtained from performing expansion moves on $y_{s_1}^{t_1}...y_{s_n}^{t_n}$ in the sense of Definition \ref{movesdefinition}.

Then there is an $f\in F$ such that $fy_{u_1}^{v_1}...y_{u_m}^{v_m}$ can be obtained from performing expansion (followed by rearrangement) moves on $y_{s_1}^{t_1}...y_{s_n}^{t_n}$ in the sense of Definition \ref{moves}.
Moreover, $Supp(f)\subseteq Supp(y_{s_1}^{t_1}...y_{s_n}^{t_n})$ and $f^{-1}$ acts on $u_1,...,u_n$.
\end{lem}

\begin{proof}
This follows immediately from the definitions.
\end{proof}

Now we state and prove the weak expansion Lemma.

\begin{lem}\label{compatibleexpansionlemma}
(Weak expansion Lemma) Let $y_{s_1}^{t_1}...y_{s_n}^{t_n}$ and $y_{u_1}^{v_1}...y_{u_m}^{v_m}$ be words satisfying that:
\begin{enumerate}
\item $s_1,...,s_n$ are independent and $t_1,...,t_n\in \{1,-1\}$.
\item $u_1,...,u_m$ are independent and $v_1,...,v_m\in \{1,-1\}$.
\item $Fy_{s_1}^{t_1}...y_{s_n}^{t_n}=Fy_{u_1}^{v_1}...y_{u_m}^{v_m}$
\end{enumerate}
Then we can perform expansion moves (in the sense of \ref{movesdefinition}) on both words to obtain the same word $y_{p_1}^{q_1}...y_{p_l}^{q_l}$
such that $p_1,...,p_l$ are independent and $q_1,...,q_l\in \{1,-1\}$.
\end{lem}

\begin{proof}
Performing a sequence of contraction moves on both words produces the same $Y$-word $y_{r_1}^{j_1}...y_{r_k}^{j_k}$ in normal form,
thanks to Corollary \ref{normalformcor}.

The set of $Y$-words obtained from applying a sequence of expansion moves on $y_{r_1}^{j_1}...y_{r_k}^{j_k}$ corresponds to finite $3$-ary forests with $k$ roots.
(Recall the discussion in the previous subsection.)
Two such forests correspond to $y_{s_1}^{t_1}...y_{s_n}^{t_n}$ and $y_{u_1}^{v_1}...y_{u_m}^{v_m}$.
By adding ternary carets to both these forests, we produce the same forest, which corresponds to a $Y$-word $y_{p_1}^{q_1}...y_{p_l}^{q_l}$ with the desired properties.
\end{proof}

The Strong Expansion Lemma will provide the core technical step in the proof of asphericity in the last section.
Before we state and prove the Lemma, we consider some additional elementary features of expansions.

The reader familiar with the arguments in \cite{LodhaMoore} will already be aware of this, but for the sake of completeness we discuss first the following feature.
As a motivating example, consider the expansion $y_s=x_sy_{s0}y_{s10}^{-1}y_{s11}$.
Now taking inverses, we obtain $y_s^{-1}=y_{s11}^{-1}y_{s10}y_{s0}^{-1} x_s^{-1}$.
Now $$(s11)\cdot x_s^{-1}=s1\qquad (s10)\cdot x_s^{-1}=s01\qquad  (s0)\cdot x_s^{-1}=s00$$

So upon applying a rearrangement move, followed by commutation moves we get $$y_{s11}^{-1}y_{s10}y_{s0}^{-1} x_s^{-1}=x_s^{-1} y_{s1}^{-1} y_{s01}y_{s00}^{-1}= x_s^{-1}y_{s00}^{-1}y_{s01}y_{s1}^{-1}$$
And $x_s^{-1}y_{s00}^{-1}y_{s01}y_{s1}^{-1}$ is precisely the result of applying an expansion move (in the sense of \ref{moves}) to $y_s^{-1}$.
Using an elementary inductive argument this generalises to the following:

\begin{lem}\label{elementaryinversion}
Let $y_s^t$ be such that $t\in \{1,-1\}$.
Let $fy_{s_1}^{t_1}...y_{s_n}^{t_n}$ be a standard form obtained from applying expansion moves (in the sense of \ref{moves}) to $y_s^t$.
Then the following holds:
\begin{enumerate}
\item $f^{-1}$ acts on $s_1,...,s_n$.
\item $f^{-1} y_{s_1\cdot f^{-1}}^{t_1}...y_{s_n\cdot f^{-1}}^{t_n}$ can be obtained from applying expansion moves (in the sense of \ref{moves}) to $y_s^{-t}$.
\end{enumerate}
\end{lem}

\begin{defn}
Let $y_s^t$ and $y_u^v$ be a pair of percolating elements such that $t,v\in \{1,-1\}$.
Then the pair is said to contain \emph{no common offsprings} if the following holds:
For any pair $y_{s_1}^{t_1}...y_{s_n}^{t_n}$
and $y_{u_1}^{v_1}...y_{u_m}^{v_m}$ obtained from performing expansion moves (in the sense of \ref{movesdefinition}) on $y_s^t$ and $y_u^v$ respectively,
each pair $y_{s_i}^{t_i}, y_{u_j}^{v_j}$ is distinct.

For the negation of the above, we say that the pair $y_s^t$ and $y_u^v$ contains \emph{common offsprings}.
That is, $y_s^t$ and $y_u^v$ contain \emph{common offsprings} if there are $y_{s_1}^{t_1}...y_{s_n}^{t_n}$
and $y_{u_1}^{v_1}...y_{u_m}^{v_m}$ obtained from performing expansion moves on $y_s^t$ and $y_u^v$ respectively (in the sense of \ref{movesdefinition})
such that $y_{s_i}^{t_i}=y_{u_j}^{v_j}$ for some $1\leq i\leq n, 1\leq j\leq m$.
\end{defn}

Note that for a pair $y_s^t$ and $y_u^v$, if $s,u$ are independent, then the pair contains no common offsprings.
However, the case where $u\subseteq s$ is more subtle.
The following criterion is useful in detecting if such a pair does not have common offsprings.

\begin{lem}\label{critcomoff}
Let $y_s^t$ and $y_u^v$ be such that $t,v\in \{1,-1\}$ and $u\subseteq s$.
If $y_s^t$ and $y_u^v$ have common offsprings then the standard form $y_s^ty_u^{-v}$ contains a potential cancellation. 
\end{lem}

\begin{proof}
Assume that $y_s^t$ and $y_u^v$ have common offsprings.
Then we can apply expansion moves in the sense of \ref{movesdefinition} to obtain words $y_{s_1}^{t_1}...y_{s_n}^{t_n}$ and $y_{u_1}^{v_1}...y_{u_m}^{v_m}$
such that $y_{s_i}^{t_i}=y_{u_j}^{v_j}$ for some pair $i,j$.
Now there are elements $f,g\in F$ such that $fy_{s_1}^{t_1}...y_{s_n}^{t_n}$ and $gy_{u_1}^{v_1}...y_{u_m}^{v_m}$ can be obtained from $y_s^t$ and $y_u^v$ respectively
using expansions in the sense of \ref{moves}.

By Lemma \ref{elementaryinversion} we know that $g^{-1}$ acts on $u_1,...,u_m$.
Note that $g^{-1}$ acts on $s_i$ since $s_i=u_j$.
We also ensure that $g^{-1}$ acts on elements of the set $s_1,...,s_n$ by performing expansion moves (in the sense of \ref{moves}) to percolating elements $$y_{s_1}^{t_1},...,y_{s_{i-1}}^{t_{i-1}},y_{s_{i+1}}^{t_{i+1}},...y_{s_n}^{t_n}$$
in the word $fy_{s_1}^{t_1}...y_{s_n}^{t_n}$ if necessary.
Note that the resulting word also witnesses the common offspring since it contains the percolating element $y_{s_i}^{t_i}$. 
For notational convenience we also denote the resulting word by $fy_{s_1}^{t_1}...y_{s_n}^{t_n}$.

By Lemma \ref{elementaryinversion}, we know that the word $g^{-1}y_{u_1\cdot g^{-1}}^{-v_1}...y_{u_m\cdot g^{-1}}^{-v_m}$
can be obtained from $y_u^{-v}$ using expansions in the sense of \ref{moves}.
So we get $$y_s^ty_u^{-v}= (fy_{s_1}^{t_1}...y_{s_n}^{t_n})g^{-1}y_{u_1\cdot g^{-1}}^{-v_1}...y_{u_m\cdot g^{-1}}^{-v_m}$$
$$=fg^{-1}(y_{s_1\cdot g^{-1}}^{t_1}...y_{s_n\cdot g^{-1}}^{t_n})y_{u_1\cdot g^{-1}}^{-v_1}...y_{u_m\cdot g^{-1}}^{-v_m}$$
Now since $y_{s_i\cdot g^{-1}}^{t_i}=y_{u_j\cdot g^{-1}}^{v_j}$ we can see that there is a potential cancellation in $y_s^ty_u^{-v}$.
\end{proof}

The Strong Expansion Lemma, which we state and prove below, shall play a key role in of the proof of asphericity of $X$ in the final section.

\begin{lem}\label{strongexpansionlemma}
(Strong Expansion Lemma) 
Let $y_{s_1}^{t_1}...y_{s_n}^{t_n}$ and $y_{u_1}^{v_1}...y_{u_m}^{v_m}$ be words such that $s_1,...,s_n$ and $u_1,...,u_m$ are independent, and each $t_i,v_j\in \{1,-1\}$.
Then there are words $y_{p_1}^{q_1}...y_{p_l}^{q_l}$ and $y_{r_1}^{w_1}...y_{r_k}^{w_k}$ such that:
\begin{enumerate}
\item $y_{p_1}^{q_1}...y_{p_l}^{q_l}$ is obtained from $y_{s_1}^{t_1}...y_{s_n}^{t_n}$ by applying expansion moves as in \ref{movesdefinition}.
\item $y_{r_1}^{w_1}...y_{r_k}^{w_k}$ is obtained from $y_{u_1}^{v_1}...y_{u_m}^{v_m}$ by applying expansion moves as in \ref{movesdefinition}.
\item For each pair $y_{p_i}^{q_i}, y_{r_j}^{w_j}$ it holds that either $y_{p_i}^{q_i}= y_{r_j}^{w_j}$ or $y_{p_i}^{q_i}, y_{r_j}^{w_j}$ contains no common offsprings.
\end{enumerate}
\end{lem}

\begin{proof}
We show this by induction on $n$.
For the base case, we consider $y_s^t$ and  $y_{u_1}^{v_1}...y_{u_m}^{v_m}$ as the words such that $u_1,...,u_m$ are independent, and each $t,v_j\in \{1,-1\}$.
We apply expansion moves to $y_{u_1}^{v_1}...y_{u_m}^{v_m}$ to obtain a special form $y_{r_1}^{w_1}...y_{r_k}^{w_k}$ with depth greater than $|s|$.
It follows that the word $(y_{r_1}^{w_1}...y_{r_k}^{w_k}) y_s^{-t}$ is a weak standard form.

Now using Step $1$ in the proof of Theorem \ref{normalform}, we reduce the word $(y_{r_1}^{w_1}...y_{r_k}^{w_k}) y_s^{-t}$ to a word with no potential cancellations.
In this process, we obtain $$(y_{r_1}^{w_1}...y_{r_k}^{w_k}) f(y_{p_1}^{q_1}...y_{p_l}^{q_l})=f (y_{r_1\cdot f}^{w_1}...y_{r_k\cdot f}^{w_k})(y_{p_1}^{q_1}...y_{p_l}^{q_l})$$
such that:
\begin{enumerate}
\item $f(y_{p_1}^{q_1}...y_{p_l}^{q_l})$ is obtained from performing a sequence of expansion moves (in the sense of \ref{moves}) on $y_s^{-t}$. 
\item Each pair $y_{r_i\cdot f}^{w_i},y_{p_j}^{q_j}$ satisfies that either $r_i\cdot f=p_j, w_i=-q_j$ or the standard form $y_{r_i\cdot f}^{w_i}y_{p_j}^{q_j}$ does not contain a potential cancellation.
\end{enumerate} 

We claim that the words $$y_{r_1}^{w_1}...y_{r_k}^{w_k}\qquad y_{p_1\cdot f^{-1}}^{-q_1}...y_{p_l\cdot f^{-1}}^{-q_l}$$
witness the statement of the Lemma.
First, thanks to Lemma \ref{elementaryinversion} note that $$f^{-1}y_{p_1\cdot f^{-1}}^{-q_1}...y_{p_l\cdot f^{-1}}^{-q_l}$$ is obtained from $y_s^t$ using expansion moves (in the sense of \ref{moves}).
In particular, $$y_{p_1\cdot f^{-1}}^{-q_1}...y_{p_l\cdot f^{-1}}^{-q_l}$$
is obtained from $y_s^t$ using expansion moves (in the sense of \ref{movesdefinition}).

Next, assume by way of contradiction that there is a pair of elements $y_{r_i}^{w_i}, y_{p_j\cdot f^{-1}}^{-q_j}$ which are not equal and have a common offspring. 
Then by Lemma \ref{critcomoff} the word $y_{r_i}^{w_i} y_{p_j\cdot f^{-1}}^{q_j}$ has a potential cancellation.
Since $f$ acts on both $r_i$ and $p_j\cdot f^{-1}$, we get $$y_{r_i}^{w_i} y_{p_j\cdot f^{-1}}^{q_j}\cdot f=fy_{r_i\cdot f}^{w_i}y_{p_j}^{q_j}$$
Since $f$ is merely a prefix replacement map, the word $y_{r_i\cdot f}^{w_i}y_{p_j}^{q_j}$ has a potential cancellation.
This contradicts conclusion $(2)$ from the above.

Now the inductive step is essentially the same as the base case.
Given $y_{s_1}^{t_1}...y_{s_n}^{t_n}$ and $y_{u_1}^{v_1}...y_{u_m}^{v_m}$,
we perform the process for $y_{s_1}^{t_1}...y_{s_{n-1}}^{t_{n-1}}$ and $y_{u_1}^{v_1}...y_{u_m}^{v_m}$,
and then for $y_{s_n}^{t_n}$ and the resulting $Y$-word obtained from $y_{u_1}^{v_1}...y_{u_m}^{v_m}$ by expansions.
In the latter step, only the part of the word whose support intersects that of $y_{s_n}^{t_n}$ is expanded, and so the resulting offsprings from those expansions have disjoint support with $y_{s_1}^{t_1}...y_{s_{n-1}}^{t_{n-1}}$.
\end{proof}


\section{The complex $X$}
The goal of this section is to define the complex $X$ by adding higher dimensional cells to the $1$-skeleton $X^{(1)}$
defined in the previous section.
We define a class of finite subgraphs in $X^{(1)}$ that shall be the $1$-skeletons of clusters in $X$.
The set of all such graphs will be denoted throughout the paper as $\mathcal{H}$.
Higher cells will be then added to graphs in $\mathcal{H}$, and hence to $X^{(1)}$
in a well defined and $G$-equivariant manner.

Before we state the main definition of $\mathcal{H}$,
we state a few preliminary definitions.
This language will be useful in describing elements of $\mathcal{H}$.

\begin{defn}
Special forms $y_{s_1}^{t_1}...y_{s_n}^{t_n}$ and $y_{p_1}^{q_1}...y_{p_m}^{q_m}$
are said to be \emph{independent} if each pair $s_i,p_j$ is independent.
A list of special forms $\tau_1,...,\tau_n$ is said to be \emph{independent} if the special forms are pairwise independent.

A list $\tau_1,...,\tau_n$ is said to be \emph{sorted} if it satisfies the following:
\begin{enumerate}
\item $\tau_1,...,\tau_n$ are independent.
\item For each $i<j$, the pair $\tau_i,\tau_j$ satisfies that
if $y_s^t$ and $y_u^v$ are percolating elements that appear in $\tau_i$ and $\tau_j$ respectively, then $s<u$.
\end{enumerate}
Informally, the second condition here means that $\tau_1,...,\tau_n$ appear in the order \emph{left to right} as visualised in the infinite rooted binary tree. 
\end{defn}

The following is a basic observation that will be useful throughout the paper.
The proof of this is an elementary exercise concerning expansion moves in the sense of \ref{movesdefinition}, and the partial action of elements of $F$ on the set of finite binary sequences.

\begin{lem}\label{FactionspecialL}
Let $\tau_1,...,\tau_n$ be a sorted list of special forms. The following holds.
\begin{enumerate}
\item If $\tau_1',...,\tau_n'$ are special forms such that $\tau_i,\tau_i'$ are equivalent,
then $\tau_1',...,\tau_n'$ is also a sorted list.

\item Let $f\in F$ be such that for each percolating element $y_s^t$ in any $\tau_i$, $f$ acts on $s$.
Let $\tau_i'$ be the special form obtained from replacing each $y_s^t$ in $\tau_i$ with $y_{s\cdot f}^t$.
Then $\tau_1',...,\tau_n'$ is a sorted list.
\end{enumerate}
\end{lem}

Throughout the paper, we fix the following convention for a sorted list $\tau_1,...,\tau_n$.
Given $X\subseteq \{1,...,n\}$, we denote $$\tau_X=\prod_{i\in X}\tau_i=\tau_{j_1}\tau_{j_2}...\tau_{j_l}\qquad \text{ where }X=\{j_1,....,j_l\}\text{ and } j_1<j_2<...<j_l$$
So $\prod_{i\in X}\tau_i$ denotes a formal word obtained from concatenation of the $\tau_i$ for $i\in X$ in the increasing order of indices.
So for instance, if $X=\{1,3,4\}$, then $\tau_X=\tau_1\tau_3\tau_4$.  

Now we are ready to define the family $\mathcal{H}$.

\begin{defn}
An element $C$ of $\mathcal{H}$ is a subgraph of $X^{(1)}$ which is determined by the following:
\begin{enumerate}
\item A sorted list of special forms $\tau_1,...,\tau_n$.
\item A coset $F\tau$.
\end{enumerate}
Then $C$ is the induced subgraph of the vertex set
$$\{ F \tau_X \tau\mid X\subseteq \{1,...,n\}\}$$
in $X^{(1)}$.
The data in $(1)$ and $(2)$ above is said to be a \emph{description} of $C$.
We say that $C$ is \emph{described by} the \emph{base} $F\tau$
and \emph{parameters} $\tau_1,...,\tau_n$.
It will be often convenient to choose $\tau$ above as a $Y$-word,
and later in the paper we shall make this a convention for certain definitions.
\end{defn}

Observe that each closed $1$-cell in $X^{(1)}$ is an element of $\mathcal{H}$.
To see this, recall that each $1$-cell $F\tau_1, F\tau_2$ satisfies that $$F\tau_2\tau_1^{-1}=F\nu$$
where $\nu$ is a special form.
It follows that $F\tau_2=F\nu \tau_1$, and so the $1$-cell admits a description with base $F\tau_1$ and parameter $\nu$.

\subsection{Descriptions of graphs in $\mathcal{H}$ and the $G$-action.}

A graph in $\mathcal{H}$ can be described in many different ways.
For example, the closed $1$-cell $Fy_{10}, F$ can be described with base $F$ and parameter $y_{10}$,
or with base $Fy_{10}$ and parameter $y_{10}^{-1}$.
Since graphs in $\mathcal{H}$ shall emerge as $1$-skeletons of clusters in $X$,
we shall describe clusters in $X$ by describing their $1$-skeleton in $\mathcal{H}$.
Therefore, it shall be useful to understand different kinds of descriptions of graphs in $\mathcal{H}$.
Much of the technical difficulty in this article arises from dealing with non-uniqueness of the descriptions of clusters.
It is not clear how to choose a canonical representative that suits all situations.

Given a graph in $\mathcal{H}$,
we can find a description with any given node of the graph as base.
This follows immediately from the definitions and is captured in the following Lemma.

\begin{lem}\label{clusterdescription}
Let $C\in \mathcal{H}$ be described with base $F\tau$ and parameters $\tau_1,...,\tau_n$.
Then for any $X\subseteq \{1,...,n\}$, $C$ admits a description with base $F \tau_X \tau$, and parameters $\tau_1^{s_1},...,\tau_n^{s_n}$, 
where $s_i=-1$ if $i\in X$ and $s_i=1$ otherwise.
\end{lem}

Next, we observe that replacing parameters by equivalent special forms describe the same graph.

\begin{lem}\label{equivalentparametrisation}
Let $C\in \mathcal{H}$ be described with base coset $F\tau$ and parameters $\tau_1,...,\tau_n$.
Let $\tau_1',...,\tau_n'$ be special forms such that $\tau_i'$ is equivalent to $\tau_i$.
Then $C$ also admits the description with base coset $F\tau$ and parameters $\tau_1',...,\tau_n'$.
\end{lem} 

\begin{proof}
Since they are equivalent, it follows that $\tau_i'$ can be obtained from $\tau_i$ by a sequence of expansion and contraction moves in the sense of \ref{movesdefinition}.
In particular, from Lemma \ref{elemrelmoves} we know that there are $f_i\in F$ such that:
\begin{enumerate}
\item $\tau_i=f_i\tau_i'$ for $ 1\leq i\leq n$.
\item The support of $f_i$ is contained in the support of $\tau_i$, which equals the support of $\tau_i'$.
\end{enumerate}
This means that $f_i,\tau_j'$ commute if $i\neq j$.
It follows that for each $X\subseteq \{1,...,n\}$, 
we have $F\tau_X\tau= F\tau_X'\tau$. 
\end{proof}

A nice consequence of the above is that when describing the graph $C$,
we can conveniently choose parameters with desired features.
For instance, in some cases it may be useful to choose parameters with sufficiently large depth.

We end this section with the following useful remark about the choice of base coset.

\begin{remark}\label{basecosetyword}
Let $C$ be a cluster described with base $F\tau$ and parameters $\tau_1,...,\tau_n$.
The word $\tau$ may not be a $Y$-word, but we can find a different description for $C$ whose base coset representative is a $Y$-word.
This is done as follows.

First we convert the word $\tau$ into a standard form $f\nu_1$, where $f\in F$ and $\nu_1$ is a $Y$-word.
Let $\tau_i'$ be a special form obtained by applying expansion moves on $\tau_i$ such that for each $y_s^t$ that occurs in $\tau_i'$,
$f$ acts on $s$.
Let $\tau_i''$ be the special form obtained by replacing each percolating element 
$y_s^{t}$ in $\tau_i'$ by $y_{s\cdot f}^{t}$.
It follows from Lemma \ref{FactionspecialL} that $\tau_1'',...,\tau_n''$ is a sorted list.

So for each $X\subseteq \{1,...,n\}$ $$F\tau_X \tau= F\tau_X f\nu_1=F\tau_X''\nu_1$$
We conclude that $C$ admits a description with base at $F\nu_1$ and with parameters
$\tau_1'',...,\tau_n''$.
\end{remark}

\begin{lem}\label{clusteraction}
Let $C\in \mathcal{H}$.
For an element $\nu\in G$, $C\cdot \nu$ is also an element of $\mathcal{H}$.
\end{lem}

\begin{proof}
Let $C$ be described with base $F\tau$ and parameters $\tau_1,...,\tau_n$.
Then $C\cdot \nu$ is a cluster that admits a description with base $F(\tau\nu)$ and parameters $\tau_1,...,\tau_n$.
\end{proof}

Lemma \ref{clusteraction} is a simple observation, but has useful consequences.
For instance, while formulating our arguments 
we shall often assume that a given cluster is based at the trivial coset, using transitivity of the $G$-action on $X^{(0)}$.

\subsection{$\mathcal{H}$ is closed under intersection}\label{intersectionclus}

\begin{prop}\label{closedintersection}
Let $C_1,C_2\in \mathcal{H}$.
If $C_1\cap C_2$ is nonempty, then it is an element of $\mathcal{H}$.
\end{prop}

First we observe the following.
If $F\tau$ is a node in $C_1\cap C_2$, acting upon this by $\tau^{-1}$,
we obtain $$(C_1\cap C_2)\cdot \tau^{-1}=(C_1\cdot \tau^{-1})\cap (C_2\cdot \tau^{-1})$$
Thanks to Lemma \ref{clusteraction}, $C_1\cap C_2\in \mathcal{H}$ if and only if 
$(C_1\cdot \tau^{-1})\cap (C_2\cdot \tau^{-1})\in \mathcal{H}$.
The advantage of considering the latter is that the node given by the trivial coset $F$ lies in $(C_1\cdot \tau^{-1})\cap (C_2\cdot \tau^{-1})$.

Before we supply a full proof of Proposition \ref{closedintersection}, we describe conceptual ingredients that are needed in the proof.

\begin{defn}
Let $\tau_1,...,\tau_n$ and $\lambda_1,...,\lambda_m$ be sorted lists.
Such a pair of sorted lists is said to be \emph{compatible} if the following holds.
For each $X\subseteq \{1,...,n\},Y\subseteq \{1,...,m\}$
if $F\tau_X=F\lambda_Y$ then $\tau_X=\lambda_Y$ (the latter denotes equality as words).
\end{defn}

\begin{lem}\label{pairsorted}
Let $\tau_1,...,\tau_n$ and $\lambda_1,...,\lambda_m$ be sorted lists.
Then there is a compatible pair of sorted lists $\tau_1',...,\tau_n'$ and $\lambda_1',...,\lambda_m'$ such that
$\tau_i,\tau_i'$ and $\lambda_i,\lambda_i'$ are equivalent for each $1\leq i\leq n$.
\end{lem}

\begin{proof}
First we prove the following.

{\bf Claim}: Let $X_1,X_2\subseteq \{1,...,n\}$ and $Y_1,Y_2\subseteq \{1,...,m\}$ be such that
$$F\tau_{X_1}=F\lambda_{Y_1}\qquad F\tau_{X_2}=F\lambda_{Y_2}$$
Then it follows that $$F\tau_{X_1\cap X_2}=F\lambda_{Y_1\cap Y_2}\qquad F\tau_{X_2\setminus X_1}=F\lambda_{Y_2\setminus Y_1}\qquad F\tau_{X_1\setminus X_2}=F\lambda_{Y_1\setminus Y_2}$$

{\bf Proof of claim}: We can apply Lemma \ref{compatibleexpansionlemma} to the pairs
$$F\tau_{X_1}=F\lambda_{Y_1}\qquad F\tau_{X_2}=F\lambda_{Y_2}$$
to obtain equal words $$\tau_{X_1}'=\lambda_{Y_1}'\qquad \tau_{X_2}'=\lambda_{Y_2}'$$ where $\tau_{X_i}'$ is obtained by applying expansion moves (in the sense of Definition \ref{movesdefinition}) on $\tau_{X_i}$,
and $\lambda_{Y_i}'$ is obtained by applying expansion moves on $\lambda_{Y_i}$, for each $1\leq i\leq 2$.
Note that these words are obtained by iteratively applying expansion moves to percolating elements of $\tau_{X_1}, \lambda_{Y_1},\tau_{X_2}, \lambda_{Y_2}$, and their offspring.
It follows that there are subwords of $\tau_{X_i}'$ that are obtained from applying expansion moves on $\tau_{X_1\cap X_2}, \tau_{X_i\setminus X_j}$ respectively, for $1\leq i\leq 2$ and $j\in \{1,2\}\setminus \{i\}$.
The same for $\lambda_{Y_1}',\lambda_{Y_2}'$ and $\lambda_{Y_1\cap Y_2}, \lambda_{Y_1\setminus Y_2},\lambda_{Y_2\setminus Y_1}$.

Recall that words obtained from applying expansion moves (in the sense of Definition \ref{movesdefinition}) on a given $Y$-word $y_{s_1}^{t_1}...y_{s_n}^{t_n}$, where $s_1,...,s_n$ are independent and $t_1,...,t_n\in \{\pm 1\}$,
are in bijective correspondence with finite $3$-ary forests with $n$ roots. (Recall the discussion from subsection \ref{Omega}).
Hence we can perform expansion moves to obtain a common refinement for the subwords from the previous paragraph, so that the proof of the claim follows.

It follows that there are sets $X_1,...,X_{l}\subseteq \{1,...,n\}$ and $Y_1,...,Y_l\subseteq \{1,...,m\}$ that satisfy the following.
\begin{enumerate}
\item $X_1,...,X_{l}$ are pairwise disjoint and $Y_1,...,Y_l$ are pairwise disjoint.
\item $F\tau_{X_i}=F\lambda_{Y_i}$ for each $1\leq i\leq l$.
\item Whenever $F\tau_X=F\lambda_Y$, it holds that $X$ is a union of some sets in $\{X_1,...,X_l\}$ and $Y$ is a union of sets in $\{Y_1,...,Y_l\}$ with the same indices.
\end{enumerate} 
(Note that it may be the case that $l=0$, and so $F\tau_X\neq F\lambda_Y$ for any $X\subseteq \{1,...,n\}, Y\subseteq \{1,...,m\}$).
Our Lemma follows from an application of Lemma \ref{compatibleexpansionlemma} on the pairs $F\tau_{X_i}=F\lambda_{Y_i}$ for each $1\leq i\leq l$.
\end{proof}

\begin{defn}
Let $\tau_1=y_{s_1}^{t_1}...y_{s_n}^{t_n}$ and $\tau_2=y_{u_1}^{v_1}...y_{u_m}^{v_m}$ be special forms.
$\tau_1,\tau_2$ are said to be \emph{consecutive}, if $s_n,u_1$ are consecutive.
$\tau_1,\tau_2$ are said to be \emph{alternating} if they are consecutive, and also satisfy that $v_1=-t_n$. 
It follows from the definition that $\tau_1,\tau_2$ are alternating if and only if the product word $\tau_1\tau_2$ is a special form.
A list $\tau_1,...,\tau_n$ of special forms is said to be alternating if each pair $\tau_i,\tau_{i+1}$ is alternating for $1\leq i\leq n-1$.
It follows that $\tau_1,...,\tau_n$ is alternating if and only if the word $\tau_1...\tau_n$ is a special form.
\end{defn}

Consider $C\in \mathcal{H}$ described at base $F$ and with parameters $\tau_1,...,\tau_n$.
Observe that for a nonempty set $X\subseteq \{1,...,n\}$, there is an edge connecting $F$ and 
$F\tau_X$ if and only if $X=\{i,i+1,...,j\}$, for some $i,j$, and satisfies that $\tau_i,\tau_{i+1},...,\tau_{j}$
are alternating.
This motivates the following. 

\begin{defn}
Let $\tau_1,...,\tau_n$ be a sorted list of special forms.
A subset $$\{\tau_{i},\tau_{i+1},...,\tau_{j}\}\subseteq \{\tau_1,...,\tau_n\}$$
is said to be an \emph{alternating block} if $\tau_{i},\tau_{i+1},...,\tau_{j}$ are alternating.
We abuse notation and also refer to the subset of indices $\{i,i+1,...,j\}\subseteq \{1,...,n\}$ as an alternating block.
This special usage will be clear from the context.

Given the sorted list above, a \emph{block decomposition}
of $X$ is a partition formed by the sets 
$$\{1,...,l_1\}, \{l_1+1,l_1+2,...,l_2\},...,\{l_{k}+1,...,n\}$$
where each set is an alternating block and $1\leq l_1<...<l_k<n$.
Clearly, any sorted list admits a unique partition into \emph{maximal alternating blocks}.
This is called the \emph{maximal block decomposition} of the sorted list. 
\end{defn}

\begin{example}
Consider the sorted list $y_{010},y_{011}^{-1}, y_{10}^{-1},y_{110}$.
The maximal block decomposition of this is $$\{y_{010},y_{011}^{-1}\}, \{y_{10}^{-1},y_{110}\}$$

Note that $y_{011}^{-1}, y_{10}^{-1}$ are consecutive (since $011,10$ are consecutive binary sequences),
but they are not alternating.
\end{example}

\begin{lem}\label{block}
Let $\tau_1,...,\tau_n$ and $\lambda_1,...,\lambda_m$ be sorted lists of special normal forms such that
$\tau_{\{1,...,n\}}$ and $\lambda_{\{1,...,m\}}$ are equal as words.
Then for each maximal alternating block $U\subseteq \{1,...,n\}$ of the former,
there is a maximal alternating block $V\subseteq \{1,...,m\}$ of the latter
such that $\tau_U=\lambda_V$.
In other words, there is a natural bijection between the maximal block decompositions of $\tau_1,...,\tau_n$ and $\lambda_1,...,\lambda_m$.
\end{lem}

\begin{proof}
We read both words from left to right and partition them as we go along.
\end{proof}

By definition, the set of nodes of a graph in $\mathcal{H}$ naturally forms a Boolean algebra
in the sense of an algebra of subsets of a given finite set.
We recall the analogous notion for a collection of subsets.

\begin{defn}
Let ${\bf X}$ be a collection of subsets of the set $\{1,...,n\}$.
We say that ${\bf X}$ is \emph{Boolean}, if there are elements $X_1,...,X_k\in {\bf X}$ such that:
\begin{enumerate}
\item $X_1,...,X_k$ are pairwise disjoint.
\item ${\bf X}=\{\bigcup_{i\in Z} X_i\mid Z\subseteq \{1,...,k\}\}$.
\end{enumerate}
Under the partial ordering of inclusion, the sets $X_1,...,X_k$ are the minimal elements of ${\bf X}$.
The minimal elements of a Boolean collection will be referred to as the \emph{atoms} of the Boolean collection.
Any subcollection ${\bf X}\subseteq 2^{\{1,...,n\}}$ that is closed under taking unions and set difference
is a Boolean subcollection.
\end{defn}

\begin{lem}\label{compboolean}
Let $\tau_1,...,\tau_n$ and $\lambda_1,...,\lambda_m$ be a compatible pair of sorted lists.
Let $${\bf X}=\{X\subseteq \{1,...,n\}\mid  \tau_X=\lambda_Y\text{ for some } Y\subseteq \{1,...,m\}\}$$
$${\bf Y}=\{Y\subseteq \{1,...,m\}\mid  \tau_X=\lambda_Y\text{ for some } X\subseteq \{1,...,n\}\}$$
Then both ${\bf X},{\bf Y}$ are Boolean subcollections.
\end{lem}

\begin{proof}
We show this for ${\bf X}$, the proof for ${\bf Y}$ is the same.
In what follows below equality denotes equality as words.
Let $X_1,X_2\subseteq {\bf X}$.
Then there are $Y_1,Y_2\in {\bf Y}$ such that $\tau_{X_1}=\lambda_{Y_1}$ and $\tau_{X_2}=\lambda_{Y_2}$.
It follows that $\tau_{X_1\cup X_2}=\lambda_{Y_1\cup Y_2}$ and $\tau_{X_1\setminus X_2}=\lambda_{Y_1\setminus Y_2}$.
Since ${\bf X}$ is closed under unions and set difference, we are done.
\end{proof}

Next we demonstrate that subgraphs of a given graph $C\in \mathcal{H}$ that are also elements of $\mathcal{H}$
correspond precisely to certain Boolean subcollections.

\begin{lem}\label{boolean}
Let $C\in \mathcal{H}$ and let $C'\subset C$ be a subgraph that is also an element of $\mathcal{H}$.
Then $C$ admits a description with bases at $F\tau$ and parameters $\tau_1,...,\tau_n$ such that
the following holds.
There is a Boolean collection ${\bf X}\subset 2^{\{1,...,n\}}$ such that $C'$ admits the following description.
\begin{enumerate}
\item It is based at $F\tau$.
\item The elements of the set $$\{\tau_X\mid X\text{ is an atom of }{\bf X}\}$$
are all special forms that are the parameters.
\end{enumerate}
\end{lem}

\begin{proof}
By the $G$-action, we can assume that the trivial coset $F$ lies in $C\cap C'$.
Using Lemma \ref{pairsorted}, we find a compatible pair of sorted lists   $\tau_1,...,\tau_n$ and $\nu_1,...,\nu_m$
that are parameters for descriptions of $C,C'$ respectively with both bases at $F$.
Let $${\bf X}=\{X\subseteq \{1,...,n\}\mid F\tau_X= F\lambda_Y\text{ for some }Y\subseteq \{1,...,m\}\}$$
Since the sorted lists are compatible, $$F\tau_X=F\lambda_Y\implies \tau_X=\lambda_Y$$ and hence $${\bf X}=\{X\subseteq \{1,...,n\}\mid \tau_X= \lambda_Y\text{ for some }Y\subseteq \{1,...,m\}\}$$
Thanks to Lemma \ref{compboolean}, ${\bf X}$ is a Boolean subcollection.
In particular, $$\{\tau_X\mid X\text{ is an atom of }{\bf X}\}$$
consists of special forms that are the parameters of $C'$ with base $F$.
\end{proof}

\begin{remark}
The following converse of the above Lemma follows immediately from the definitions.
Let $C$ be the cluster described with base $F\tau$ and parameters $\tau_1,...,\tau_n$.
Consider a Boolean collection ${\bf X}\subset 2^{\{1,...,n\}}$
which satisfies that the elements of $$\{\tau_X\mid X\text{ is an atom of }{\bf X}\}$$
are all special forms.
Then there is a subgraph $C'\subset C$ which lies in $\mathcal{H}$, and is described with base $F\tau$ and parameters in the set above.
\end{remark}

{\bf Proof of Proposition \ref{closedintersection}}

\begin{proof}
Let $C_1,C_2$ be our clusters in $\mathcal{H}$.
Thanks to Lemmas \ref{equivalentparametrisation} and \ref{pairsorted}, we can assume that $C_1,C_2$ are both described with base at the trivial coset $F$ and parametrized by a compatible pair of sorted lists
$\tau_1,...,\tau_n$ and $\nu_1,...,\nu_m$ respectively.
Let $${\bf X}=\{X\subseteq \{1,...,n\}\mid F\tau_X\in C_1\cap C_2\}$$ $$=\{X\subseteq \{1,...,n\}\mid F\tau_X=F\nu_Y\text{ for some }Y\subseteq \{1,...,m\}\}$$
$$=\{X\subseteq \{1,...,n\}\mid \tau_X=\nu_Y\text{ for some }Y\subseteq \{1,...,m\}\}$$
Thanks to Lemma \ref{compboolean}, ${\bf X}$ is a Boolean subcollection.

We will show that each atom $U$ of ${\bf X}$ satisfies that 
$U$ is an alternating block, i.e. $\tau_U$ is a special form.
Let $V\subseteq \{1,...,m\}$ be such that 
$\tau_U=\nu_V$.
Let $U=\{u_1,...,u_k\}$ and $V=\{v_1,...,v_l\}$ be the elements listed in increasing order.
Applying Lemma \ref{block} to the sorted lists (of normal special forms)
$$\tau_{u_1},\tau_{u_2}...,\tau_{u_k}\qquad \nu_{v_1},\nu_{v_2},...,\nu_{v_l}$$
we conclude that any maximal alternating block $U_1\subseteq U$ corresponds to a maximal alternating block
$V_1\subseteq V$ satisfying equality of words $\tau_{U_1}=\nu_{V_1}$.
But this means that $F\tau_{U_1}=F\nu_{V_1}$ and so $F\tau_{U_1}\in C_1\cap C_2$.
Therefore $U_1\in {\bf X}$.
By minimality of $U$ it follows that $U_1=U$ and so
$U$ is an alternating block.

This means that $C_1\cap C_2$ is described with base $F$ and parameters in the set
$$\{\tau_U\mid U\text{ is an atom of }{\bf X}\}$$
and hence is an element of $\mathcal{H}$ as desired.
\end{proof}

\subsection{Graphs in $\mathcal{H}$ as $1$-skeletons of clusters}

We demonstrate that an element $C\in \mathcal{H}$ is the $1$-skeleton of a cluster in the sense of Definition \ref{cluster}.
This identification emerges in a natural way by considering descriptions of graphs in $\mathcal{H}$ that we call \emph{proper descriptions}.
A proper description for $C$ involves a certain elementary condition on the parameters which we describe below.

\begin{defn}
Let $\tau_1,...,\tau_n$ be a sorted list of special forms.
We say that the sorted list is \emph{proper}, if whenever a pair $\tau_i,\tau_{i+1}$ is consecutive,
then it is also alternating.
\end{defn}

Observe that given any pair of consecutive special forms $\tau_1,\tau_2$,
exactly one of $\tau_1,\tau_2^{-1}$ or $\tau_1,\tau_2$ is alternating.
(Recall the convention that $\tau^{-1}=y_{s_1}^{-t_1}...y_{s_n}^{-t_n}$ whenever $\tau=y_{s_1}^{t_1}...y_{s_n}^{t_n}$.)
Indeed, given any sorted list of special forms $\tau_1,...,\tau_n$,
there are $s_1,...,s_n\in \{1,-1\}$ such that
$\tau_1^{s_1},...,\tau_n^{s_n}$ is proper.
(To see this, perform an elementary induction on $n$.)

\begin{lem}\label{proper}
Any element $C\in \mathcal{H}$ admits a description where the parameters are proper.
Moreover, we can also arrange it so that the parameters are all normal forms.
\end{lem}

\begin{proof}
Let $C$ be described with base $F\tau$ where $\tau$ is a $Y$-word, and parameters $\tau_1,...,\tau_n$.
Let $s_1,...,s_n\in \{1,-1\}$ be such that $\tau_1^{s_1},...,\tau_n^{s_n}$ is proper.
Then $C$ admits the following description:
\begin{enumerate}
\item Base $F\tau_X \tau$, where $X=\{i\in \{1,...,n\}\mid s_i=-1\}$.
\item Parameters $\tau_1^{s_1},...,\tau_n^{s_n}$.
\end{enumerate}
Finally, by applying a sequence of contraction moves to $\tau_1^{s_1},...,\tau_n^{s_n}$,
we obtain a proper list of special normal forms.
To see this, one simply checks the following.
If $\nu_1',\nu_2'$ are obtained from $\nu_1,\nu_2$ respectively by applying contraction and/or expansion moves,
then $\nu_1',\nu_2'$ are alternating if and only if $\nu_1,\nu_2$ are alternating.
\end{proof}

\begin{defn}
We call a description of $C\in \mathcal{H}$ with proper parameters as a \emph{proper description}.
\end{defn}

The advantage of a proper description for $C$ is that the edge relation between nodes of $C$ is easier to characterise.

\begin{lem}\label{edgerel}
Let $C\in \mathcal{H}$ be described with base $F\tau$ where $\tau$ is a $Y$-word, and proper parameters $\tau_1,...,\tau_n$.
Then the nodes $F\tau_X\tau,F\tau_Y\tau$ for $X,Y\subseteq \{1,...,n\}$ form an edge if and only if one of the following holds:
\begin{enumerate}
\item $X\subset Y$ and $\tau_{Y\setminus X}$ is a special form. 
\item $Y\subset X$ and $\tau_{X\setminus Y}$ is a special form.
\end{enumerate} 
\end{lem}

\begin{proof}
If $F\tau_X\tau,F\tau_Y\tau$ form an edge, this means that
$$\tau_X\tau_Y^{-1}=\prod_{i\in X\setminus Y} \tau_i \prod_{j\in Y\setminus X}\tau_j^{-1}$$
is a special form.
If both $X\setminus Y$ and $Y\setminus X$ are nonempty, then there must be an alternating pair of the form $\tau_i,\tau_{i+1}^{-1}$
or $\tau_i^{-1},\tau_{i+1}$ in the word above.
In either case, this contradicts the assumption that the parameters are proper. 
The converse follows immediately from the definition of the edge relation in $X^{(1)}$.
\end{proof}

Now we are ready to provide the identification between elements of $\mathcal{H}$ and $1$-skeletons of clusters.

\begin{prop}\label{identification}
Let $C\in \mathcal{H}$ be described by base $F\tau$ where $\tau$ is a $Y$-word, and proper parameters $\tau_1,...,\tau_n$. 
Then $C$ is isomorphic to the $1$-skeleton of an $n$-cluster in the sense of Definition \ref{cluster}.
\end{prop}

\begin{proof}
Without loss of generality (thanks to Lemma \ref{clusteraction}) 
we can assume that $C$ is based at the trivial coset.
 Define $${\bf Y}=\{i\in \{1,...,n-1\}\mid \tau_i\tau_{i+1}\text{ is a special form}\}$$
Now define the hyperplane arrangement $\mathcal{A}$ in $\mathbf{R}^n$ consisting of the following hyperplanes:
\begin{enumerate}
\item $x_i=0, x_i=1$, $1\leq i\leq n$.
\item $x_i=x_{i+1}$ for $i\in {\bf Y}$.
\end{enumerate}
Clearly, $\mathcal{A}$ is admissible.
We claim that the $1$-skeleton of $\mathcal{C(A)}$ is isomorphic to $C$.

The $0$-cells of $\mathcal{C(A)}$ are naturally in bijection with subsets of $\{1,...,n\}$
as follows.
A subset $X\subseteq \{1,...,n\}$ corresponds to the $0$-cell
given by the coordinates $(y_1,...,y_n)$ where $y_i=0$ if $i\notin X$ and $y_i=1$ if $i\in X$.

Observe that the $1$-cells of $\mathcal{C(A)}$ are in a natural bijection with 
partitions of $\{1,...,n\}$ into three sets $X_1,X_2,X_3$ such that
 $X_2=\{j,j+1,...,k\}\subset \{1,...,n\}$ satisfies that for each $j\leq i\leq k-1$, $i\in {\bf Y}$.
The $1$-cell corresponding to such a pair is the affine line segment connecting the $0$-cells
$(y_1,...,y_n)$ and $(z_1,...,z_n)$ given by the following coordinates:
$$(y_1,...,y_n)\qquad y_i=0\text{ if }i\in X_1\cup X_2\text{ and } y_i=1\text{ if }i\in X_3$$
$$(z_1,...,z_n)\qquad z_i=0\text{ if }i\in X_1\text{ and } z_i=1\text{ if }i\in X_2\cup X_3$$
This is precisely the intersection of the following collection of hyperplanes with $[0,1]^n$:
$$\{ \{x_i=0\}\mid i\in X_1\}\cup \{\{x_i=1\}\mid i\in X_3\}\cup \{\{x_i=x_j\}\mid i,j\in X_2\}$$

We now define the graph isomorphism between $\mathcal{C(A)}^{(1)}$ and $C$.
First, we identify the $0$-cell given by $X\subseteq \{1,...,n\}$ in $\mathcal{C(A)}$ with
the $0$-cell $F\tau_X$ of $C$.
Then we identify the $1$-cell of $\mathcal{C(A)}$ given by a triple $X_1,X_2,X_3$ (satisfying the above)
with the $1$-cell connecting the pair $F\tau_{X_3},F\tau_{X_2\cup X_3}$.
This identification provides the isomorphism between $\mathcal{C(A)}^{(1)}$ and $C$.
\end{proof}

We now show that the identification above is independent of the choice of proper description.

\begin{lem}\label{choiceproper}
Let $C\in \mathcal{H}$ and let $\mathcal{C(A)}$ be the cluster identified with a proper description of $C$ in Proposition \ref{identification}.
Then this identification is the same for any other choice of proper parameters describing $C$.
\end{lem}

\begin{proof}
By $G$-equivariance, we can assume the the base of the given description is the trivial coset,
with proper parameters $\tau_1,...,\tau_n$.
Let $A_1,...,A_k\subseteq \{1,...,n\}$ be the maximal block decomposition for $\{\tau_1,...,\tau_n\}$.

With Lemmas \ref{clusterdescription} and \ref{equivalentparametrisation} in mind, 
observe that any proper description of $C$ (up to replacing equivalent special forms as parameters) is determined by a set of numbers $s_1,...,s_n\in \{1,-1\}$
such that:
\begin{enumerate}
\item $s_i=s_j$ whenever $i,j\in A_l$ for some $1\leq l\leq k$.
\item The base coset is $F\tau_X$ where $X=\{1\leq i\leq n\mid s_i=-1\}$.
\item The parameters are $\tau_1^{s_1},...,\tau_n^{s_n}$. 
\end{enumerate}

Let $\mathcal{A}_1$ be the hyperplane arrangement in $\mathbf{R}^n$ given by:

\begin{enumerate}
\item $\{x_i=0\}, \{x_i=1\}$, $1\leq i\leq n$.
\item $\{x_i=x_{i+1}\}$ for $i\in \{1\leq j\leq n-1\mid \tau_j\tau_{j+1}\text{ is a special form}\}$.
\end{enumerate}

Let $\mathcal{A}_2$ be the hyperplane arrangement in $\mathbf{R}^n$ given by:

\begin{enumerate}
\item $\{x_i=0\}, \{x_i=1\}$, $1\leq i\leq n$.
\item $\{x_i=x_{i+1}\}$ for $i\in \{1\leq j\leq n-1\mid \tau_j^{s_j}\tau_{j+1}^{s_{j+1}}\text{ is a special form}\}$.
\end{enumerate}

It follows that both arrangements are the same and hence describe the same cluster.
Considering symmetries of this $n$-cluster, we are done.
\end{proof}

\subsection{Higher dimensional cells in $X$}

Proposition \ref{identification} provides an identification of an element $C\in \mathcal{H}$
with the $1$-skeleton of an $n$-cluster $\mathcal{C(A)}$.

\begin{defn}
Given $C\in \mathcal{H}$, let $\mathcal{C(A)}$ be the $n$-cluster whose $1$-skeleton
is identified with $C$ in Proposition \ref{identification}. 
A \emph{filling} $\mathbf{C}$ of $C$ is obtained by adding the higher cells of $\mathcal{C(A)}$ to $C$ in $X^{(1)}$.
The complex $X$ is obtained from $X^{(1)}$ by ``filling" each graph in $\mathcal{H}$.
In other words, $$X=\bigcup_{C\in \mathcal{H}}\mathbf{C}$$
\end{defn}

We show that the fillings are well defined and indeed produce a CW complex $X$
upon which $G$ acts.

\begin{lem}\label{fillingequivariant}
The fillings of $\mathcal{H}$ are $G$-equivariant. 
In other words, for each $C\in \mathcal{H},\nu \in G$, if $C_1=C\cdot \nu$ then $\mathbf{C}\cdot \nu=\mathbf{C_1}$ 
In particular, the $G$-action on $X^{(1)}$ naturally extends to the $G$-action on $X$.
\end{lem}

\begin{proof}
Let $C$ be described with base $F\tau$, where $\tau$ is $Y$-word, and proper parameters $\tau_1,...,\tau_n$. 
Let $\nu\in G$.
Then $C\cdot \nu$ admits a description with base $F(\tau\nu)$ and proper parameters $\tau_1,...,\tau_n$.
The hyperplane arrangements $\mathcal{A}_1$ for $C$ and $\mathcal{A}_2$ for $C\cdot \nu$
as described in the proof of Proposition \ref{identification} are both:
\begin{enumerate}
\item $\{x_i=0\}, \{x_i=1\}$, $1\leq i\leq n$.
\item $\{x_i=x_{i+1}\}$ for $i\in \{1\leq j \leq n-1\mid \tau_j\tau_{j+1}\text{ is a special form}\}$.
\end{enumerate}
Therefore, the group action preserves the identification.
\end{proof}

We end the section by demonstrating that $X$ is indeed a CW complex.
This follows immediately once we combine Proposition \ref{closedintersection} with the following.

\begin{prop}\label{subclusterfilling}
Let $C_1,C_2\in \mathcal{H}$ such that $C_1\subset C_2$.
Then $\mathbf{C_1}$ agrees with the subcomplex of $\mathbf{C_2}$ determined by $C_1\subset \mathbf{C_2}$.
\end{prop}

\begin{proof}
By considering the $G$-action, we assume without loss of generality that $C_1,C_2$ both contain the trivial coset as a node.
Let $C_2$ be described with base $F$ and parameters $\tau_1,...,\tau_n$.
By Lemma \ref{boolean}, there is a Boolean collection ${\bf X}\subset 2^{\{1,...,n\}}$
such that $C_1$ admits the following description.
\begin{enumerate}
\item Base at $F$.
\item The elements of the set $\{\tau_X\mid X\text{ is an atom of }{\bf X}\}$
are all special forms that are the parameters.
\end{enumerate}

Let $X_1,...,X_k\subset \{1,...,n\}$ be the minimal elements of this Boolean collection appearing in increasing order.
If $i,i+1\in X_j$ for some $1\leq j\leq k$, then $\tau_i,\tau_{i+1}$
are alternating.
Using this fact, it is easy to find numbers $s_1,...,s_n$ such that:
\begin{enumerate}
\item $\tau_1^{s_1},...,\tau_n^{s_n}$ is proper.
\item $s_i=s_{i+1}$ if $i,i+1\in X_j$ for some $1\leq j\leq k$.
\end{enumerate}

In particular, the list $\tau_{X_1}^{t_1},...,\tau_{X_k}^{t_k}$ is proper,
where $t_j$ equals $s_l$ for any $l\in X_j$.
This provides proper descriptions of $C_1,C_2$ with base $F\tau_Y$ where $Y=\{i\in \{1,...,n\}\mid s_i=-1\}$,
and proper parameters $\tau_1^{s_1},...,\tau_n^{s_n}$ and $\tau_{X_1}^{t_1},...,\tau_{X_k}^{t_k}$ respectively.

Let $\mathcal{A}_1$ be the arrangement of hyperplanes:
\begin{enumerate}
\item $\{x_i=0\},\{x_i=1\}$, $1\leq i\leq n$.
\item $\{x_i=x_{i+1}\}$ whenever $\tau_i\tau_{i+1}$ is a special form.
\end{enumerate}

Now define a $k$-dimensional flat $U$ of $\mathcal{A}_1$ as the intersection of the following hyperplanes.
\begin{enumerate}
\item $\{x_i=0\}$ for $i\notin \bigcup_{1\leq i\leq k}X_k$.
\item $\{x_i=x_{i+1}\}$ with $i$ for which there is a $1\leq j\leq k$ such that $i,i+1\in X_j$.
\end{enumerate}

Let $\mathcal{A}_2$ be the arrangement of hyperplanes obtained from restricting the arrangement $\mathcal{A}_1$
to the flat $U$.
$\mathcal{A}_2$ is naturally isomorphic to the arrangement
in $\mathbf{R}^k$ given by:
\begin{enumerate}
\item $\{x_i=0\},\{x_i=1\}$, $1\leq i\leq k$.
\item $\{x_i=x_{i+1}\}$ for $i$ such that $\tau_{X_i}\tau_{X_{i+1}}$ is a special form.
\end{enumerate}  
This proves that the restriction of $\mathbf{C_2}$ to $C_1$ is indeed equal to $\mathbf{C_1}$.
\end{proof}

\section{The action of $G$ on $X$}

For the rest of the article, we shall identify elements of $\mathcal{H}$ with their fillings,
and denote these as clusters.
So given $C\in \mathcal{H}$, we shall simply use $C$ to denote the filled cluster.
In this section we study the action of $G$ on $X$.
We prove the following:

\begin{prop}\label{stabilizer}
The stabilizer $Stab_G(e)$ of any cell $e$ of $X$ is a group of type $\mathbf{F}_{\infty}$.
More particularly, it is a finite product of copies of the groups $F_3$ and $F$.
\end{prop}

\begin{prop}\label{quotient}
The quotient $X/G$ has finitely many cells in each dimension.
\end{prop}

\subsection{stabilizers of cells}

Given any $n$-cell $e$ in $X$, by the construction of $X$ there is an $n$-cluster $C\in \mathcal{H}$
that contains $e$.
Let $C$ be described with base $F\tau$, where $\tau$ is a $Y$-word and proper parameters $\tau_1,...,\tau_n$.
Then the cluster $C_1=C\cdot \tau^{-1}$ admits a description with base $F$ and the same parameters $\tau_1,...,\tau_n$ as $C$.
It follows that there is an $n$-cell $e'$ of $C_1$ such that $e'=e\cdot \tau^{-1}$, and hence $Stab_G(e)\cong Stab_G(e')$.
So it suffices to understand the stabilizer of $e'$.
For the rest of this subsection we fix the description of $C_1$ and $e'$.
First, we make an elementary observation.

\begin{lem}\label{stab1}
The closure of $e'$ contains $F$ and $F\tau_{\{1,...,n\}}$.
\end{lem}

\begin{proof}
Since the parameters are proper, the cluster $C_1\in \mathcal{H}$ is isomorphic to the cluster (in the sense of definition \ref{cluster})
described by the following hyperplanes in $\mathbf{R}^n$.
\begin{enumerate}
\item $\{x_i=1\}, \{x_i=0\}$, $1\leq i\leq n$.
\item $\{x_i=x_{i+1}\}$ for $i$ such that $\tau_i\tau_{i+1}$ is a special form.
\end{enumerate}
The closure of an $n$-cell is an intersection of certain half spaces of the form $$\{x_i\geq 0\}\qquad \{x_i\leq 1\}\qquad \{x_i\leq x_{i+1}\}\qquad \{x_i\geq x_{i+1}\}$$
All these half-spaces contain the points $(0,...,0)$ and $(1,...,1)$,
which correspond to the nodes $F$ and $F\tau_{\{1,...,n\}}$ of $C_1$.
\end{proof}

An important step in understanding the stabilizer $Stab_G(e')$ is to demonstrate that $Stab_F(e')=Stab_G(e')$.
The stabilizer $Stab_F(e')$ is much easier to understand, since we already have developed an understanding of the action of $F$ on
$\Omega$.
In order to show that $Stab_F(e')=Stab_G(e')$, we first need to prove the following generalisation of Proposition \ref{stabilizerspecialform}.

\begin{lem}\label{blockstab}
Let $\tau_1,...,\tau_n$ be a proper list of special forms as above,
and let $X_1,...,X_k\subset \{1,...,n\}$ be pairwise disjoint sets.
Then 
$$\bigcap_{1\leq i\leq k}Stab_F(F\tau_{X_i})$$ is a finite direct product of copies of $F$ and $F_3$.
\end{lem}

\begin{proof}
For a given $X_j$, let $Y_1,...,Y_{l}$ be the maximal block decomposition of $X_j$.
First we claim that $$Stab_F(F\tau_{X_j})=\bigcap_{1\leq i\leq l}Stab_F(F\tau_{Y_{i}})$$
It is clear that the latter is a subset of the former.
So we need to show that if $f\in Stab_F(F\tau_{X_j})$, then $f\in Stab_F(F\tau_{Y_i})$ for each $1\leq i\leq l$.

We show this for the case $l=2$.
The general case follows from a straightforward induction that uses the same idea as in the base case.
For notational simplicity, we denote $\tau_{Y_1}=\psi_1$ and $\tau_{Y_2}=\psi_2$.
Let $\psi_1',\psi_2'$ be special forms obtained by applying expansion moves on $\psi_1,\psi_2$ such that 
for each percolating element $y_s^t$ of $\psi_1$ or $\psi_2$, $f$ acts on $s$.
Let $\psi_1'',\psi_2''$ be special forms obtained by 
replacing each percolating element $y_s^t$ by $y_{s\cdot f}^t$.
Since $$F\psi_1\psi_2\cdot f=F\psi_1''\psi_2''=F\psi_1\psi_2$$
we can apply a sequence of contraction and expansion moves on $\psi_1''\psi_2''$ to obtain $\psi_1\psi_2$.

Now since $Y_{1},Y_2$ were assumed to be maximal alternating blocks, we know that $\psi_1,\psi_2$ are not alternating.
Therefore, it must be the case that we can apply a sequence of contraction and expansion moves to $\psi_1''$ to obtain $\psi_1$ 
and a sequence of contraction and expansion moves on $\psi_2''$ to obtain $\psi_2$.
It follows that $$f\in Stab_F(F\psi_1)\cap Stab_F(F\psi_2)$$
So our claim holds.

Now we apply the above claim to each $X_i$ in the sets $X_1,...,X_k$ from the statement of the Lemma. 
Reasoning along the lines of Proposition \ref{stabilizerspecialform}, we obtain that the restriction of the element $f$
on the support of each maximal alternating block of each $X_i$ must be a tree diagram in natural correspondence with a tree diagram of an element of $F_3$.
(We do not recall the notation and details of this here, as they have been spelled out in subsection \ref{Omega}.)
In this way, we obtain the required direct product decomposition.
\end{proof}

For our cell $e'$, we define $${\bf Y}=\{Y\subseteq \{1,...,m\}\mid F\tau_Y\text{ lies in the closure of } e'\}$$
Let ${\bf X}$ be the Boolean subcollection obtained by taking the closure of ${\bf Y}$ in $2^{\{1,...,m\}}$ under unions and set difference.
Further, let $X_1,...,X_k$ be the atoms of ${\bf X}$.
For the remainder of this subsection, we fix $e'$, ${\bf Y}, {\bf X}$ and $X_1,...,X_k$ as above.

\begin{lem}\label{booleanstab}
$$\bigcap_{Y\in {\bf Y}}Stab_F(F\tau_Y)=\bigcap_{1\leq i\leq k}Stab_F(F\tau_{X_i})$$
In particular, $$\bigcap_{Y\in {\bf Y}}Stab_F(F\tau_Y)$$ is a finite product of copies of $F$ and $F_3$, and is of type $\mathbf{F}_{\infty}$.
\end{lem}

\begin{proof}
Let $U,V\subseteq \{1,...,m\}$.
We claim that $$Stab_F(F\tau_U)\cap Stab_F(F\tau_V)$$ $$=Stab_F(F\tau_{U\setminus V})\cap Stab_F(F\tau_{V\setminus U})\cap Stab_F(F\tau_{U\cap V})$$

It is obvious that the latter is contained in the former.
It suffices to show the reverse containment. 

Let $\tau_i'$ be a special form obtained by performing expansion moves on $\tau_i$
such that for each percolating element $y_s^t$ of $\tau_i'$, $f$ acts on $s$.
Now let $\tau_i''$ denote special form obtained by replacing each percolating element $y_s^t$ of $\tau_i'$ with $y_{s\cdot f}^t$.
It follows from Lemma \ref{FactionspecialL} that $\tau_1'',...,\tau_n''$ is a sorted list of special forms.

It follows from our assumption that
$$F\tau_U\cdot f=F\tau_U''\qquad F\tau_V\cdot f=F\tau_V''$$

By our assumption the words $\tau_U'',\tau_V''$ are both products of independent special forms, and are equivalent to $\tau_U,\tau_V$ respectively.
This means that we can perform a sequence of contraction and expansion moves on $\tau_{U} '',\tau_V''$ to obtain $\tau_U,\tau_V$
respectively.
So indeed upon performing contraction and expansion moves on the subwords $\tau_{U\cap V}'', \tau_{U\setminus V}'',\tau_{V\setminus U}''$ we must obtain
$\tau_{U\cap V}, \tau_{U\setminus V},\tau_{V\setminus U}$ respectively.
In particular, it follows that $f$ fixes the cosets $F\tau_{U\cap V}, F\tau_{U\setminus V},F\tau_{V\setminus U}$.
This proves our assertion.

Using this idea, we conclude the proof of the main assertion
$$\bigcap_{X\in {\bf X}}Stab_F(F\tau_X)=\bigcap_{1\leq i\leq k}Stab_F(F\tau_{X_i})$$
using a standard induction argument of the Boolean algebra of sets ${\bf X}$ obtained from ${\bf Y}$ by taking the closure under unions and set difference.
\end{proof}

\begin{lem}\label{shortlem}
Given $\nu\in Stab_G(e')$, for each $X\in {\bf X}$ it holds that 
$$Supp(\tau_X)\cdot \nu=Supp(\tau_X)$$
\end{lem}

\begin{proof}
Given $\nu\in Stab_G(e')$, $\nu$ induces a permutation of the $0$-cells $\{F\tau_Y\mid Y\in {\bf Y}\}$.
In particular, a power $\nu^k$ fixes each $0$-cell in this set.
In particular, $\nu^k\in F$ since it fixes the trivial coset.

Since $\nu^k$ lies in $$\bigcap_{Y\in {\bf Y}}Stab_F(F\tau_Y)$$
by Lemma \ref{booleanstab} it also lies in
$$\bigcap_{X\in {\bf X}}Stab_F(F\tau_X)$$
This implies that for each $X\in {\bf X}$, $$Supp(\tau_X)\cdot \nu^k=Supp(\tau_X)$$
Since $\nu$ is a $k$-th root of $\nu$ in $Homeo_+(\mathbf{R})$,
it follows that $$Supp(\tau_X)\cdot \nu=Supp(\tau_X)$$
\end{proof}

\begin{lem}\label{StabinF}
$Stab_G(e')=Stab_F(e')=\bigcap_{1\leq i\leq k}Stab_F(F\tau_{X_i})$
\end{lem}

\begin{proof}
Let $\nu\in Stab_G(e')$.
Since $\nu$ permutes the $0$-cells incident to $e'$,
it follows that $F\cdot \nu=F\tau_X$ for some $X\in {\bf Y}$.
In particular, it holds that there is an $f\in F$ such that $\nu=f\tau_X$.
Now assume by way of contradiction that $X$ is nonempty.

By Lemma \ref{shortlem} $Supp(\tau_X)\cdot \nu=Supp(\tau_X)$.
It follows that $Supp(\tau_X)\cdot f=Supp(\tau_X)$.

Now let $\tau_X=y_{s_1}^{t_1}...y_{s_n}^{t_n}$.
Let $y_{u_1}^{v_1}...y_{u_m}^{v_m}$ be a special form obtained by applying expansion moves on $y_{s_1}^{t_1}...y_{s_n}^{t_n}$
such that:
\begin{enumerate}
\item $f$ acts on each $u_i$ for $1\leq i\leq m$.
\item $|u_i\cdot f|>|s_j|$ for each $1\leq i\leq m, 1\leq j\leq n$.
\end{enumerate}

Since the interval $Supp(\tau_X)$ is invariant under the action of $f$, it follows that $u_1\cdot f=s_10^l$ for some $l\in \mathbf{N}$.
Also, note that $v_1=t_1$.

So we obtain that $$F\tau_X\cdot (f\tau_X)=F(y_{s_1}^{t_1}...y_{s_n}^{t_n}) (f y_{s_1}^{t_1}...y_{s_n}^{t_n})$$
$$=F(y_{u_1}^{v_1}...y_{u_m}^{v_m})(fy_{s_1}^{t_1}...y_{s_n}^{t_n})$$
$$=F(y_{u_1\cdot f}^{v_1}...y_{u_m\cdot f}^{v_m})(y_{s_1}^{t_1}...y_{s_n}^{t_n})$$

Since $Supp(\tau_X)\cdot f=Supp(\tau_X)$, for the infinite binary sequence $s_10^{\infty}$ the associated calculation $\Lambda$ of the standard form
$$(y_{u_1\cdot f}^{v_1}...y_{u_m\cdot f}^{v_m})(y_{s_1}^{t_1}...y_{s_n}^{t_n})$$
contains two occurrences of $y^{t_1}=y^{v_1}$.
In particular, there is no potential cancellation and the exponent of this calculation equals $2$.
Since $F\tau_X\cdot (f\tau_X)$ is equal to a coset of the form $F\tau_U$ for some $U\in {\bf Y}$,
this is a contradiction.
Therefore, $X$ must be the empty set, and so $\nu\in F$.

To finish the proof, we must show the second equality in the statement of our Lemma.
We already know from Lemma \ref{shortlem} that $Supp(\tau_X)\cdot \nu=Supp(\tau_X)$ for each $X\in {\bf X}$.

For any $Y\in {\bf Y}$, $F\tau_Y\cdot \nu = F\tau_{Y'}$ where $Y'\in {\bf Y}$.
Now $Y$ is a union of minimal elements of ${\bf X}$. 
Since $\nu$ fixes the support of each $\tau_X$ for such a minimal $X$,
it must be the case that $F\tau_X\cdot \nu=F\tau_X$ and hence $F\tau_Y\cdot \nu =F\tau_Y$.
This means that $\nu$ fixes each $0$-cell in the closure of $e'$.
We conclude by applying Lemma \ref{booleanstab} to obtain
$$\bigcap_{Y\in {\bf Y}}Stab_F(F\tau_Y)=\bigcap_{1\leq i\leq k}Stab_F(F\tau_{X_i})$$
\end{proof}

{\bf Proof of Proposition \ref{stabilizer}}
From the discussion at the beginning of this subsection, it suffices to consider a cell of the form $e'$ described there.
We denote by ${\bf Y}, {\bf X}$ and $X_1,...,X_k$ the same sets as above.

By Lemma \ref{StabinF} we know that $$Stab_G(e')=\bigcap_{1\leq i\leq k} Stab_F(F\tau_{X_i})$$
Moreover, by Lemma \ref{blockstab} it follows that $\bigcap_{1\leq i\leq k} Stab_F(F\tau_{X_i})$
is a finite product of copies of $F$ and $F_3$.
These groups are of type $\mathbf{F}_{\infty}$, which is a property closed under taking finite products.

\subsection{The quotient $X/G$}

Now we show that the quotient $X/G$ has finitely many cells in each dimension.

{\bf Proof of Proposition \ref{quotient}}

We show this for a given dimension $n\in \mathbf{N}$.
Any $n$-cell is a maximal dimensional cell of an $n$-cluster by construction.
Given an $n$-cluster, it contains in its $G$-orbit an $n$-cluster described with base the trivial coset $F$ and proper parameters.
So it suffices to show that there are finitely many $G$-orbits of $n$-clusters described with base $F$ and proper parameters.

Let $C_1,C_2$ be $n$-clusters described with base $F$ and proper parameters $\tau_1,...,\tau_n$ and $\nu_1,...,\nu_n$ respectively.

Assume the following holds for each $1\leq i\leq n$.
\begin{enumerate}
\item The parity of the number of percolating elements in $\nu_i$ and $\tau_i$ is the same.
\item $\tau_i,\tau_{i+1}$ are consecutive if and only if $\nu_i,\nu_{i+1}$ are consecutive. (For $1\leq i\leq n-1$.)
\item If $y_s^t,y_u^v$ are the leftmost percolating elements occurring in $\lambda_i,\nu_i$ respectively,
then $t=u$.
\end{enumerate}

Then we can apply expansion moves to the respective special forms to obtain proper parameters $\tau_1',...,\tau_n'$ and $\nu_1',...,\nu_n'$ such that
in addition to the above, the following holds.
\begin{enumerate}
\item[(4)] For each $1\leq i\leq n$, the word length of $\tau_i',\lambda_i'$ is the same.
\end{enumerate}
Let $$y_{s_1}^{t_1}...y_{s_k}^{t_k}=\tau_{\{1,...,n\}}'\qquad y_{u_1}^{v_1}...y_{u_k}^{u_k}=\nu_{\{1,...,n\}}'$$
Given these conditions, it is an elementary exercise in the group $F$ to construct an element $f\in F$ such that $s_i\cdot f=u_i$ for each $1\leq i\leq k$.
It follows that $C_1\cdot f=C_2$.
Since the number of combinatorial conditions in $(1)-(3)$ above is finite, we are done.

\section{$X$ is simply connected.}
In this section we will show that $X$ is simply connected.
This is shown by providing an explicit homotopy for a given loop
to a trivial loop.
Recall that in \cite{LodhaMoore}, we provide a procedure that takes as input a word 
that represents the identity, and using the moves described in the preliminaries in Definition \ref{moves}, converts it into the empty word.
A direct consequence of the result is that we can use the same procedure to do the following.
If we input a word in the infinite generating set $$\{y_v,x_{u}\mid u,v\in 2^{<\mathbf{N}}, v\neq 0^k,1^k,\emptyset\}$$ of $G$ that represents an element of the subgroup $F$,
then using the moves we can convert it into a word in the generators $$\{x_{u}\mid u\in 2^{<\mathbf{N}}\}$$

In what follows the reader is not expected to have the knowledge of this procedure,
but needs to assume that indeed it is possible to do the above using the combinatorial moves described in
the preliminaries.  
We provide a visual interpretation of the combinatorial moves inside the complex.
We will show that each move provides a homotopy between two paths,
and combine this with the existence of the procedure to demonstrate that $X$ is simply connected.

 Given a loop $l$ in $X$, (up to homotopy) we can assume that $l$ is a path along
 the $1$-skeleton that is described as a path along a sequence of nodes
 $$F\tau,F(\nu_k \tau),F(\nu_{k-1} \nu_k \tau),...,F(\nu_1...\nu_k \tau),F\tau$$
 where 
 \begin{enumerate}
 \item $\tau$ is a $Y$-word.
 \item $\nu_1,...,\nu_k$ 
 are special forms.
 \item $F(\nu_1...\nu_k)=F\psi$ for a special form $\psi$.
 \end{enumerate}
 Note that in general, the special forms $\nu_1,...,\nu_k$ are completely arbitrary. 
 The product word $\nu_1...\nu_k$ (viewed as a word obtained by concatenating $\nu_1,...,\nu_k$) may not be a special form,
 and pairs $\nu_i,\nu_j$ are not necessarily independent.
 However, since the last two nodes in the sequence are connected by an edge
 we deduce that $F\nu_1...\nu_k=F\psi$ where $\psi$ is a special form.
 
 Let $l_1$ be the loop $l\cdot \tau^{-1}$.
 By construction of $X$, $l$ is homotopic to the trivial loop if and only if $l_1$ has this property.
 $l_1$ is described as a path along a sequence of nodes:
  $$F,F(\nu_k),F(\nu_{k-1} \nu_k),...,F(\nu_1...\nu_k),F$$
 where as above:
 \begin{enumerate}
 \item $\nu_1,...,\nu_k$ 
 are special forms.
 \item $F(\nu_1...\nu_k)=F\psi$ for a special form $\psi$.
 \end{enumerate}
 
 We describe a list of \emph{homotopies} that correspond to moves 
 in the analysis of standard forms.
 We list them below.
 In what follows, $\tau\in G$ is an arbitrary element.
 
 \begin{enumerate}
  \item The expansion move 
  $y_{\sigma}\to x_{\sigma}y_{\sigma \seq{0}}y_{\sigma \seq{10}}^{-1}y_{\sigma \seq{11}}$
 corresponds to a homotopy between the paths $F\tau,F(y_{\sigma}\tau)$ and 
 $$F\tau,F(y_{\sigma\seq{11}}\tau),F(y_{\sigma\seq{10}}^{-1}y_{\sigma\seq{11}}\tau),
 F(y_{\sigma \seq{0}}y_{\sigma \seq{10}}^{-1}y_{\sigma \seq{11}}\tau)$$
 Similarly the expansion move 
 $y_{\sigma}^{-1}\to x_{\sigma}^{-1}y_{\sigma\seq{00}}^{-1}y_{\sigma\seq{01}}
 y_{\sigma\seq{1}}^{-1}$
 corresponds to a homotopy between paths $F\tau,F(y_{\sigma}^{-1}\tau)$ and 
 $$F\tau,F(y_{\sigma\seq{1}}^{-1}\tau),
 F(y_{\sigma\seq{01}}y_{\sigma\seq{1}}^{-1}\tau),
 F(y_{\sigma\seq{00}}^{-1}y_{\sigma\seq{01}}y_{\sigma\seq{1}}^{-1}\tau)$$
 These paths are homotopic in $X$ because they are homotopic in a $3$-cluster of $X$.
 For instance, the first homotopy is performed in the cluster with base $F\tau$ and parameters $y_{\sigma \seq{0}},y_{\sigma \seq{10}}^{-1},y_{\sigma \seq{11}}$.
 
 \item Consider an edge of the form
 $F\tau,F(y_s^tf\tau)$.
This is an edge since it satisfies $(E2)$ of Definition \ref{edgerelation}, which can be seen by performing an expansion followed by rearrangement move on $y_s^tf$.
  Performing expansion moves on $y_s^t$ we obtain a special form 
 $y_{s_1}^{t_1}...y_{s_n}^{t_n}$ equivalent to  $y_s^t$ 
 such that $f$ acts on $s_1,...,s_n$. 
 The single edge path $F\tau,F(y_s^tf\tau)$ is homotopic to
 the path described by the sequence 
 $$F\tau,F(y_{s_n}^{t_n}f\tau),F(y_{s_{n-1}}^{t_{n-1}}y_{s_n}^{t_n}f\tau),...,
 F(y_{s_1}^{t_1}...y_{s_n}^{t_n}f\tau)$$
 in the $n$-cluster described with base $F\tau$ and parameters $y_{s_1}^{t_1},...,y_{s_n}^{t_n}$. 
 Applying the rearranging move
 $$y_{s_1}^{t_1}...y_{s_n}^{t_n}f=
 fy_{r_1}^{t_1}...y_{r_n}^{t_n}\qquad \text{ where }r_i=s_i\cdot f$$
 produces a new description of this path as 
 $$F,F(y_{r_n}^{t_n}\tau),F(y_{r_{n-1}}^{t_{n-1}}y_{r_n}^{t_n}\tau),...,
 F(y_{r_1}^{t_1}...y_{r_n}^{t_n}\tau)$$
 
 \item A commuting move $y_s^ty_p^q=y_p^qy_s^t$ for $t,q\in \{-1,1\}$ 
 and $s,p$ independent,
 corresponds to a homotopy between paths of the form $$F\tau,F(y_s^t\tau),F(y_p^qy_s^t\tau)$$
 and $$F\tau,F(y_p^q\tau),F(y_s^ty_p^q\tau)$$
 This can be performed in the complex since these paths are homotopic in the $2$-cluster
 described with base $F\tau$ and parameters $y_s^t,y_p^q$.
 
 \item A cancellation move $y_s^{\pm 1}y_s^{\mp 1}\to \emptyset$ corresponds to shrinking
 a path of the form $$F\tau,F(y_s^{\pm 1}\tau),F(y_s^{\mp 1}y_s^{\pm 1}\tau)$$ to a trivial path.
 This is possible since the path is obtained by traversing the edge $F\tau,F(y_s^{\pm 1}\tau)$ forwards and then backwards.
 \end{enumerate}
 
 Now we observe the following.
 
 \begin{lem}
 Consider a loop $l$ described by the path 
 $$F,F(\lambda_n),F(\lambda_{n-1}\lambda_n),...,
 F(\lambda_1...\lambda_n), F$$ 
 where $\lambda_1,...,\lambda_n$ are special forms.
 It is homotopic to a loop described by a path of the form
 $$F,F(y_{s_m}^{t_m}),F(y_{s_{m-1}}^{t_{m-1}}y_{s_m}^{t_m}),...,
 F(y_{s_1}^{t_1}...y_{s_m}^{t_m}),F$$
 \end{lem}
 \begin{proof}
 We show this by induction on $n$.
 For $n=1$, let $\lambda_1=y_{p_1}^{q_1}...y_{p_k}^{q_k}$.
 Now the $1$-cell $\{F\tau,F(\lambda_1\tau)\}$ is the cross diagonal $1$-cell
 of the $k$-cluster at $\tau$ parametrized by the special forms 
 $y_{p_1}^{q_1},...,y_{p_k}^{q_k}$.
 It follows that this $1$-cell is homotopic to the path
 $$F\tau,F(y_{p_k}^{q_k}\tau),F(y_{p_{k-1}}^{q_{k-1}}y_{p_k}^{q_k}\tau),...,
 F(y_{p_1}^{q_1}...y_{p_k}^{q_k}\tau)$$
 The inductive step is essentially the same as the base case,
 since by the inductive hypothesis 
 we replace the path $$F\tau,F(\lambda_n\tau),F(\lambda_{n-1}\lambda_n\tau),...,
 F(\lambda_2...\lambda_n\tau),F(\lambda_1...\lambda_n\tau)$$
 by a path $$F\tau,F(y_{u_l}^{v_l}\tau),F(y_{u_{l-1}}^{v_{l-1}}y_{u_l}^{v_l}\tau),...,
 F(\lambda_1y_{u_1}^{v_1}...y_{u_l}^{v_l}\tau)$$
 and then argue the last edge $$F(y_{u_1}^{v_1}...y_{u_l}^{v_l}\tau),
 F(\lambda_1 y_{u_1}^{v_1}...y_{u_l}^{v_l}\tau)$$ traversed in the path is homotopic to path 
 of a suitable sequence of edges as in the base case.

\end{proof} 
\begin{proposition}\label{simpconn}
The complex $X$ is simply connected.
\end{proposition}

\begin{proof}
Let $L$ be a loop described as a path in $X^{(1)}$.
 By considering the group action we can assume that the loop $L$ begins and ends at
 $F$.
 By the previous Lemma this loop is homotopic to a loop of the form
 $$F,F(y_{s_n}^{t_n}),F(y_{s_{n-1}}^{t_{n-1}}y_{s_n}^{t_n}),...,
 F(y_{s_1}^{t_1}...y_{s_n}^{t_n})=F$$
 Since $F=F(y_{s_1}^{t_1}...y_{s_n}^{t_n})$ 
 it follows that the word $y_{s_1}^{t_1}...y_{s_n}^{t_n}$ represents an element of $F$.
 From \cite{LodhaMoore} we know that this word can be reduced to an
 $X$-word by applying a sequence of expansion, commuting, cancellation and rearranging substitutions.
 Since an application of each substitution produces a loop homotopic to $L$,
 we observe that this process provides an explicit homotopy between
 $L$ and the trivial loop.
\end{proof}

\section{$X$ is aspherical}

In this section, we shall demonstrate that the complex $X$ is a nonpositively curved cluster complex in the sense of Definition \ref{npccluster},
and hence is aspherical.
In particular, we will show the following.
For each finite subcomplex $Y$ of $X$ which is a union of clusters, there is a subcomplex ${\bf Y}$ of $X$ such that:
\begin{enumerate}
\item $Y\subset {\bf Y}$.
\item ${\bf Y}$ is homeomorphic to a nonpositively curved cube complex.
\end{enumerate}

We describe a process that takes as an input $Y$ and produces as an output ${\bf Y}$.
There are two main structural concepts defined and used in this section.
These are the notions \emph{parallel $1$-clusters} and \emph{orthogonal $1$-clusters}.
A key algorithmic ingredient is what we refer to as \emph{amplification} of clusters.
An amplification takes as an input a cluster, and produces as an output a larger cluster that contains the input as a subcluster.

In this section whenever we refer to a description of a $1$-cluster with base $F\tau$,
it shall be assumed that $\tau$ is a $Y$-word (possibly empty) (unless specified otherwise).
We shall write $F\tau, F\nu \tau$ to denote a $1$-cluster with base $F\tau$ and parameter $\nu$.
It is important that the reader keeps this in mind, since this will be crucial to formulate the definitions in this section.

\subsection{The parallel equivalence relation on the $1$-cells of $X$}

Consider the regular Euclidean cube $[0,1]^n\subset \mathbf{R}^n$.
Two $1$-dimensional faces, i.e. edges of the cube are parallel if they are segments of parallel lines in $\mathbf{R}^n$.
Now consider an $n$-cluster $C$, viewed as a CW subdivision of $[0,1]^n$.
The facial $1$-subclusters of $C$ are in natural correspondence with the closed edges of $[0,1]^n$.
So we may view a pair of facial $1$-subclusters of $C$ as \emph{parallel} if the corresponding closed edges of $[0,1]^n$ are 
parallel.
More concretely, let $C$ be described with base $F\tau$ and parameters $\tau_1,...,\tau_n$.
The set of facial $1$-subclusters of $C$ consists of:
$$\{F\tau_X\tau,F\tau_i\tau_X\tau\mid X\cup \{i\}\subseteq \{1,...,n\}, i\notin X\}$$

We say that two facial $1$-subclusters of $C$ are \emph{parallel}, if they are of the form:
$$F\tau_X\tau, F\tau_i(\tau_X\tau) \qquad  F\tau_Y\tau, F\tau_i(\tau_Y\tau)$$
where $X,Y\subseteq \{1,...,i-1,i+1,...,n\}$. 

We would like to extend this notion of parallel $1$-clusters to pairs that
are not necessarily facial $1$-subclusters of the same cluster.
As a motivating example, consider the situation below.

Let $C$ be as above,
but for simplicity of notation, let us assume that $\tau$ is the empty word,
and hence $C$ is described with base as the trivial coset and parameters $\tau_1,...,\tau_n$.
Let $$e_1=F\tau_X, F\tau_i\tau_X\qquad  e_2=F\tau_Y, F\tau_i\tau_Y$$
be parallel, facial $1$-subclusters of $C$.

Now consider a cluster $D$ described with base $F\tau_Y$
and parameters $\nu_1,...,\nu_m$ such that
$\nu_j=\tau_i$ for some $1\leq j\leq m$.
Consider a $1$-cell $$e_3=F(\nu_Z\tau_Y),F\nu_j(\nu_Z\tau_Y)$$ of $D$ such that $Z\subseteq \{1,...,j-1,j+1,...,m\}$.
Note that in this description of $D$
$$e_2=F\tau_Y,F\nu_j\tau_Y$$

It follows that $e_2,e_3$ are parallel, facial $1$-subclusters of the cluster $D$.
It is natural to view $e_1,e_3$ above as in a certain sense, parallel.
Indeed, there is a natural notion of \emph{parallel $1$-clusters} in $X$. 
This is precisely the transitive closure of the relation obtained by declaring $1$-clusters $e_1,e_2$ as parallel
if they are parallel, facial $1$-subclusters of a cluster.

We provide a more useful, concrete formulation of this definition below.
In this formulation, it is not immediately apparent that the resulting relation is an equivalence relation.
This will be shown in a subsequent Lemma.

\begin{defn}\label{parequiv}
Two $1$-clusters $e_1,e_2$ in $X$ are said to be \emph{parallel}
if they admit descriptions
$$e_1=F\tau_1, F\nu \tau_1\qquad e_2=F\tau_2, F\nu \tau_2$$
with a common parameter $\nu$ so that there is a sequence of special forms $\nu_1,...,\nu_n$ satisfying:
\begin{enumerate}
\item $\nu_i, \nu$ are independent for each $1\leq i\leq n$.
\item $\nu_1...\nu_n=\tau_2\tau_1^{-1}$. 
\end{enumerate}
\end{defn}

\begin{remark}\label{parequivrem}
In the above definition, it is not required that $\nu_1,...,\nu_n$ are independent,
sorted etc.  In fact in most cases they shall not be independent.

Note that the definition is symmetric. 
The special forms $\nu_1^{-1},...,\nu_n^{-1}$
satisfy that $\nu_i^{-1},\nu$ are independent and $\nu_n^{-1}...\nu_1^{-1}=\tau_1\tau_2^{-1}$.
Also, note that the definition is $G$-invariant.
That is, if $e_1,e_2$ are parallel then it follows from the definition that 
$e_1\cdot g, e_2\cdot g$ are parallel for any $g\in G$.

Finally, note that if the descriptions $$e_1=F\tau_1, F\nu \tau_1\qquad e_2=F\tau_2, F\nu \tau_2$$
satisfy Definition \ref{parequiv}, then so do the descriptions $$e_1=F(\nu \tau_1), F\nu^{-1} (\nu \tau_1)\qquad e_2=F(\nu \tau_2), F\nu^{-1}  (\nu \tau_2)$$
In fact, the same $\nu_1,...,\nu_n$ as above witness the conditions of the Definition.
Indeed it holds that $\nu_i, \nu^{-1}$ are independent for each $1\leq i\leq n$.
And also, $$\nu_1...\nu_n =\nu\nu^{-1}(\nu_1...\nu_n)= \nu (\nu_1...\nu_n)\nu^{-1}=\nu(\tau_2 \tau_1^{-1})\nu^{-1}=(\nu \tau_2) (\nu \tau_1)^{-1}$$
\end{remark}

Before we show that the parallel relation is an equivalence relation,
we show that the definition is independent of the choice of description in the following sense.

\begin{lem}\label{indparpar}
Let $e_1,e_2$ be parallel $1$-clusters in $X$.
Let $e_1=F\tau_1, F\nu\tau_1$ be any choice of description.
Then there is a description $e_2=F\tau_2,F\nu \tau_2$,
so that these descriptions satisfy Definition \ref{parequiv}.
\end{lem}

\begin{proof}
We shall find a $Y$-word $\tau_2$ such that $e_2$ admits a description
$e_2=F\tau_2,F\nu \tau_2$
satisfying the required property.
Since $e_1,e_2$ are parallel, by definition there are descriptions
$$e_1=F\psi_1,F\eta\psi_1\qquad e_2=F\psi_2, F\eta\psi_2$$
and special forms $\eta_1,...,\eta_n$
satisfying that  $\eta_i,\eta$ are independent and $\eta_1...\eta_n\psi_1=\psi_2$.
Thanks to the last paragraph of Remark \ref{parequivrem}, we can also find the above descriptions keeping the same base coset, so that $F\psi_1=F\tau_1$.

There is an $f\in F$ such that $f\psi_1=\tau_1$.
Let $\eta',\eta_1',...,\eta_n'$ be special forms obtained by applying expansion moves on 
$\eta,\eta_1,...,\eta_n$ respectively such that for each percolating element $y_s^t$ that occurs,
$f^{-1}$ acts on $s$.
Let $\eta'',\eta_1'',...,\eta_n''$ be special forms obtained by replacing each percolating element $y_s^t$ of $\eta',\eta_1',...,\eta_n'$ by $y_{s\cdot f^{-1}}^t$.

It follows that $$F\eta\psi_1=F\eta' (f^{-1}f) \psi_1=F\eta'' \tau_1$$
and hence $e_1=F\tau_1, F\eta'' \tau_1$.
Moreover, 
$$\psi_2=\eta_1...\eta_n\psi_1=\eta_1...\eta_n(f^{-1} f)\psi_1= f^{-1}\eta_1''...\eta_n''\tau_1$$
and hence 
$$f\psi_2=\eta_1''...\eta_n''\tau_1$$
We define $\tau_2=f\psi_2$.
It follows that $\tau_2=\eta_1''...\eta_n''\tau_1$.
So in particular $\tau_2$ is expressible as a $Y$-word, which we fix and refer to as $\tau_2$ for the remainder of the proof.

Now $F\eta''\tau_1=F\nu \tau_1$ and so $F\eta''= F\nu$.
Since $\nu,\eta''$ are special forms that lie in the same $F$-coset, $\nu$ can be obtained from $\eta''$ by a sequence of expansion and contraction moves and hence has the same support.
In particular, it follows that the pairs $\eta_i'', \nu$ are independent for each $1\leq i\leq n$.
Therefore, the descriptions 
$$F\tau_1,F\nu\tau_1\qquad F\tau_2,F\nu\tau_2$$
satisfy the condition of Definition \ref{parequiv} with the special forms
$\eta_1'',...,\eta_n''$ witnessing the conditions.
\end{proof}

\begin{lem}\label{equiv}
The parallel $1$-cluster relation is an equivalence relation.
The relation is precisely the transitive closure of the relation obtained by declaring $1$-clusters $e_1,e_2$ as parallel
if they are parallel, facial $1$-subclusters of a cluster. 
\end{lem}

\begin{proof}
In remark \ref{parequivrem} we discussed why Definition \ref{parequiv} is symmetric.
Therefore, it suffices to show that this is transitive, i.e.
given $1$-cells $e_1,e_2,e_3$ such that $e_1,e_2$ and $e_2,e_3$ are parallel,
then $e_1,e_3$ are parallel.
Fix a description $F\tau_1, F\nu \tau_1$ of $e_1$.
By two applications of Lemma \ref{indparpar},
first on $e_1,e_2$ and then on $e_2,e_3$
we find descriptions
$$e_1=F\tau_1, F\nu \tau_1\qquad e_2=F\tau_2, F\nu \tau_2\qquad e_3=F\tau_3, F\nu \tau_3$$
that satisfy the following conditions of Definition \ref{parequiv}:
\begin{enumerate}
\item There are special forms $\eta_1,...,\eta_n$ such that each pair $\nu,\eta_i$ is independent and $\eta_1...\eta_n=\tau_1\tau_2^{-1}$.

\item There are special forms $\psi_1,...,\psi_m$ such that each pair $\nu,\psi_i$ is independent and $\psi_1...\psi_m=\tau_2\tau_3^{-1}$.
\end{enumerate}

It follows that $e_1,e_3$ are parallel since each pair $\nu,\eta_i$ and $\nu,\psi_i$ is independent and 
$$\eta_1...\eta_n\psi_1...\psi_m=\tau_1\tau_2^{-1}\tau_2\tau_3^{-1}=\tau_1\tau_3^{-1}$$
This finishes the proof that our relation is an equivalence relation.

We leave the claim concerning the transitive closure as a pleasant exercise for the reader.
This requires translating the definition into a sequence of clusters witnessing the transitive closure. 
This statement will not be used in the rest of the article, and is stated to provide motivation for the concept.
\end{proof}

\begin{lem}\label{clusterequiv}
The set of facial $1$-subclusters of an $n$-cluster comprise of exactly $n$ parallel equivalence classes.
A set of representatives for these equivalence classes is given by the set of facial $1$-subclusters incident to any given $0$-cell in the cluster.
\end{lem}

\begin{proof}
Consider an $n$-cluster $C$ described with base at $F\tau$ and parameters $\nu_1,...,\nu_n$.
By $G$-invariance of Definition \ref{parequiv}, we can assume that $\tau$ is the empty word.
First note that each facial $1$-subcluster in $C$ is parallel to a facial $1$-subcluster of the form $F, F\nu_i$ for some $1\leq i\leq n$.
We will show that a pair $$e_1=F,F\nu_i\qquad e_2=F, F\nu_j$$ is not parallel if $i\neq j$.
Assume by way of contradiction that $e_1,e_2$ are parallel.
By Lemma \ref{indparpar} there is a description $e_2=F\psi, F\nu_i\psi$, where $\psi$ is a $Y$-word, such that the descriptions 
$$F,F\nu_i\qquad F\psi, F\nu_i\psi$$ satisfy Definition \ref{parequiv}.
So there are special forms $\eta_1,...,\eta_m$ such that $\eta_k, \nu_i$ are independent for each $1\leq k\leq m$,
and $\eta_1...\eta_m=\psi$.

Since we are considering two different descriptions of $e_2$: 
$$e_2=F,F\nu_j\qquad e_2=F\psi, F\nu_i\psi$$
we have to consider two possibilities:

\begin{enumerate}
\item $F=F\psi$ and $F\nu_j=F\nu_i\psi$.
\item $F=F\nu_i\psi$ and $F\nu_j=F\psi$.
\end{enumerate}

Before we analyse each case separately,
we observe the following.
Since $\eta_1...\eta_m=\psi$, and for each $1\leq j\leq m$, $\eta_j,\nu_i$ are independent,
it follows that: $$F\nu_i\psi=F\nu_i(\eta_1....\eta_n)=F(\eta_1...\eta_n)\nu_i=F\psi \nu_i$$

Case $1$:
Since $F=F\psi$, we conclude that $\psi\in F$ (even though $\psi$ is formally a $Y$-word, it is an element of $F$).
It follows that $$F\nu_i\psi=F\psi\nu_i=F\nu_i$$
But this means that $F\nu_j=F\nu_i$.
This leads to a contradiction since $\nu_i,\nu_j$ are independent special forms and so the element $\nu_j\nu_i^{-1}\notin F$.

Case $2$:
In this case, observe that:
$$F=F\nu_i\psi=F\psi\nu_i=F\nu_j\nu_i$$
This means that $F=F\nu_j\nu_i$, which is a contradiction since $\nu_i,\nu_j$ are independent special forms, so the element $\nu_j\nu_i\notin F$.

This means that our assumption that $e_1,e_2$ are parallel must be false,
and hence we conclude that the set of facial $1$-subclusters incident to the trivial coset $F$ in $C$ are pairwise non-parallel.
\end{proof}

We conclude with the following strengthening of the main idea in the previous proof.

\begin{lem}\label{samenodeparallel}
Let $e_1,e_2$ be distinct $1$-clusters incident to the same $0$-cell $u$.
Then $e_1,e_2$ cannot be parallel.
\end{lem}

\begin{proof}
Considering the invariance of the notion under the $G$-action, we may assume for simplicity that $u$ is the trivial coset.
Let $$e_1=F, F\nu \qquad e_2=F\tau, F\nu \tau$$
with a common parameter $\nu$ so that there is a sequence of special forms $\nu_1,...,\nu_n$ satisfying:
\begin{enumerate}
\item $\nu_i, \nu$ are independent for each $1\leq i\leq n$.
\item $\nu_1...\nu_n=\tau$. 
\end{enumerate}

Note that the conditions above imply that $\tau,\nu$ have disjoint supports.
There are two cases:
\begin{enumerate}
\item $F=F\tau$.
\item $F=F\nu\tau$.
\end{enumerate}

Case $1$:
Since $F=F\tau$, $\tau\in F$.
Hence $F\nu\tau=F\tau\nu=F\nu$.
This means that $e_1=e_2$.

Case $2$: This implies that $\nu\tau\in F$.
But since $\nu,\tau$ have disjoint supports, this is impossible since $\nu \notin F$.
\end{proof}

While working with finite sets of (parallel) $1$-clusters,
it will often be useful to consider descriptions that share the same parameter.

\begin{defn}
(Common description) Consider the descriptions:
$$e_1=F\tau_1,F\nu\tau_1\qquad e_2=F\tau_2,F\nu\tau_2\qquad ...\qquad e_n=F\tau_n,F\nu\tau_n$$
This shall be denoted as a \emph{common description} for the $1$-cells $e_1,...,e_n$, and $\nu$ is denoted as a \emph{common parameter}.
\end{defn}

\begin{remark}
If a set of $1$-clusters $e_1,...,e_n$ are pairwise parallel, then by repeated application of Lemma \ref{indparpar}, as in the proof of Lemma \ref{equiv},
we can find a common description for these $1$-clusters satisfying Definition \ref{parequiv} for each pair.
\end{remark}

\subsection{Orthogonal $1$-clusters}

As discussed before, our proof of asphericity involves finding a subcomplex which is homeomorphic to a nonpositively curved cube complex.
The idea behind proving nonpositive curvature is Gromov's link condition, and we shall establish the link condition 
by means of the following notion of \emph{orthogonality} of $1$-clusters.

\begin{defn}\label{orth}
(Orthogonal $1$-clusters) Let $e_1,e_2$ be $1$-clusters that are both incident to a $0$-cell $u$.
They are said to be \emph{orthogonal} if they admit descriptions:
$$e_1=F\tau,F\nu_1\tau\qquad e_2=F\tau,F\nu_2\tau$$
such that $F\tau=u$ and $\nu_1,\nu_2$ are independent special forms.
This is equivalent to requiring that there is a $2$-cluster $C$ which contains $e_1,e_2$ as facial $1$-subclusters incident to $u$.

A list $e_1,...,e_n$ of $1$-clusters that are all incident to a $0$-cell $u$ are said to be \emph{orthogonal} if they are pairwise orthogonal.
\end{defn}

\begin{remark}
The notion of orthogonality is clearly $G$-invariant, i.e. $e_1,e_2$ are orthogonal if and only if $e_1\cdot g,e_2\cdot g$
are orthogonal for any $g\in G$.
We leave it as an elementary exercise for the reader to check that the notion does not depend on the choice of descriptions, as long as the choice of base is the same.
In this case, the parameters shall automatically be independent.
\end{remark}

\begin{lem}\label{orthlem}
Let $e_1,...,e_n$ be $1$-clusters that are all incident to a $0$-cell $u$.
If they are orthogonal, then there is an $n$-cluster that contains $e_1,...,e_n$ as facial $1$-cells.
\end{lem}

\begin{proof}
By $G$-invariance of the notion of orthogonality, we can assume without loss of generality that $u$ is the trivial coset $F$.
Now we find descriptions for our $1$-clusters with base $F$:
$$e_1=F,F\tau_1\qquad ....\qquad e_n=F,F\tau_n$$
for special forms $\tau_1,...,\tau_n$.
Since $e_1,...,e_n$ are pairwise orthogonal, it follows that $\tau_1,...,\tau_n$ are pairwise independent.
So the cluster $C$ can be described with base $F$ and parameters $\tau_1,...,\tau_n$.
\end{proof}

The above Lemma provides a bijective correspondence 
between the set of $n$-clusters incident to a given $0$-cell $u$,
and the set of $n$-tuples of pairwise orthogonal $1$-clusters incident to $u$.

$$\{\text{ $n$-clusters incident to $u$}\}\leftrightarrow \{\text{$n$-tuples of orthogonal $1$-clusters incident to $u$}\}$$

This discussion motivates the following definition.

\begin{defn}
A \emph{pre-cluster} $C$ consists of $1$-clusters $e_1,...,e_n$ that are all incident to a $0$-cell $u$, and are orthogonal.
The $0$-cell $u$ is said to be the \emph{base} of the pre-cluster $C$. 

Pre-clusters admit descriptions just like clusters.
For instance, a pre-cluster $C$ described with base $F\tau$, and parameters $\tau_1,...,\tau_n$ consists of the $1$-clusters:
$$F\tau,F\tau_1\tau\qquad ... \qquad F\tau,F\tau_n\tau$$
with base the $0$-cell given by $F\tau$.
The list of parameters $\tau_1,...,\tau_n$ is always assumed to be a sorted list for convenience.

The \emph{closure} $\bar{C}$ of the $n$-pre-cluster $C$ is the cluster described with base $F\tau$
and parameters $\tau_1,...,\tau_n$.
Clearly,  the $1$-clusters in $C$ are all facial $1$-subclusters of $\bar{C}$,
and provide a set of representatives of the parallel equivalence classes of the facial $1$-subclusters of $\bar{C}$.
\end{defn}
Note that in the above definition we assume that a $1$-cluster is also a $1$-pre-cluster, whose closure is itself.

\subsection{Amplification of pre-clusters}

An amplifying move is an operation that takes as an input a pre-cluster $C_1$ and produces as an output a ``larger" pre-cluster
$C_2$ such that $\bar{C}_1$ is a subcluster of $\bar{C}_2$.

The idea behind the amplifying move is elementary, yet technical difficulties
in implementing the move on a given pre-cluster arise from the fact that we need to choose a description.
We will show that in a certain precise sense, the concept is independent of the choice of description.
We first provide an elementary example to illustrate the motivation behind this concept.

\begin{example}
Consider the $1$-cluster $F, Fy_{10}$.
We can apply an expansion move on the special form $y_{10}$
to obtain the special form $y_{100}y_{1010}^{-1}y_{1011}$.
Consider the pre-cluster $C_1$ consisting of the $1$-clusters 
$$F, Fy_{100}\qquad F,Fy_{1010}^{-1} \qquad F,Fy_{1011}$$
The closure is a $3$-cluster $\bar{C}_1$ described with base $F$ and parameters $y_{100},y_{1010}^{-1},y_{1011}$.
This cluster contains the $1$-cluster $F, Fy_{10}$ as a cross diagonal.

Now consider the expansion move $y_{100}\to y_{1000}y_{10010}^{-1}y_{10011}$.
Consider the $5$-pre-cluster $C_2$ comprising of
$$\{F,F\tau\mid \tau\in \{y_{1000},y_{10010}^{-1},y_{10011},y_{1010}^{-1},y_{1011}\}\}$$
Then $\bar{C}_2$ is described with base $F$ and parameters $$y_{1000},y_{10010}^{-1},y_{10011},y_{1010}^{-1},y_{1011}$$
and contains $\bar{C}_1$ as a diagonal subcluster, and also the $1$-cluster $F, Fy_{10}$ as a cross diagonal.
Continuing in this fashion, we can produce arbitrarily large (pre)-clusters, each new cluster (in the closure) containing the previous ones as subclusters.
\end{example}

\begin{defn}\label{amp1cl}
(Amplifying a $1$-cluster) 
Consider a $1$-cluster $C_1=F\tau,F\nu\tau$ with base $u=F\tau$.
Each sorted list of special forms $\nu_1,...,\nu_n$ satisfying $F\nu=F\nu_1...\nu_n$ determines an \emph{amplification of} $C_1$ \emph{at} $u$, which produces a pre-cluster $C_2$ consisting of:
$$F\tau,F\nu_1\tau\qquad F\tau,F\nu_2\tau\qquad...\qquad F\tau,F\nu_n\tau$$

Such an amplification is informally denoted as $C_1\to C_2$.
The $1$-clusters of $C_2$ shall be referred to as the \emph{offsprings} of the $1$-cluster $C_1$. 
\end{defn}

\begin{remark}

Note that the definition is $G$-invariant in the following sense.
Let $C_1,C_2$ be as above.
For any $g\in G$, the pre-cluster
$C_2\cdot g$ consisting of 
$$F(\tau g),F\nu_1(\tau g)\qquad F(\tau g),F\nu_2(\tau g)\qquad...\qquad F(\tau g),F\nu_n(\tau g)$$
is produced by an amplification of $C_1\cdot g= F(\tau g), F\nu (\tau g)$ given by $F\nu=F\nu_1...\nu_n$.
\end{remark}

Now we show that the same amplifying move can be obtained from any description of a given $1$-cluster, with the same base coset.
In this sense, the notion is independent of the choice of description.

\begin{lem}\label{inddesc1}
Let $C_1$ be a $1$-cluster $F\tau, F\nu\tau$.
Let $C_2$ be the amplification consisting of
$$e_1=F\tau,F\nu_1\tau\qquad ... \qquad e_n=F\tau,F\nu_n\tau$$
where $\nu_1,...,\nu_n$ is a sorted list of special forms such that $F\nu= F\nu_1...\nu_n$.

Consider a different description $C_1=F\psi,F\eta \psi$,
so that $F\psi=F\tau$ and $\eta$ is a special form.
Then there is a sorted list of special forms $\eta_1,...,\eta_n$ such that:
\begin{enumerate}
\item $F\eta=F\eta_1...\eta_n$.
\item $e_i=F\psi, F\eta_i\psi$ for each $1\leq i\leq n$.
\end{enumerate}
\end{lem}

\begin{proof}
Since $F\psi=F\tau$, there is an $f\in F$ such that $\psi=f\tau$.
Let $\nu_i'$ be a special form obtained by applying expansion moves on $\nu_i$
such that for each percolating element $y_s^t$ of $\nu_i'$, $f^{-1}$ acts on $s$.
Let $\nu_i''$ be the special form obtained by replacing each percolating element $y_s^t$ of $\nu_i'$ with $y_{s\cdot f^{-1}}^t$.

It follows that $$F\nu_1...\nu_n \tau=F\nu_1'...\nu_n' \tau=F \nu_1'...\nu_n'f^{-1} (f\tau)$$
$$=F\nu_1''...\nu_n'' (f\tau)=F\nu_1''...\nu_n'' \psi$$
In particular $F\nu_1''...\nu_n''= F\eta$.
By Lemma \ref{FactionspecialL} we know that $\nu_1'',...,\nu_n''$ is a sorted list of special forms.
Since 
$$e_i=F\tau,F\nu_i\tau= F\psi,F\nu_i''\psi$$
we conclude the statement of the Lemma with $\eta_i=\nu_i''$.
\end{proof}

We obtain the following corollary.

\begin{cor}\label{inddesc}
Given a $1$-cluster, and a base $0$-cell incident to it, 
the set of pre-clusters obtained by amplifying moves on a given description (with the given $0$-cell as base) 
is the same as the set of pre-clusters obtained by amplifying moves on any other description
(with the given $0$-cell as base). 
\end{cor}

To conclude our discussion, we show that the set of clusters obtained from taking closures of amplifications of a $1$-cluster are the same regardless of the choice of base $0$-cell.

\begin{lem}\label{ampindbase}
Let $C_1$ be a $1$-cluster with incident $0$-cells $u,v$.
If $C_2$ is a pre-cluster obtained by performing an amplification of $C_1$ at $u$,
then there is a pre-cluster $C_3$ obtained by performing an amplification of $C_1$ at $v$ such that $\bar{C_3}=\bar{C_2}$.
\end{lem}

\begin{proof}
Let $C_1$ be a $1$-cluster $F\tau, F\nu\tau$.
Let $C_2$ be the amplification consisting of
$$F\tau,F\nu_1\tau\qquad ... \qquad F\tau,F\nu_n\tau$$
where $\nu_1,...,\nu_n$ is a sorted list of special forms such that $F\nu= F\nu_1...\nu_n$.
The closure $\bar{C_2}$ is the cluster described with base $F\tau$ and parameters $\nu_1,...,\nu_n$.

Now consider the description $C_1=F\nu^{-1}(\nu \tau),F(\nu\tau)$ with base coset $F(\nu\tau)$ and parameter $\nu^{-1}$.
Since $\nu$ is equivalent to $\nu_1\nu_2...\nu_n$, it follows from applying Lemma \ref{elementaryinversion} that $\nu^{-1}$ is equivalent to $\nu_1^{-1}\nu^{-1}_2...\nu_n^{-1}$.

Now consider the amplification $C_3$ given by 
$$F(\nu\tau),F\nu_1^{-1}(\nu\tau)\qquad ... \qquad F(\nu\tau),F\nu_n^{-1}(\nu\tau)$$
The closure $\bar{C_3}$ is the cluster described with base $F(\nu\tau)$ and parameters $\nu_1^{-1},...,\nu_n^{-1}$.

Clearly, $\bar{C_2}=\bar{C_3}$ since the above descriptions are different descriptions for the same cluster.
\end{proof}


Now we define the amplifying move for an $n$-pre-cluster for $n>1$.
An amplifying move, or an amplification of a pre-cluster produces another pre-cluster.

\begin{defn}\label{ampclus}
(Amplifying an $n$-pre-cluster) 
Consider a pre-cluster $C_1$ with base $u$.
Let $E\subseteq C_1$ be a set of $1$-clusters $E=\{e_1,...,e_k\}$.
An \emph{amplification} $C_1\to C_2$ \emph{along} $E$
is given by replacing $e_1,...,e_k$ in $C_1$ by amplifications $D_1,...,D_k$ of $e_1,...,e_k$ respectively, at $u$.

More concretely $$C_2= (C_1\setminus \{e_1,...,e_k\})\bigcup (D_1\cup...\cup D_k)$$
We say that $C_2$ \emph{is obtained from amplifying} $C_1$ \emph{along} $e_1,...,e_k$. 
The output is also a pre-cluster with the same base.
Note that the choice of base in such an amplification is unique,
since there is precisely one $0$-cell incident to all the $1$-clusters in $C_1$ since $n>1$.
\end{defn}

Let us consider a concrete example.
\begin{example}
Let $C_1$ be a pre-cluster described as
$$e_1=F\tau,F\tau_1\tau\qquad ... \qquad e_n=F\tau,F\tau_n\tau$$

Fix $j,k\in \{1,...,n\}, j\neq k$.
Let  $\nu_1,...,\nu_{l}$ and $\psi_1,...,\psi_{m}$ be sorted lists such that 
$$F\tau_{j}=F\nu_1...\nu_{l}\qquad F\tau_{k}=F\psi_1...\psi_{m}$$
Given this data, we now describe a $(n+l+m-2)$-pre-cluster $C_2$ which is an amplification of $C_1$ along $\{e_{j},e_{k}\}$.
This pre-cluster $C_2$ is the following union:
$$(\{e_1,...,e_n\}\setminus \{e_{j},e_{k}\} )\bigcup \{F\tau, F\nu_i\tau\mid 1\leq i\leq l\}\bigcup \{F\tau, F\psi_i\tau\mid 1\leq i\leq m\}$$
\end{example}

Note that Definition \ref{ampclus} is $G$-invariant in the same sense as Definition \ref{amp1cl}.
Also, this is independent of the choice of description just as in \ref{amp1cl}.
The above discussion combined with an analogue of Lemma \ref{inddesc1} provides the following.

\begin{lem}\label{ID1}
Let $C$ be a pre-cluster endowed with two different descriptions.
Then the set of pre-clusters obtained from performing amplifying moves using one description is the same as the set of clusters
obtained from performing amplifying moves using the other description.
\end{lem}

\begin{proof}
This proof is almost verbatim the same as the proof of Lemma \ref{inddesc1}.
\end{proof}

\begin{remark}
Let $C_2$ be an amplification of a pre-cluster $C_1$, and let $C_3$ be an amplification of $C_2$.
Then $C_3$ can be obtained from a single amplification of $C_1$.
\end{remark}

We now extend the notion of amplification of a $1$-cluster to a set of pairwise parallel $1$-clusters.

\begin{defn}\label{parallelexpansion}
(Common amplification for parallel $1$-clusters)
Let $e_1,e_2$ be parallel $1$-clusters.
Let $$e_1=F\tau_1,F\nu\tau_1\qquad e_2=F\tau_2,F\nu\tau_2$$
be two descriptions satisfying the conditions of Definition \ref{parequiv}.
Let $\nu_1,...,\nu_n$ be a sorted list of special forms such that $F\nu=F\nu_1...\nu_n$.

Now consider the amplification of $e_1,e_2$ to obtain pre-clusters $C_1, C_2$ respectively, as:
\begin{enumerate}
\item $C_1$ consists of $$F\tau_1,F\nu_1\tau_1\qquad F\tau_1,F\nu_2\tau_1 \qquad...\qquad F\tau_1,F\nu_n\tau_1$$
\item $C_2$ consists of $$F\tau_2,F\nu_1\tau_2\qquad F\tau_2,F\nu_2\tau_2\qquad ...\qquad F\tau_2,F\nu_n\tau_2$$
\end{enumerate}
Such an amplification of parallel $1$-clusters is said to be a \emph{common amplification}.

A second variant of this definition is considered for the description $e_2=F\nu\tau, F\nu^{-1} (\nu \tau)$.
First, note that since $F\nu=F\nu_1...\nu_n$, it follows from applying Lemma \ref{elementaryinversion} that $F\nu^{-1}=F\nu_1^{-1}...\nu_n^{-1}$.
Now consider the amplification $C_1$ of $e_1$ with base $F\tau_1$ as before.
The corresponding \emph{common amplification} of $e_2$ with base $F(\nu\tau_2)$ is $C_2'$ which consists of:
$$F(\nu \tau_2),F\nu_1^{-1}(\nu \tau_2)\qquad F(\nu \tau_2),F\nu_2^{-1}(\nu \tau_2)\qquad ...\qquad F(\nu\tau_2),F\nu_n^{-1}(\nu \tau_2)$$
Note that $\bar{C_2}=\bar{C_2'}$.
\end{defn}

\begin{remark}
Note that in the above definition, for each $1\leq i\leq n$, the $1$-clusters
$$F\tau_1,F\nu_i\tau_1\qquad F\tau_2,F\nu_i\tau_2$$
are parallel.
In the second variant of the definition the $1$-clusters
$$F\tau_1,F\nu_i\tau_1\qquad F(\nu \tau_2),F\nu_i^{-1}(\nu \tau_2)$$
are parallel.

The above definition can be extended to any finite number of $1$-clusters that are pairwise parallel.
This is done by first finding common descriptions that satisfy the conditions of Definition \ref{parequiv}, using Lemma \ref{indparpar}. 
And then we perform common amplifications as above.
\end{remark}

Definition \ref{parallelexpansion} has a natural extension to pre-clusters.

\begin{defn}\label{parallelexpansionclusters}
(Common amplification of pre-clusters along parallel subsets)
Let $C_1,C_2$ be pre-clusters based at $u,v$ respectively.
Let $E_1\subseteq C_1,E_2 \subseteq C_2$ be sets of $1$-clusters such that:
\begin{enumerate}
\item $E_1=\{e_1,...,e_k\}$ and $E_2=\{e_1',...,e_k'\}$.
\item $e_i,e_i'$ are parallel for each $1\leq i\leq k$.
\end{enumerate} 

Fix an amplification $C_1\to C_1'$ along $E_1$ by
$$e_1\to D_1\qquad ... \qquad e_k\to D_k$$
at $v$.

We define a \emph{common amplification} $C_2\to C_2'$
using the amplifications
$$e_1'\to D_1'\qquad ... \qquad e_k'\to D_k'$$
at $u$ where $e_i'\to D_i'$ and $e_i\to D_i$ are common amplifications for each $1\leq i\leq k$.

We denote such a common amplification as:
$$C_1\to C_1'\qquad C_2\to C_2'\qquad \text{ along }E_1,E_2$$
\end{defn}

\subsection{Disparate $1$-clusters}

Now we define a notion that is in a certain strong sense the opposite of the notion of parallel.

\begin{defn}
A pair of $1$-clusters $e_1,e_2$ is said to be \emph{disparate} if the following holds:
\begin{enumerate}
\item $e_1,e_2$ are not parallel.
\item Any pair of amplifications $$e_1\to C_1\qquad e_2\to C_2$$ satisfy that for each pair $e_1'\in C_1,e_2'\in C_2$, $e_1',e_2'$ are not parallel.
\end{enumerate}
\end{defn}

As a simple example of disparate $1$-clusters, consider any pair of orthogonal $1$-clusters incident to a common $0$-cell.
We now demonstrate that being disparate is a property of parallel equivalence classes of $1$-clusters.

\begin{lem}\label{remdisp}
Given $1$-clusters $e_1,e_2,e_3$ such that $e_1,e_2$ are parallel, and $e_1,e_3$ are disparate,
it also holds that $e_2,e_3$ are disparate.
\end{lem}
\begin{proof}
To see this, observe that if $e_2,e_3$ were not disparate, then 
we can perform amplifications to obtain pre-clusters $C_2,C_3$ such that there is a pair $e_2'\in C_2,e_3'\in C_3$
such that they are parallel.
However, performing the common amplification for $e_1$, as for $e_2$, we obtain a pre-cluster $C_1$ with $e_1'\in C_1$
which is parallel to $e_2'$ and hence $e_3'$.
This means that $e_1,e_3$ were not disparate to begin with.
\end{proof}

\begin{lem}\label{disp2}
Let $e_1,e_2$ be disparate, and let $C$ be an amplification of $e_1$.
Then for each $e_3\in C$, $e_2,e_3$ are disparate.
\end{lem}

\begin{proof}
This follows immediately from the definitions.
\end{proof}

\begin{lem}
Let $C$ be a cluster.
Any pair of facial $1$-cells of $C$ is either parallel or disparate.
\end{lem}

\begin{proof}
Let $C$ be a cluster described with base at $F\tau$, and parameters $\tau_1,...,\tau_n$.
Thanks to Lemma \ref{clusterequiv} and Remark \ref{remdisp} it suffices to show that any pair of $1$-clusters
$$e_1=F\tau,F\tau_i\tau\qquad e_2=F\tau,F\tau_j\tau$$
is disparate if $i\neq j$.
From Lemma \ref{clusterequiv} we know that they are not parallel.
Now if $C_1,C_2$ are pre-clusters obtained from amplifications of $e_1,e_2$ respectively (performed at the base $F\tau$),
then for each $e\in C_1,e'\in C_2$ the pair $e,e'$ is orthogonal.
It follows again from Lemma \ref{clusterequiv} that $e,e'$ are not parallel.
\end{proof}

\subsection{Weakly balanced systems of pre-clusters}

\begin{defn}\label{disparate01}
Consider a pair $e,u$ comprising of a $1$-cluster $e$ and a $0$-cell $u$.
$e,u$ are said to be \emph{disparate} if for any $e'$ incident to $u$,  the $1$-clusters $e,e'$ are disparate.
$e,u$ are said to be \emph{compatible}, if there is a $1$-cluster $e'$ incident to $u$ such that $e,e'$ are parallel.

Note that this definition is $G$-invariant.
So if $e,u$ is compatible or disparate, then for each $g\in G$, $e\cdot g, u\cdot g$ also satisfies the same.
This follows from the fact that the definition of parallel and disparate $1$-clusters is $G$-invariant.
\end{defn}

The following elementary observation follows immediately from the definitions and Lemma \ref{disp2}.

\begin{lem}\label{disp3}
Let $e$ be a $1$-cluster and $u$ be a $0$-cell.
Then the following holds:
\begin{enumerate}
\item Let $e,u$ be disparate (or compatible) and let $e,e'$ be parallel $1$-clusters. Then $e',u$ are disparate (respectively, compatible).
\item Let $e,u$ be disparate (or compatible) and let $C$ be a pre-cluster obtained from performing an amplification of $e$.
Then for each $e'\in C$, $e',u$ are disparate (respectively, compatible).
\end{enumerate}
\end{lem}

\begin{defn}\label{weakbalance}
Let $C_1,...,C_n$ be pre-clusters with bases $u_1,...,u_n$ respectively. 
We say that this is a \emph{weakly balanced system} if for any
$$e\in \bigcup_{1\leq i\leq n}C_i\qquad u_j\in \{u_1,...,u_n\}$$ 
the pair $e,u_j$ is either disparate or compatible.
\end{defn}

\begin{lem}\label{prebalancedexp}
Let $C_1,...,C_n$ be a weakly balanced system based at $u_1,...,u_n$.
Let $$C_1,...,C_n\to C_1',...,C_n'$$
be any amplification.
Then the system $C_1',...,C_n'$ is also weakly balanced.
\end{lem}

\begin{proof}
Consider a pair $e,u_j$ where $1\leq j\leq n$ and $e\in C_i'$ for some $1\leq i\leq n$.
Then $e$ is an offspring of a $1$-cluster $e'\in C_i$. 
Note that $e',u_j$ are compatible or disparate by our hypothesis.
Since both properties are inherited for the offspring $e$, thanks to Lemma \ref{disp3}, we are done.
\end{proof}

Our goal in this subsection is to prove the following Proposition.

\begin{prop}\label{prebalancing}
(Weakly balancing procedure) Let $C_1,...,C_n$ be clusters.
There is an amplification $$C_1,...,C_n\to C_1',...,C_n'$$
such that $C_1',...,C_n'$ is a weakly balanced system.
\end{prop}

The technical core of the proof of this proposition is the following lemma.

\begin{lem}\label{dispfinal}
Let $e$ be a $1$-cluster incident to a $0$-cell $u$ and let $u_1,...,u_n$ be $0$-cells.
Then there is an amplification of $e$ that produces a pre-cluster $C$ such that the following holds. 
For each $e'\in C$ and each $1\leq i\leq n$, the pair $e',u_i$
is disparate or compatible.
\end{lem}


Using this we can finish the proof of Proposition \ref{prebalancing}.

{\bf Proof of Proposition \ref{prebalancing}.}

\begin{proof}
For each $1$-cluster $e\in \bigcup_{1\leq i\leq n}C_i$,
using Lemma \ref{dispfinal} we find an amplification $e\to C_e$ (at the same base)
such that for each $u_i$ and each $e'\in C_e$, the pair $e',u_i$ is compatible or disparate.
(Note that this could possibly be the trivial amplification, i.e. it could be that $C_e=e$.)
Our amplification $C_i\to C_i'$ is then the amplification obtained by replacing each $e\in C_i$ with $C_e$,
i.e. $C_i'=\bigcup_{e\in C_i}C_e$.
\end{proof}

Before we provide proofs of the above, we study the notions of disparate or compatible pairs $e,u$ in depth. 
Given suitable descriptions of $e$, we would like to detect whether $e,u$ are disparate, or whether there is a $1$-cluster $e'$
incident to $u$ such that $e,e'$ are parallel.
To describe such situations we require certain restrictions on descriptions of $e$.
Recall the content of Lemma \ref{taileq}, since it shall be used in what follows.

\begin{defn}
Let $F\tau, F\nu\tau$ be the description of a $1$-cluster.
We say that such a description is \emph{tame}, if $\tau$ is a $Y$-word in standard form with no potential cancellations and $\nu\tau$ is in standard form.

\end{defn}

\begin{remark}
It is straightforward to produce tame descriptions for any given $1$-cluster.
For instance, choose a normal form representative for the base coset, and choose the parameter to have sufficiently large depth (by performing expansion moves).
Consider such a tame description $F\tau, F\nu\tau$. 
We know that $\tau$ is a standard form with no potential cancellations, however $\nu\tau$ is a standard form that may admit potential cancellations.
Any such potential cancellation must occur between a pair of occurrences of $y^{\pm}$, one of which lies in $\nu$ and the other in $\tau$.
\end{remark}

Now we establish criteria to detect whether a pair $e,u$ is compatible or disparate.
Note that considering the group action it suffices to consider the case when $u$ is the trivial coset.
The first one establishes when such a pair is compatible.

\begin{lem}\label{VEpar}
Let $u$ be the trivial coset $F$. 
Suppose that a $1$-cluster $e$ admits a tame description $F\tau, F\nu \tau$ that satisfies either one of the following:
\begin{enumerate}
\item $\nu,\tau$ have disjoint supports.
\item The restriction of $\nu \tau$ on $Supp(\nu)$ equals the restriction of an element of $F$ on $Supp(\nu)$.
\end{enumerate}
Then there is a $1$-cluster $e'$ incident to $u$ such that $e,e'$ are parallel.
\end{lem}

\begin{proof}
Assume that $(1)$ holds.
Let $$\tau^{-1}=y_{s_1}^{t_1}...y_{s_n}^{t_n}$$ where equality denotes equality as words.
We know that $$Supp(\tau^{-1})\cap Supp(\nu)=\emptyset$$
Since the given description of $e$ is tame, $\tau$ is a standard form with no potential cancellations.
So $$Supp(y_{s_i}^{t_i})\subset Supp(\tau)$$
for each $1\leq i\leq n$.
In particular, $$Supp(y_{s_i}^{t_i})\cap Supp(\nu)=\emptyset$$
for each $1\leq i\leq n$. 

It follows that the $1$-clusters $$F,F\nu\qquad F\tau,F\nu\tau$$
are parallel with the sequence of special forms $y_{s_1}^{t_1},...,y_{s_n}^{t_n}$ witnessing the conditions of Definition \ref{parequiv}. 

Now we consider the case when $(2)$ holds.
We shall produce another tame description of $e$ which will land us in case $(1)$.
We first convert $\nu\tau$ into a normal form $f\psi$ where $f\in F$ and $\psi$ is a $Y$-word.
Recall from Lemma \ref{taileq} that $f\psi$ is densely mixing on $Supp(\psi)\cdot f^{-1}$
and preserves tail equivalence on the complement of this set.

Since $\nu\tau=f\psi$ and since $\nu\tau$ preserves tail equivalence on $Supp(\nu)$,
it follows that $$Supp(\nu)\cap Supp(\psi)\cdot f^{-1}=\emptyset$$
This means that $$Supp(\nu)\cdot f \cap Supp(\psi)=\emptyset$$
Now consider the description of $e$
$$F(\nu\tau), F\nu^{-1}(\nu\tau)$$

Now let $\eta_1$ be a special form obtained by performing expansion moves on $\nu^{-1}$
such that for each percolating element $y_s^t$ of $\eta_1$, $f$ acts on $s$.
Let $\eta_2$ be the special form obtained by replacing each $y_s^t$ in $\eta_1$ by $y_{s\cdot f}^t$. 
Moreover, by ensuring that $\eta_1$ has sufficient depth we may guarantee that $\eta_2$ has sufficiently large depth so that $\eta_2\psi$ is a standard form.  

It follows that
$$F\nu^{-1}(f\psi)=F\eta_1(f\psi)=F\eta_2 \psi$$

So we get a new tame description of our $1$-cluster $e$
$$F\psi, F\eta_2\psi$$

Since $$Supp(\eta_2)=Supp(\eta_1)\cdot f=Supp(\nu^{-1})\cdot f=Supp(\nu)\cdot f$$
and $$Supp(\nu)\cdot f \cap Supp(\psi)=\emptyset$$ it follows that $$Supp(\eta_2)\cap Supp(\psi)=\emptyset$$
Using this new description we land in Case $(1)$ above and therefore 
conclude that there is a $1$-cluster $e'$ incident to $u$ such that $e,e'$ are parallel.

\end{proof}

The next criterion establishes when a pair $e,u$ is disparate.

\begin{lem}\label{VEdis}
Let $u$ be the trivial coset $F$. 
Suppose that a $1$-cluster $e$ admits a tame description $F\tau, F\nu \tau$ satisfying that:
\begin{enumerate}
\item $Supp(\nu)\subseteq Supp(\tau)$.
\item $\nu\tau$ restricted to $Supp(\nu)$ is densely mixing on $Supp(\nu)$.
\end{enumerate}
Then $e,u$ are disparate.
\end{lem}

\begin{proof}

Thanks to Lemma \ref{ampindbase} it suffices to consider amplifications of $e$ with a preferred choice of base, which will be $F\tau$.
Now assume by way of contradiction that there is an amplification $C$ of $e$ with base $F\tau$ given by
$$F\tau,F\nu_1\tau\qquad ... \qquad F\tau,F\nu_n\tau$$
(with $F\nu=F\nu_1...\nu_n$) so that for some $1\leq l\leq n$, the pair $$e'=F\tau,F\nu_l\tau\qquad u=F$$ is compatible.
We can also assume that $\nu_l\tau$ is a standard form (by replacing $\nu_l$ with an equivalent special form, if necessary).

In other words, there is a $1$-cell $e''$ incident to $F$ such that $e',e''$ are parallel.
By Lemma \ref{indparpar}, there is a description $$e''=F\eta,F\nu_l \eta$$ for a $Y$-word $\eta$, and a sequence of special forms $\psi_1,...,\psi_m$ such that:
\begin{enumerate}
\item $\psi_i,\nu_l$ have disjoint supports for each $1\leq i\leq m$.
\item $(\psi_1...\psi_m)\tau=\eta$.
\end{enumerate}

Now since the trivial coset $F$ is incident to $e''$, we have two possible cases:

 {\bf Case (1)} $F\eta=F$.
 
 {\bf Case (2)} $F\nu_l\eta=F$.

{\bf Proof for Case (1)}: 
It follows that $$F(\psi_1...\psi_m)\tau=F\eta=F\qquad (\psi_1...\psi_m)\tau\in F$$
Since $\tau$ is a $Y$-standard form with no potential cancellations, by Lemma \ref{taileq} it is densely mixing does on its support.
Since $Supp(\nu)\subseteq Supp(\tau)$, this implies that the action of $\tau$ on $Supp(\nu)$ is densely mixing.
In particular, the action of $\tau$ on $Supp(\nu_l)\subset Supp(\nu)$ is densely mixing.

Since $\psi_i,\nu_l$ have disjoint supports, the action of $(\psi_1...\psi_m)\tau$ restricted to $Supp(\nu_l)$ equals that of
$\tau$, and hence is densely mixing on $Supp(\nu_l)$.
So this cannot act as an element of $F$, and we get a contradiction.

{\bf Proof for Case (2)}: 
It follows that $$\nu_l\eta=\nu_l(\psi_1...\psi_m)\tau=(\psi_1...\psi_m)\nu_l\tau$$
This means that $$F(\psi_1...\psi_m)(\nu_l \tau)=F\nu_l\eta=F$$ and so $(\psi_1...\psi_m)(\nu_l \tau)\in F$.
Since $\psi_i,\nu_l$ have disjoint supports, the action of $(\psi_1...\psi_m)(\nu_l \tau)$ restricted to $Supp(\nu_l)$ equals the action of $\nu_l \tau$ on $Supp(\nu_l)$.
If we show that the action of $\nu_l \tau$ does not preserve tail equivalence on some sequence in $Supp(\nu_l)$ we shall obtain the desired contradiction.

Recall from the above that $\nu\tau$ is densely mixing on $Supp(\nu)$ and hence also on $Supp(\nu_l)$.
Since $(\nu_1...\nu_n)$ and $\nu$ are equivalent special forms, this means that $(\nu_1...\nu_n)\tau$ is densely mixing on $Supp(\nu_l)$.

Now note that the actions of $(\nu_1...\nu_n)\tau$ and $\nu_l\tau$ agree on $Supp(\nu_l)$ since $\nu_1,...,\nu_n$ are independent and have disjoint supports.
Hence $\nu_l\tau$ is densely mixing on $Supp(\nu_l)$.
\end{proof}

{\bf Proof of Lemma \ref{dispfinal}}

\begin{proof}
We proceed by induction on $n$.
We first consider the base case $n=1$.
Since the notion of parallel $1$-clusters is invariant under the $G$ action, we can assume that $u$ is the trivial coset $F$ for the sake of this proof.
We fix a tame description $e=F\tau,F\nu \tau$.

We perform a sequence of expansion moves on percolating elements of $\nu$ to obtain an equivalent special form $y_{s_1}^{t_1}...y_{s_m}^{t_m}$
such that for each $1\leq i\leq m$, either $Supp(y_{s_i}^{t_i})\subset Supp(\tau)$ or $Supp(y_{s_i}^{t_i})\cap Supp(\tau)=\emptyset$.
We use this data for an amplification of $e$ to obtain the pre-cluster: 
$$e_1=F\tau, Fy_{s_1}^{t_1}\tau\qquad ... \qquad e_m=F\tau, Fy_{s_m}^{t_m}\tau$$

If $Supp(y_{s_i}^{t_i})\cap Supp(\tau)=\emptyset$, then by Lemma \ref{VEpar} $e_i,u$ are compatible.

We claim that for any given $1\leq i\leq m$, if $Supp(y_{s_i}^{t_i})\subset Supp(\tau)$, then exactly one of the following holds:
\begin{enumerate}
\item  $y_{s_i}^{t_i}\tau$ is densely mixing $Supp(y_{s_i}^{t_i})$.
\item $y_{s_i}^{t_i}\tau$ acts like an element of $F$ on $Supp(y_{s_i}^{t_i})$.
\end{enumerate}

By applications of Lemmas \ref{VEpar} and \ref{VEdis}, we conclude that
in case $(1)$ the pair $e_i,u$ is disparate and in case $(2)$ the pair is compatible. 

We show that if $(2)$ does not hold then $(1)$ holds.
Since the description of $e$ is tame, we know that $\tau$ is a standard form with no potential cancellations, and that $\nu\tau$ is a standard form.
Furthermore, since $y_{s_i}^{t_i}$ is an offspring of a percolating element of $\nu$, $y_{s_i}^{t_i}\tau$ is a standard form.

Since $y_{s_i}^{t_i}\tau$ does not act like $F$ on $Supp(y_{s_i}^{t_i})$, it does not preserve tail equivalence on some sequence of the form $s_i\sigma$ for $\sigma\in 2^{\mathbf{N}}$.

{\bf Claim}: The calculation of $y_{s_i}^{t_i}\tau$ on the sequence $s_i\omega$, for any $\omega\in 2^{\mathbf{N}}$, has a nonzero exponent.

{\bf Proof of claim}: Since $y_{s_i}^{t_i}\tau$ is a standard form, the calculation of $y_{s_i}^{t_i}\tau$ on $s_i\sigma$ is of the form

$$(1)\qquad  p_1y^{q_1}p_2y^{q_2}...p_ly^{q_l}\sigma\qquad \text{ for some }p_1,...,p_l\in 2^{< \mathbf{N}}, q_1,...,q_l\in \mathbf{Z}\setminus \{0\}$$

Since this calculation does not preserve tail equivalence, we know that the exponent of this calculation is nonzero.
Moreover, the calculation of $y_{s_i}^{t_i}\tau$ on $s_i\omega$ is of the form 
$$(2)\qquad  p_1y^{q_1}p_2y^{q_2}...p_ly^{q_l}\omega$$

Note that the calculations $(1),(2)$ have a common prefix.
We claim that the exponent of calculation $(2)$ is the same as the exponent of the calculation in $(1)$.
Thanks to an application of Lemma \ref{potcanlemma}, any potential cancellations that exist in $(1)$
must indeed exist in the prefix $$(3)\qquad  p_1y^{q_1}p_2y^{q_2}...p_ly^{q_l}$$
To see this, note that if there were no potential cancellations within the prefix, then by Lemma \ref{potcanlemma},
there would be no new potential cancellations resulting from advancing occurrences of $y^{\pm 1}$ beyond the prefix, by applying moves
in calculation $(1)$.

So we perform all such cancellations within $(3)$.
Once all the cancellations have been performed,
some occurrences of $y^{\pm}$ must remain in the resulting string since the same prefix exists in the calculation obtained by applying the same moves to calculation $(1)$, which we know has nonzero exponent. 
The result is a finite calculation string with no potential cancellations and nonzero exponent.
Thanks to another application of Lemma \ref{potcanlemma}, upon adding the suffix $\omega$, we obtain a calculation without potential cancellations and the same nonzero exponent.
This proves our claim.

Therefore, from an elementary application of Lemma \ref{ulemma}, it follows that $y_{s_i}^{t_i}\tau$ is densely mixing on $Supp(y_{s_i}^{t_i})=[s_i0^{\infty}, s_i1^{\infty}]$.
This establishes the base case of the induction.

Now assume by the inductive hypothesis that we have a pre-cluster $C$ from an amplification $e\to C$ such that
the statement holds for $C$ and $u_1,...,u_{n-1}$.
Now apply the base case to each pair $e\in C, u_n$
to obtain amplifications $e\to C_e$ such that for each $e'\in C_e$ the pair $e',u_n$ is disparate or compatible. 
Note that it is possible that $e\to C_e$ is the trivial amplification and $e=C_e$.

Now consider the amplification $C\to C'$ which is given by replacing each $e\in C$ by $C_e$.
In particular, $C'=\bigcup_{e\in C}C_e$. 
The pre-cluster $C'$ satisfies the desired property, thanks to part $(2)$ of Lemma \ref{disp3}.
Since iterated amplifications of a $1$-cluster can be viewed as a single amplification of the $1$-cluster, we are done.
\end{proof}

\subsection{Balanced systems of clusters}

\begin{defn}\label{BSys}
We say that a family of clusters $C_1,...,C_n$ based at $u_1,...,u_n$ is a \emph{balanced system} if:
\begin{enumerate}
\item For each pair $e\in \bigcup_{1\leq i\leq n} C_i$ and $u_j$, it holds that $e,u_j$ are disparate or compatible.
\item Any pair of $1$-clusters in $\bigcup_{1\leq i\leq n}C_i$ is either parallel or disparate.
\end{enumerate} 
\end{defn}

In other words, a balanced system is a weakly balanced system that also satisfies that any pair of $1$-clusters in the system is either parallel or disparate.
In this subsection we will prove the following.

\begin{prop}\label{mainpropasp}
(Balancing procedure) Let $C_1,...,C_n$ be a weakly balanced system of clusters in $X$.
We can perform an amplification $$C_1,...,C_n\to C_1',...,C_n'$$
to obtain a balanced system $C_1',...,C_n'$.
\end{prop}

Balanced systems are not as well behaved as weakly balanced systems.
Performing ad hoc amplifications of a balanced system will in general not produce a balanced system. 
In particular, the analog of Lemma \ref{prebalancedexp} fails.
However if we perform an amplification of a balanced system $C_1,...,C_n$ carefully enough, this produces another balanced system.

\begin{defn}\label{comamppre}
(Common amplification for a balanced system)
Let $C_1,...,C_n$ be a balanced system, and let $E_j$ be a set of $1$-clusters in $C_j$ for some $1\leq j\leq n$.
Fix an amplification $$C_j\to C_j'\qquad \text{ along }E_j$$

We define an amplification $$C_1,...,C_n\to C_1',...,C_n'$$ as follows.
First we define 
$$E_i=\{e\in C_i\mid \exists e'\in E_j\text{ such that }e,e' \text{ are parallel } \}\qquad \text{ for } i\neq j$$
Then $C_i\to C_i'$ is defined as follows.
For each pair $e\in E_i,e'\in E_j$ of parallel $1$-clusters, the restriction of $C_i\to C_i'$ to $e$ is the common amplification of the restriction of $C_j\to C_j'$ to $e'$.
The restriction to every other $1$-cluster is the trivial amplification.



The resulting amplification is said to be a \emph{common amplification for a balanced system}.
\end{defn}

\begin{lem}\label{balancedsystem}
Let $$C_1,...,C_n\to C_1',...,C_n'$$ be a common amplification of a balanced system $C_1,...,C_n$.
Then $C_1',...,C_n'$ is a balanced system.
\end{lem}

\begin{proof}
First observe that by Lemma \ref{prebalancedexp}, $C_1',...,C_n'$ is weakly balanced and hence condition $(1)$ of Definition \ref{BSys} holds.
We will demonstrate that condition $(2)$ of Definition \ref{BSys} holds.

Consider $1$-clusters $$e_i'\in C_i',e_k'\in C_k'$$ for some $1\leq i,k\leq n$.
We would like to show that they are parallel or disparate.
Let $$e_i\in C_i, e_k\in C_k$$ be $1$-clusters such that $e_i',e_k'$ are offsprings of $e_i,e_k$ in the amplifications $$C_i\to C_i', C_k\to C_k'$$ respectively.
Note that $e_i,e_k$ are parallel or disparate, since $C_1,...,C_n$ is balanced.


If $e_i,e_k$ are disparate, then so are $e_i',e_k'$ thanks to Lemma \ref{disp2}.
If $e_i,e_k$ are parallel, then following Definition \ref{comamppre}, in this case both $e_i,e_k$ are amplified as parallel $1$-clusters in the sense of Definition \ref{parallelexpansion}
(including the possibility of the trivial amplification).
So the union of their sets of offsprings must satisfy that each pair is parallel or disparate.
In particular, $e_i',e_k'$ are either parallel or disparate.
\end{proof}

We now prove the \emph{balancing lemma} which is an important technical step in the proof of Proposition \ref{mainpropasp}.
This is where the Strong Expansion Lemma is used.

\begin{lem}\label{balancelem}
(Balancing Lemma) Let $C_1,C_2$ be pre-clusters with a common base $u$.
Then we can perform amplifications $C_1\to D_1$ and $C_2\to D_2$ such that $D_1,D_2$ is a balanced system.
\end{lem}

\begin{proof}
Note that since both clusters share the same base $u$, the system is weakly balanced.
In this proof we shall use the word \emph{offspring} for percolating elements obtained from expansion moves on special forms,
as well as elements of pre-clusters obtained from amplifying a pre-cluster.
The usage will be clear from the context. 

For simplicity of notation we can assume that  $u$ is the trivial coset, since the proof works in exactly the same way in the general case.
Also, by performing an amplification if necessary, we assume that $C_1$ is a cluster of the form
$$Fy_{s_1}^{t_1}, F\qquad Fy_{s_2}^{t_2},F\qquad ...\qquad Fy_{s_n}^{t_n},F$$
where $t_i\in \{1,-1\}$ and $s_1,...,s_n$ are independent. And $C_2$ is a cluster of the form
$$Fy_{u_1}^{v_1}, F\qquad Fy_{u_2}^{v_2},F\qquad ...\qquad Fy_{u_m}^{v_m},F$$
where $v_i\in \{1,-1\}$ and $u_1,...,u_m$ are independent.

By the Strong Expansion Lemma \ref{strongexpansionlemma}, there are words $y_{p_1}^{q_1}...y_{p_l}^{q_l}$ and $y_{r_1}^{w_1}...y_{r_k}^{w_k}$ such that:
\begin{enumerate}
\item $y_{p_1}^{q_1}...y_{p_l}^{q_l}$ is obtained from $y_{s_1}^{t_1}...y_{s_n}^{t_n}$ by applying expansion moves as in \ref{movesdefinition}.
\item $y_{r_1}^{w_1}...y_{r_k}^{w_k}$ is obtained from $y_{u_1}^{v_1}...y_{u_m}^{v_m}$ by applying expansion moves as in \ref{movesdefinition}.
\item For each pair $y_{p_i}^{q_i}, y_{r_j}^{w_j}$ it holds that either $y_{p_i}^{q_i}= y_{r_j}^{w_j}$ or the pair $y_{p_i}^{q_i}, y_{r_j}^{w_j}$ contains no common offsprings.
\end{enumerate}

Let $D_1$ be given by the amplification of $C_1$ consisting of the $1$-clusters $$F y_{p_1}^{q_1},F\qquad ... \qquad Fy_{p_l}^{q_l}, F$$
and $D_2$ be the amplification of $C_2$ consisting of the $1$-clusters
$$F y_{r_1}^{w_1} ,F\qquad ... \qquad Fy_{r_k}^{w_k},F$$
It remains to check that the $1$-clusters in each pair $$e_1=Fy_{r_i}^{w_i},F\qquad e_2=Fy_{r_j}^{w_j},F$$
are either equal (hence parallel) or disparate.

If $1$-clusters in such a pair are neither parallel nor disparate, we claim that $y_{r_i}^{w_i}, y_{r_j}^{w_j}$ have a common offspring, 
contradicting condition $(3)$ above.
To see this, observe that there are amplifications $$e_1\to E_1\qquad e_2\to E_2$$
using the descriptions above that have parallel offsprings.

Let the two descriptions of the parallel offsprings obtained in the amplifications be $$F\nu_1,F\in E_1 \qquad F\nu_2,F\in E_2$$
Note that we use Lemma \ref{samenodeparallel} to see that these two $1$-clusters are equal.
Since $\nu_1,\nu_2$ are equivalent, we can perform a sequence of expansions to obtain the same special form. 
But this means that $y_{r_i}^{w_i}, y_{r_j}^{w_j}$ have a common offspring.
\end{proof}

{\bf Proof of Proposition \ref{mainpropasp}}
\begin{proof}
We fix a weakly balanced system $C_1,...,C_n$. 
Thanks to Lemma \ref{prebalancedexp}, performing any amplifications on a weakly balanced system produces a weakly balanced system.
We shall omit the usage of the phrase ``weakly balanced" in much of the remainder of the proof.
It shall be evident that all systems under consideration are weakly balanced.

Our goal will now be to perform a sequence of amplifications on this system to obtain a balanced system.
Since iterates of amplifications can be expressed as a single amplification, the proposition will follow.
The proof is an induction performed on the number $n$ of pre-clusters in the system $C_1,...,C_n$.
Note that the base case is trivial, since $1$-clusters in a pre-cluster are pairwise disparate, and hence a system with only a single pre-cluster is balanced.

{\bf Inductive hypothesis}
A system of weakly balanced pre-clusters $C_1,...,C_n$ such that $C_1,...,C_{n-1}$ is balanced.

{\bf Inductive Step}
The inductive step will be performed in $n-1$ stages.
Declare $$C_1^{(0)}=C_1\qquad C_2^{(0)}=C_2\qquad ...\qquad C_{n}^{(0)}=C_n$$
Stage $i$ involves performing an amplification
$$C_1^{(i-1)},...,C_{n}^{(i-1)}\to C_1^{(i)},...,C_{n}^{(i)}$$
{\bf Output at stage $i$}: A system 
$C_1^{(i)},...,C_{n}^{(i)}$ such that:
\begin{enumerate}
\item $C_1^{(i)},...,C_{n-1}^{(i)}$ is a balanced system.
\item  $C_1^{(i)},...,C_i^{(i)}, C_n^{(i)}$
is a balanced system.
\end{enumerate}
{\bf Output at stage $n-1$}: A balanced system
$C_1^{(n-1)},...,C_{n}^{(n-1)}$.

We now describe the amplifications at stage $i$.
We partition the set of $1$-clusters of $C_n^{(i-1)}$ into sets $E_1,E_2$ as follows.
Let $$E_1=\{e\in C_n^{(i-1)}\mid \exists e'\in C_i^{(i-1)} \text{ such that }e'\text{ is neither disparate nor parallel to } e\}\qquad E_2=C_n^{(i-1)}\setminus E_1$$

{\bf Claim 1}: For each $e\in E_1$, the pair $e,u_i$ is compatible.

If this were not the case, since our system is weakly balanced, it would follow that $e$ is disparate with each $1$-cluster of $C_i^{(i-1)}$.
This contradicts the definition of $E_1$.

{\bf Claim 2}: For each $e\in E_1$, $e$ is not parallel to any $1$-cluster in $C_i^{(i-1)}$.

Each pair of $1$-clusters in $C_i^{(i-1)}$ are disparate and if the above was false $e$ would be parallel or disparate with every $1$-cluster in $C_i^{(i-1)}$.
Again, this contradicts the definition of $E_1$.


{\bf Claim 3}: Each $1$-cluster in $E_1$ is disparate with each $1$-cluster in $\bigcup_{1\leq j\leq i-1}C_j^{(i-1)}$.

Assume that there is a pair 
$$e\in E_1\qquad e'\in \bigcup_{1\leq j\leq i-1}C_j^{(i-1)}$$
such that $e,e'$ are not disparate.
Using the inductive hypothesis, it follows that $e,e'$ are parallel.
Since there is an $e''\in C_{i}^{(i-1)}$ such that $e,e''$ are neither parallel nor disparate, it follows that $e',e''$ are neither parallel nor disparate,
contradicting the assumption that the system $C_1^{(i-1)},...,C_{n-1}^{(i-1)}$ is balanced. This proves the claim.

Now let $E_1=\{e_1,...,e_k\}$.
Following Claim $1$, for each $1\leq j\leq k$, let $e_j'$ be the $1$-cluster incident to $u_i$ such that $e_j,e_j'$ are parallel.
We denote the corresponding pre-cluster based at $u_i$ as $$E_1'=\{e_1',...,e_k'\}$$
Using the Balancing Lemma \ref{balancelem}, there is an amplification $$C_i^{(i-1)}\to C_i^{(i)}\qquad E_1'\to E_1''$$
such that $C_i^{(i)},E_1''$ is a balanced system.
Let $D\subseteq C_i^{(i-1)}$ be the set of $1$-clusters that is amplified above.

{\bf Claim 4}: For each pair
$$e\in D\qquad e'\in \bigcup_{1\leq l\leq i-1}C_{l}^{(i-1)}$$
$e,e'$ are disparate.

Since $C_1^{(i-1)},...,C_{n-1}^{(i-1)}$ is balanced, $e,e'$ are either parallel or disparate.
Assume by way of contradiction that $e,e'$ are parallel.
There is a $e_j'\in E_1'$ such that $e,e_j'$ are neither parallel nor disparate.
It follows that $e,e_j$ are neither parallel nor disparate, and the same for $e',e_j$.
This contradicts our inductive hypothesis which assumes that $C_1^{(i-1)},...,C_{i-1}^{(i-1)},C_n^{(i-1)}$ is balanced.

Next, we let $$E_1'\to E_1''\qquad C_n^{(i-1)}\to C_n^{(i)}$$ and
$$C_i^{(i-1)}\to C_i^{(i)}\qquad C_1^{(i-1)},...,C_{n-1}^{(i-1)}\to C_1^{(i)},...,C_{n-1}^{(i)}$$
be the respective common amplifications.

{\bf Claim 5}: $C_j^{(i-1)}=C_j^{(i)}$ for each $1\leq j\leq i-1$.

This is an immediate consequence of Claim $4$ and the definition of a common amplification.

{\bf Claim 6}: The system $C_1^{(i)},...,C_i^{(i)},C_n^{(i)}$ is balanced. 

We already know that the systems $$C_i^{(i)},C_n^{(i)}\qquad C_1^{(i)},...,C_{n-1}^{(i)}$$ are balanced.
So if there are a pair of $1$-clusters $e,e'$ in $$C_1^{(i)},...,C_i^{(i)},C_n^{(i)}$$ that are neither parallel nor disparate,
it must be that $$e\in \bigcup_{1\leq j\leq i-1}C_j^{(i)}\qquad e'\in C_n^{(i)}$$
From Claims $3,5$ we know that $e'$ is the result of an amplification of a $1$-cluster that is disparate
with each $1$-cluster in $\bigcup_{1\leq j\leq i-1}C_j^{(i)}$. Hence this cannot be the case.

This concludes the induction.
\end{proof}

\subsection{Balanced systems and nonpositive curvature}

In this subsection we show that given a balanced system $C_1,...,C_n$, there is 
a subcomplex $\mathcal{C}$ of $X$ such that the following holds:
\begin{enumerate}
\item $\bigcup_{1\leq i\leq n}\bar{C}_i$ is a subcomplex of $\mathcal{C}$.
\item $\mathcal{C}$ is homeomorphic to a nonpositively curved cube complex.
\end{enumerate}

The main idea behind this is the following.

\begin{defn}
(Parallel closure) Let $C_1,...,C_n$ be a balanced system of pre-clusters in $X$.
Let $E$ be the set of $1$-clusters in $X$ such that for each $e\in E$, there is an $e'\in \bigcup_{1\leq i\leq n}C_i$
such that $e$ and $e'$ are parallel.
Denote by $\mathcal{C}$ as the union of the following set of clusters in $X$:
$$\{C\text{ is a cluster in }X\mid \text{Each facial $1$-subcluster of }C\text{ is an element of }E\}$$
We call $\mathcal{C}$ the \emph{parallel closure} of $C_1,...,C_n$.
\end{defn}

\begin{prop}\label{npccc}
Let $\mathcal{C}$ be the parallel closure of a balanced system $C_1,...,C_n$.
Then the following holds:
\begin{enumerate}
\item $\bigcup_{1\leq i\leq n}\bar{C}_i\subset \mathcal{C}$. 
\item $\mathcal{C}$ is homeomorphic to a nonpositively curved cube complex.
\end{enumerate}
\end{prop}

\begin{proof}
Each facial $1$-subcluster of $\bar{C}_i$ is parallel to a $1$-cluster in $C_i$.
So the first claim follows immediately from the definition of $\mathcal{C}$.

We shall prove the second claim in two parts.
First we demonstrate that $\mathcal{C}$ is homeomorphic to a cube complex which is locally finite and finite dimensional.
Then we shall demonstrate that this is nonpositively curved in the sense of Gromov.
In the first part, we show that the intersection of each pair of clusters in $\mathcal{C}$ is a facial subcluster of both.

First observe that if $e_1$ and $e_2$ are facial $1$-subclusters of clusters $C_1$ and $C_2$, respectively, in $\mathcal{C}$,
then $e_1,e_2$ are parallel or disparate.
This follows immediately from the definition of $\mathcal{C}$.

Assume by way of contradiction that there are two clusters $D_1,D_2$ in $\mathcal{C}$
such that $D_1\cap D_2$ is a diagonal subcluster of $D_1$.

Now we find descriptions satisfying:
\begin{enumerate}
\item $D_1$ is described with base $F\tau$ and parameters $\psi_1,...,\psi_n$.
\item $D_2$ is described with base $F\tau$ and parameters $\eta_1,...,\eta_m$.
\item There is a facial $1$-subcluster $e$ of $D_1\cap D_2$, incident to $F\tau$, which is a diagonal subcluster of $D_1$.
\end{enumerate}

By $(3)$ above there are sets $$X\subseteq \{1,...,n\}, |X|>1\qquad Y\subseteq \{1,...,m\},|Y|>0$$ satisfying: 
$$e=F\tau,F\psi_X\tau=F\tau,F\eta_Y\tau $$

{\bf Claim}: Each pair of $1$-clusters
$$F\tau,F\psi_i\tau\qquad F\tau,F\eta_j\tau$$
is distinct for any $i\in X, j\in Y$.

By way of contradiction, assume that $$e'=F\tau,F\psi_i\tau= F\tau,F\eta_j\tau$$ for some $i\in X,j\in Y$.
This implies that $e'$ is a facial $1$-subcluster of $D_1\cap D_2$ since it is a facial $1$-subcluster of both $D_1$ and $D_2$.
This contradicts our hypothesis that $e$ is a facial $1$-subcluster of the intersection 
$D_1\cap D_2$ since $e,e'$ cannot be orthogonal.

Since  $F\psi_X=F\eta_Y$, and since both $\psi_X,\eta_Y$ are products of independent special forms, there is a $Y$-word $y_{s_1}^{t_1}...y_{s_n}^{t_n}$ which can be obtained from both $\eta_Y,\psi_X$ using expansion moves.
This means that there are $j\in X, k\in Y$ such that we can perform expansion moves on $\psi_{j},\eta_{k}$ respectively to obtain special forms $\psi_{j}',\eta_{k}'$
satisfying that some $y_{s_i}^{t_i}$ occurs in both $\psi_{j}',\eta_{k}'$.

It follows that we can perform an amplification of the $1$-clusters
$$e_1=F\tau,F\psi_{j}\tau\qquad e_2=F\tau,F\eta_{k}\tau$$
to obtain pre-clusters that both contain the $1$-cluster
$F\tau, Fy_{s_i}^{t_i}\tau$.
This means that $e_1,e_2$ are not disparate.
Note that from the claim above we know that $e_1,e_2$ are distinct and hence not parallel.

But since they are both facial $1$-clusters of clusters in $\mathcal{C}$, they must then be parallel or disparate.
This is a contradiction.
Therefore, our original assumption must be false, and it must be the case that $D_1\cap D_2$ is a facial subcluster of both $D_1$ and $D_2$.

It follows that $\mathcal{C}$ is naturally homeomorphic to a cube complex.
We shall abuse notation and also refer to $\mathcal{C}$ as a cube complex itself.
For the rest of the proof, each $n$-cluster in $\mathcal{C}$ is assumed to be replaced by a regular Euclidean $n$-cube,
facial subclusters are declared as subcubes, and the facial intersections provide the gluing maps.

We now observe that $\mathcal{C}$ is locally finite and finite dimensional.
Consider a $0$-cell $u$ in $\mathcal{C}$.
Let $E_u$ be the set of closed $1$-cells incident to $u$ in $C$.
Each closed $1$-cell in $E_u$, by definition, corresponds to a $1$-cluster in $X$ which is parallel to a $1$-cluster
in $\bigcup_{1\leq i\leq n} C_i$.
We abuse notation and also denote this set of $1$-clusters as $E_u$.

By Lemma \ref{samenodeparallel}, any two distinct $1$-clusters in $E_u$ lie in distinct ``parallel equivalence classes".
Since the set $\bigcup_{1\leq i\leq n} C_i$ is finite, the number of ``parallel equivalence classes" determined by them is finite,
and hence $E_u$ is finite.

This observation implies that the cube complex $\mathcal{C}$ is locally finite, and the largest dimension of a cube
in $\mathcal{C}$ is bounded above by $\sum_{1\leq i\leq n} |C_i|$.

Next, we demonstrate that $\mathcal{C}$ is nonpositively curved.
It is easy to see that the link of each $0$-cell in $\mathcal{C}$ is a simplicial complex.
We will show that the link of each $0$-cell in $\mathcal{C}$ does not contain empty simplices, and hence is a flag simplicial complex.
To see this, let $e_1,...,e_n$ be closed $1$-cells in $\mathcal{C}$ incident to a $0$-cell $u$
with the property that each pair $e_i,e_j$ comprises of $1$-faces of a square incident to $u$.

The closed $1$-cells $e_1,...,e_n$ correspond to $1$-clusters $e_1',...,e_n'$
in $X$ with the property that each pair $e_i',e_j'$ are facial $1$-subclusters of a $2$-cluster.
In particular, $e_1',...,e_n'$ are pairwise orthogonal, and hence by Lemma \ref{orthlem} there is an $n$-cluster $D$
in $X$ which contains $e_1',...,e_n'$ as facial $1$-subclusters.

Each $1$-cluster $e_1',...,e_n'$ is parallel to some $1$-cluster in $\bigcup_{1\leq i\leq n} C_i$.
And each facial $1$-subcluster of $D$ is parallel to some $e_i'$.
So it follows that $D$ corresponds to a cube $D'$ in $\mathcal{C}$, so that the $1$-faces of $D'$ incident to $u$ are precisely $e_1,...,e_n$.
Therefore, the link at the vertex $u$ in $\mathcal{C}$ contains an $n$-simplex whose vertices correspond to 
 $e_1,...,e_n$.
\end{proof}

Now we conclude the proof of asphericity.

\subsection{Proof of asphericity}

Any image of a sphere $S^n$ in $X$ under a continuous map is contained in some finite subcomplex $Y$.
Since $X$ is defined as a union of clusters, $Y$ itself is contained in a finite union of clusters.
In particular, we find a finite list of pre-clusters $C_1,...,C_n$ such that $Y\subset \bigcup_{1\leq i\leq n}\bar{C}_i$.
Using Propositions \ref{prebalancing} and \ref{mainpropasp}, we obtain a balanced system $C_1',...,C_n'$ such that 
$$\bigcup_{1\leq i\leq n} \bar{C}_i\subseteq \bigcup_{1\leq i\leq n}\bar{C}_i'$$
Using Proposition \ref{npccc}, we conclude that the parallel closure $\mathcal{C}$ of $C_1',...,C_n'$ satisfies that:
\begin{enumerate}
\item $Y\subseteq \bigcup_{1\leq i\leq n} \bar{C}_i\subseteq \bigcup_{1\leq i\leq n}\bar{C}_i'\subseteq \mathcal{C}$.
\item $\mathcal{C}$ is homeomorphic to a nonpositively curved cube complex, hence is aspherical.
\end{enumerate}
It follows that the image of the sphere is nullhomotopic in $\mathcal{C}$ and hence in $X$.
Therefore, $X$ is aspherical.

\section{A question}

The following question, suggested to the author by Mikhael Gromov, could provide an interesting direction for future research.

\begin{question}
Is there a group $G$ satisfying the following?
\begin{enumerate}
\item $G$ is of type $\mathbf{F}$.
\item $G$ is nonamenable.
\item $G$ does not contain nonabelian free subgroups.
\end{enumerate}
\end{question}

Recall that a group is of type $\mathbf{F}$ if it admits a finite Eilenberg-Mclane complex.
Examples of such groups include finitely generated free abelian groups, torsion free hyperbolic groups (including finitely generated free groups), torsion free subgroups of finite index in 
finitely generated Coxeter groups, torsion free subgroups of finite index in arithmetic groups, 
and torsion free subgroups of finite index in the outer automorphism group of a finitely generated free group.
(See \cite{Geoghegan} for further details.)
None of the above examples can be sources of such a group. 

The group studied in this article is not of type $\mathbf{F}$, since it contains subgroups isomorphic to $\mathbb{Z}^{\infty}$.
Interesting subgroups of the group of piecewise projective homeomorphisms of the real line seem to always contain $\mathbb{Z}^{\infty}$ subgroups,
hence may not provide such an example.
This includes the case of Thompson's group $F$, for instance.
Containing torsion elements is also an obstruction for type $\mathbf{F}$. 
Hence constructions that involve nonamenable Burnside groups, torsion Tarski monsters, Golod-Shafarevich groups as well
as the finitely presentable nonamenable group of Olshanskii-Sapir cannot provide a source of such a group.

\end{document}

%% file: macros.tex
\newtheorem{thm}{Theorem}[section]

\newtheorem{lem}[thm]{Lemma}
\newtheorem*{lem*}{Lemma}

\newtheorem{prop}[thm]{Proposition}
\newtheorem{proposition}[thm]{Proposition}

\newtheorem{cor}[thm]{Corollary}

\newtheorem{question}[thm]{Question}

\theoremstyle{remark}
\newtheorem{remark}[thm]{Remark}

\theoremstyle{definition}
\newtheorem{defn}[thm]{Definition}

\newtheorem{example}[thm]{Example}

\newcounter{my_enumerate_counter}

\newcommand\comment[1]{}

\newcommand\Tcal{\mathcal{T}}

\newcommand\Nbb{\mathbb{N}}

\newcommand{\arrow}[1]{\overrightarrow{#1}}

\newcommand{\seq}[1]{\mathtt{#1}}

\renewcommand{\epsilon}{\varepsilon}

%% file: crowded_cluster.tex
\begin{tikzpicture}

\coordinate (a) at (0,0);
\coordinate (b) at ($(a)+(4,0)$);
\coordinate (c) at ($(a)+(0,4)$);
\coordinate (d) at ($(a)+(4,4)$);
\coordinate (e) at ($(a)+(2,1)$);
\coordinate (f) at ($(b)+(2,1)$);
\coordinate (g) at ($(c)+(2,1)$);
\coordinate (h) at ($(d)+(2,1)$);

\filldraw[lightgray] (a) -- (b) -- (f) -- (h) -- (g) -- (c) -- (a);

\draw[line width=1] (a) -- (b) -- (d) -- (c) -- (a) -- (e) -- (f) -- (h) -- (g) -- (e)   (b) -- (f)   (d) -- (h)   (c) -- (g);
\draw[line width=1] (a) -- (g)   (a) -- (f)   (a) -- (h)   (b) -- (h)   (c) -- (h);

\filldraw (a) circle (2pt);
\filldraw (b) circle (2pt);
\filldraw (c) circle (2pt);
\filldraw (d) circle (2pt);
\filldraw (e) circle (2pt);
\filldraw (f) circle (2pt);
\filldraw (g) circle (2pt);
\filldraw (h) circle (2pt);

\node at ($(a)+(-.3,0)$) {$F$};
\node at ($(b)+(0,-.5)$) {$Fy_{s11}$};
\node at ($(c)+(0,.5)$) {$Fy_{s0}$};
\node at ($(d)+(-1,-.4)$) {$Fy_{s0}y_{s11}$};
\node at ($(e)+(.7,.4)$) {$Fy_{s10}^{-1}$};
\node at ($(f)+(1,0)$) {$Fy_{s10}^{-1}y_{s11}$};
\node at ($(g)+(0,.5)$) {$Fy_{s0}y_{s10}^{-1}$};
\node at ($(h)+(0,.5)$) {$Fy_{s0}y_{s10}^{-1}y_{s11}=Fy_s$};

\end{tikzpicture}

%% file: medium_cluster.tex
\begin{tikzpicture}

\coordinate (a) at (0,0);
\coordinate (b) at ($(a)+(4,0)$);
\coordinate (c) at ($(a)+(0,4)$);
\coordinate (d) at ($(a)+(4,4)$);
\coordinate (e) at ($(a)+(2,1)$);
\coordinate (f) at ($(b)+(2,1)$);
\coordinate (g) at ($(c)+(2,1)$);
\coordinate (h) at ($(d)+(2,1)$);

\filldraw[lightgray] (a) -- (b) -- (f) -- (h) -- (g) -- (c) -- (a);

\draw[line width=1] (a) -- (b) -- (d) -- (c) -- (a) -- (e) -- (f) -- (h) -- (g) -- (e)   (b) -- (f)   (d) -- (h)   (c) -- (g);
\draw[line width=1] (a) -- (g)   (b) -- (h);

\filldraw (a) circle (2pt);
\filldraw (b) circle (2pt);
\filldraw (c) circle (2pt);
\filldraw (d) circle (2pt);
\filldraw (e) circle (2pt);
\filldraw (f) circle (2pt);
\filldraw (g) circle (2pt);
\filldraw (h) circle (2pt);

\node at ($(a)+(-.3,0)$) {$F$};
\node at ($(b)+(0,-.5)$) {$Fy_{s111}$};
\node at ($(c)+(0,.5)$) {$Fy_{s0}$};
\node at ($(d)+(-1,-.4)$) {$Fy_{s0}y_{s111}$};
\node at ($(e)+(.7,.4)$) {$Fy_{s10}^{-1}$};
\node at ($(f)+(1,0)$) {$Fy_{s10}^{-1}y_{s111}$};
\node at ($(g)+(0,.5)$) {$Fy_{s0}y_{s10}^{-1}$};
\node at ($(h)+(0,.5)$) {$Fy_{s0}y_{s10}^{-1}y_{s111}$};

\end{tikzpicture}